\documentclass[11pt,reqno,twoside]{article}

\usepackage{fixltx2e} % Makes \( \) equation style robust, among other
\usepackage{cmap} % Load before fontenc 

\usepackage[T1]{fontenc}
\usepackage[utf8]{inputenc}
\usepackage{graphicx}
\usepackage{placeins}
\usepackage{enumerate}
\usepackage{amssymb}
\usepackage{algcompatible}
\usepackage{pifont}

\usepackage{subcaption}
\usepackage{verbatim}

\usepackage{mathtools}
\mathtoolsset{showonlyrefs=true}

\usepackage{setspace}

\let\counterwithin\relax  %DSA: i had to include this to be able to compile
\usepackage{lmodern} % Improved version of computer modern
\usepackage[scale=0.88]{tgheros} % Helvetica clone for sans serif font

\usepackage{bm} % boldmath must be called after the package

\usepackage{amsmath,amsbsy,amsgen,amscd,amsthm,amsfonts,amssymb} 

\usepackage[centering,top=1.0in,bottom=1.1in,left=1in,right=1in]{geometry}

\usepackage{titling}
\usepackage{musicography}
\graphicspath{ {./images/} }

\usepackage[sf,bf,compact]{titlesec}

\usepackage{booktabs,longtable,tabu} % Nice tables
\usepackage[font=small,margin=12pt,labelfont={sf,bf},labelsep={space}]{caption}

\usepackage[usenames,dvipsnames]{xcolor}

\definecolor{dark-gray}{gray}{0.3}
\definecolor{dkgray}{rgb}{.4,.4,.4}
\definecolor{dkblue}{rgb}{0,0,.5}
\definecolor{medblue}{rgb}{0,0,.75}
\definecolor{rust}{rgb}{0.5,0.1,0.1}

\usepackage{url}
\usepackage[colorlinks=true]{hyperref}
\hypersetup{linkcolor=dkblue}    
\hypersetup{citecolor=rust}      
\hypersetup{urlcolor=rust}     

\usepackage[final]{microtype} 

\newtheoremstyle{myThm} % name
    {\topsep}                    % Space above
    {\topsep}                    % Space below
    {\itshape}                   % Body font
    {}                           % Indent amount
    {\sffamily\bfseries}                   % Theorem head font
    {.}                          % Punctuation after theorem head
    {.5em}                       % Space after theorem head
    {}  % Theorem head spec (can be left empty, meaning ‘normal’)

\newtheoremstyle{myRem} % name
    {\topsep}                    % Space above
    {\topsep}                    % Space below
    {}                   % Body font
    {}                           % Indent amount
    {\sffamily}                   % Theorem head font
    {.}                          % Punctuation after theorem head
    {.5em}                       % Space after theorem head
    {}  % Theorem head spec (can be left empty, meaning ‘normal’)

\newtheoremstyle{myDef} % name
    {\topsep}                    % Space above
    {\topsep}                    % Space below
    {}                   % Body font
    {}                           % Indent amount
    {\sffamily\bfseries}                   % Theorem head font
    {.}                          % Punctuation after theorem head
    {.5em}                       % Space after theorem head
    {}  % Theorem head spec (can be left empty, meaning ‘normal’)

\theoremstyle{myThm}
\newtheorem{theorem}{Theorem}[section]
\newtheorem{lemma}[theorem]{Lemma}
\newtheorem{proposition}[theorem]{Proposition}

\theoremstyle{myRem}
\newtheorem{remark}[theorem]{Remark}

\theoremstyle{myDef}

 \newtheorem{example}[theorem]{Example}

\usepackage{fancyhdr}
\usepackage{algorithm}
\usepackage{algpseudocode}

\let\originalleft\left
\let\originalright\right
\renewcommand{\left}{\mathopen{}\mathclose\bgroup\originalleft}
\renewcommand{\right}{\aftergroup\egroup\originalright}

\usepackage{mathtools}
\mathtoolsset{centercolon}  % Makes := typeset correctly for definitions

\renewcommand{\phi}{\varphi}

\providecommand{\mathbbm}{\mathbb} % In case we don't load bbm

\newcommand{\R}{\mathbbm{R}}

\newcommand{\OO}{\mathcal{O}}

\definecolor{mygreen}{rgb}{0.1,0.75,0.2}

\newcommand{\nc}{\normalcolor}

\newcommand{\Expect}{\operatorname{\mathbb{E}}}

\newcommand{\Nc}{\mathcal{N}}

\usepackage[]{algorithm}

\usepackage{graphicx}

\usepackage{soul}
\usepackage{ulem}
\newcommand{\stkout}[1]{\ifmmode\text{\sout{\ensuremath{#1}}}\else\sout{#1}\fi}
\usepackage{authblk}
\usepackage[square,numbers]{natbib}
\usepackage{chngcntr}
\usepackage{mathrsfs}
\counterwithin{table}{section}
\counterwithin{algorithm}{section}
\title{Iterative Ensemble Kalman Methods: \\
 A Unified Perspective with Some New Variants} % CORRECT TITLE
\author{Neil K. Chada$^\dagger$, Yuming Chen$^{*}$ and Daniel Sanz-Alonso$^{*}$}

\vspace{.25in} 

\date{}

\makeatletter\@addtoreset{section}{part}\makeatother%
\numberwithin{equation}{section}

\newcommand{\upperRomannumeral}[1]{\uppercase\expandafter{\romannumeral#1}}

\newcommand{\J}{{\mathsf{J}}}

\newcommand{\Jtp}{\J_{\mbox {\tiny{\rm TP}}}}
\newcommand{\Jdm}{\J_{\mbox {\tiny{\rm DM}}}}

\newcommand{\Jl}{\J_i^{\mbox {\tiny{$\ell$}}}}
\newcommand{\Juc}{\J_i^{\mbox {\mbox{\tiny{UC} }  }}}
\newcommand{\Jtpl}{\J_{\mbox{\tiny{\rm TP,}}i}^{\ell} }

\newcommand{\Juctp} {\J_{\mbox {\tiny{\rm TP}},i}^{\mbox {\mbox{\tiny{UC} }  }}}
\newcommand{\Jucdm} {\J_{\mbox {\tiny{\rm DM}},i}^{\mbox {\mbox{\tiny{UC} }  }}}
\newcommand{\Jucdmn} {\J_{\mbox {\tiny{\rm DM}},i}^{(n),\mbox {\mbox{\tiny{UC} }  }}}
\newcommand{\Jtpn}{\J_{\mbox {\tiny{\rm TP}},i}^{(n)}}
\newcommand{\Juctpn} {\J_{\mbox {\tiny{\rm TP}},i}^{(n),\mbox {\mbox{\tiny{UC} }  }}}

\newcommand{\rtp}{r_{\mbox {\tiny{\rm TP}}}}
\newcommand{\rdm}{r_{\mbox {\tiny{\rm DM}}}}

\newcommand{\uin}{u_i^{(n)}}
\newcommand{\yin}{y_i^{(n)}}
\newcommand{\vin}{v_i^{(n)}}
\newcommand{\xiin}{\xi_i^{(n)}}
\newcommand{\zin}{z_i^{(n)}}

\newcommand{\Cf}{\mathfrak{C}}
\newcommand{\mf}{\mathfrak{m}}

\newcommand*\diff{\mathop{}\!\mathrm{d}}

\begin{document}
\maketitle %  LEAVE HERE
% The command above causes the title to be displayed.

%>>>>> DELETE ALL CONTENT UNTIL "\end{document}"
% This is the body of your document.
\begin{abstract}
This paper provides a unified perspective of iterative ensemble Kalman methods, a family of derivative-free algorithms for parameter reconstruction and other related tasks. We identify, compare and develop three subfamilies of ensemble methods that differ in the objective they seek to minimize and the derivative-based optimization scheme they approximate through the ensemble.  Our work  emphasizes two principles for the derivation and analysis of iterative ensemble Kalman methods: statistical linearization and continuum limits. Following these guiding principles, we introduce new iterative ensemble Kalman methods that show promising  numerical performance in Bayesian inverse problems, data assimilation and machine learning tasks. 
\let\thefootnote\relax\footnote{$^{\dagger}$ Applied Mathematics and Computational Science, King Abdullah University of Science and Technology.}
\let\thefootnote\relax\footnote{$^*$ Department of Statistics, University of Chicago.}
%Ensemble Kalman inversion (EKI) has emerged as a promising numerical methodology for parameter estimation. Conceptually it is based on the application of the ensemble Kalman filter, which is a popular data assimilation method, to inverse problems. The attraction of EKI is that it is a derivative free optimizer which poses benefits for high-dimensional systems, where it can be viewed as a black box solver. Since the paper of Iglesias et al. \cite{ILS13} there has been an extensive amount of research in the field, which is related to both developing and deriving new theory, methodology and applications. The purpose of this paper is to consider and introduce new variants of EKI which are motivated from earlier literature of parameter estimation from data assimilation. These include new variants where we modify known methods such as EKI and regularized Tikhonov EKI, whilst introducing an EKI variant based on ideas of statistical linearization. As a result we are able to provide a unified perspective on these methods, through both their discrete forms and their respective continuous-time limits. We test our new methodologies on a variety of applied problems arising in geophysical sciences, numerical weather prediction and machine learning. 
\end{abstract}

\section{Introduction}
This paper provides an accessible introduction to the derivation and foundations of iterative ensemble Kalman methods, a family of derivative-free algorithms for parameter reconstruction and other related tasks. The overarching theme behind these methods is to iteratively update via Kalman-type formulae an ensemble of candidate reconstructions,  aiming to bring the ensemble closer to the unknown parameter with each iteration. The ensemble Kalman updates approximate derivative-based nonlinear least-squares optimization schemes without requiring gradient evaluations. Our presentation emphasizes that iterative ensemble Kalman methods can be naturally classified in terms of the nonlinear least-squares objective they seek to minimize and the derivative-based optimization scheme they approximate through the ensemble. This perspective allows us to identify three subfamilies of iterative ensemble Kalman methods, creating unity into the growing  literature on this subject. Our work also emphasizes two principles for the derivation and analysis of iterative ensemble Kalman methods: statistical linearization and continuum limits. Following these principles we introduce new iterative ensemble Kalman methods that show promising numerical performance in Bayesian inverse problems, data assimilation and machine learning tasks.

\normalem

We consider the application of iterative ensemble Kalman methods to the problem of reconstructing an unknown $u\in \R^d$ from corrupt data $y\in \R^k$ related by 
\begin{equation}\label{IP}
y = h(u) + \eta,
\end{equation}
where $\eta$ represents measurement or model error and $h$ is a given map. A wide range of inverse problems, data assimilation and machine learning tasks can be cast into the framework \eqref{IP}. In these applications the unknown $u$ may  represent, for instance, an input parameter of a differential equation, the current state of a time-evolving signal and a  regressor, respectively. 
 Ensemble Kalman methods were first introduced as filtering schemes for sequential data assimilation \cite{GE09,EL96,MH12,RC13,SST19} to reduce the computational cost of the Kalman filter \cite{REK60}. Their use for state and parameter estimation and inverse problems was further developed in \cite{A01,LFFLNV01,NMV02,SE01}. The idea of \emph{iterating} these methods was considered in \cite{CO12,ER13}. Iterative ensemble Kalman methods are now popular in inverse problems and data assimilation; they have also shown some potential in machine learning applications \cite{HLR18,GSW20,KS19}. 

Starting from an initial ensemble $\{u_0^{(n)} \}_{n=1}^N,$ iterative ensemble Kalman methods use various ensemble-based empirical means and covariances to update 
\begin{equation}\label{eq:update}
\{u_i^{(n)} \}_{n=1}^N  \to \{u_{i+1}^{(n)}\}_{n=1}^N,
\end{equation}
until a stopping criteria is satisfied; the unknown parameter $u$ is reconstructed by the mean of the final ensemble. The idea is analogous to classical Kalman methods and optimization schemes which, starting with a \emph{single} initialization $u_0$ use evaluations of derivatives of $h$ to iteratively update $u_i \to u_{i+1}$ until a stopping criteria is met. The initial ensemble is viewed as an input to the algorithm, obtained in a problem-dependent fashion. In Bayesian inverse problems and machine learning it may be obtained by sampling a prior, while in data assimilation the initial ensemble may be a given collection of particles that approximates the prediction distribution. In either case, it is useful to view the initial ensemble as a sample from a probability distribution.

There are two main computational benefits in updating an \emph{ensemble} of candidate reconstructions rather than a single estimate. First, the ensemble update can be performed without evaluating derivatives of $h,$ effectively approximating them using \emph{statistical linearization}. This is important in applications where computing derivatives of $h$ is expensive, or where the map $h$ needs to be treated as a black-box. Second, the use of empirical rather than model covariances can significantly reduce the computational cost whenever the ensemble size $N$ is smaller than the dimension $d$ of the unknown parameter $u$. Another key advantage of the ensemble approach is that, for problems that are not strongly nonlinear, the spread of the ensemble may contain meaningful information on the uncertainty in the reconstruction.

%\red
%Ensemble methods that propagate a set of particles using its moments were first introduced in sequential data assimilation under the name of Ensemble Kalman filters (EnKF)  \cite{GE09,EL96}  to alleviate the computational complexity of the Kalman filter \cite{REK60}.  As a result the computational complexity was reduced from $\mathcal{O}(d^2)$ to $\mathcal{O}(Nd)$, where $d \in \mathbb{R}$ is the dimension and $N \in \mathbb{R}$ is the ensemble size. Commonly in practical applications only a small number of particles are available i.e. $N \ll d$. Since then, the EnKF has been widely used for state estimation in different applications such as numerical weather prediction and geophysical sciences \cite{MH12,ORL08}. As ensemble methods have the advantage of scaling well in high-dimensions, the methodology has also been applied for parameter estimation. In this context the EnKF has been modified and treated as an iterative optimization procedure.
%\nc

\subsection{Overview: Three Subfamilies}
This paper identifies, compares and further develops three subfamilies of iterative ensemble Kalman methods to implement the ensemble update \eqref{eq:update}. Each subfamily employs a different Kalman-type formulae, determined by a choice of objective to minimize and a derivative-based optimization scheme to approximate with the ensemble. All three approaches impose some form of regularization, either explicitly through the choice of the objective, or implicitly through the choice of the optimization scheme. Incorporating regularization is essential in parameter reconstruction problems encountered in applications, which are typically under-determined or ill-posed \cite{LP11,SST19}.

The first subfamily considers a \emph{Tikhonov-Phillips}  objective associated with the parameter reconstruction problem \eqref{eq:update}, given by
\begin{equation}\label{eq:TPobjective}
\Jtp(u) := \frac12|y-h(u)|^2_R   +  \frac12|u-m|^2_P,
\end{equation}
where $R$ and $P$ are symmetric positive definite matrices that model, respectively, the data measurement precision and the level of regularization ---incorporated explicitly through the choice of objective--- and $m$ represents a background estimate of $u.$ Here and throughout this paper we use the notation $| v |^2_A := |A^{-1/2} v|^2 = v^T A^{-1} v$ for symmetric positive  definite $A$ and vector $v.$ The ensemble is used to approximate a Gauss-Newton method applied to the Tikhonov-Phillips objective $\Jtp.$ Algorithms in this subfamily were first introduced in geophysical data assimilation \cite{ANOR09,CO12,ER13,GO07,LR07,RRZL06} and were inspired by iterative, derivative-based, extended Kalman filters \cite{BMB94,BC93,AHJ07}. Extensions to more challenging problems with strongly nonlinear dynamics are considered in \cite{SOB12,SU12}. In this paper we will use a new Iterative Ensemble Kalman Filter (IEKF) method as a prototypical example of an algorithm that belongs to this subfamily.

The second subfamily considers the \emph{data-misfit} objective
 \begin{equation}\label{eq:objective}
\Jdm(u) :=  \frac12|y-h(u)|^2_R.
\end{equation}
When the parameter reconstruction problem is ill-posed, minimizing $\Jdm$ leads to unstable reconstructions. For this reason, iterative ensemble Kalman methods in this subfamily are complemented with a Levenberg-Marquardt  optimization scheme that implicitly incorporates regularization. The ensemble is used to approximate a regularizing Levenberg-Marquardt  optimization algorithm to minimize $\Jdm.$ Algorithms in this subfamily were introduced in the applied mathematics literature \cite{MAI16,ILS13} building on ideas from classical inverse problems \cite{MH97}. Recent theoretical work has focused on developing continuous-time and mean-field limits, as well as various convergence results \cite{BSWW19,BSW18,CT19,HV18,KS19,SS17}. Methodological extensions based on Bayesian hierarchical techniques were introduced in \cite{NKC18,CIRS18} and the incorporation of constraints has been investigated in \cite{ABL19,CSW19}. In this paper we will use the Ensemble Kalman Inversion (EKI) method \cite{ILS13} as a prototypical example of an algorithm  that belongs to this subfamily.

\begin{table}
\begin{center}
	\begin{tabular}{ | c | c |c|c|c|}
		\hline 
	   Objective  & Optimization & Derivative Method & Ensemble Method	 & New Variant \\ \hline
$\Jtp$ 	&  GN & IExKF  \eqref{ssec:GNTP}  & IEKF \eqref{ssec:IenKFSL} &  IEKF-SL \eqref{sec:IEKS}  \\ \hline
	$\Jdm$ 	& LM & LM-DM \eqref{ssec:LMDM}  &  EKI  \eqref{ssec:EKI}  & EKI-SL \eqref{sec:EKIN}  \\ \hline
	$\Jtp$ 	& LM  & LM-TP \eqref{ssec:LMTP}  & TEKI  \eqref{ssec:TEKI} & TEKI-SL  \eqref{ssec:gradientstructure}    \\ \hline
	\end{tabular}
	\bigskip
	\caption{Roadmap to the algorithms considered in this paper. We use the abbreviations GN and LM for Gauss-Newton and Levenberg-Marquardt. The numbers in parenthesis represent the subsection in which each algorithm is introduced.}		
		\label{AlgorithmsSummary}
		\end{center}
\end{table}

The third subfamily, which has emerged more recently, combines explicit regularization through the Tikhonov-Phillips objective and an implicitly regularizing optimization scheme \cite{CT19,CST20}. Precisely, a Levenberg-Marquardt scheme is approximated through the ensemble in order to minimize the Tikhonov-Phillips objective $\Jtp.$
 In this paper we will use the Tikhonov ensemble Kalman inversion method (TEKI)  \cite{CST20} as a prototypical example.
 
To conclude this overview we note that while in this paper we will only consider least-squares objectives, iterative ensemble Kalman methods that use other regularizers have been recently proposed \cite{KS19,L20,S20}. 

\subsection{Statistical Linearization, Continuum Limits and New Variants}
Each subfamily of iterative ensemble Kalman methods stems from a derivative-based optimization scheme. However, there is substantial freedom as to how to use the ensemble to approximate a derivative-based method. We will focus on randomized-maximum likelihood implementations \cite{GO07,KLS14,SST19}, rather than square-root or ensemble adjustment approaches \cite{A01,T03,G20}. Two principles will guide our derivation and analysis of ensemble methods: the use of statistical linearization  \cite{SU12} and their connection with gradient descent methods through the study of continuum limits \cite{SS17}.

The idea behind statistical linearization is to approximate the gradient of $h$ using  pairs $\big\{ \big(\uin, h(\uin) \big) \big\}_{n=1}^N$ in such a way that if $h$ is linear and the ensemble size $N$ is sufficiently large, the approximation is exact. As we shall see, this idea tacitly underlies the derivation of all the ensemble methods considered in this paper, and will be explicitly employed in our derivation of new variants. Statistical linearization has also been used within Unscented Kalman filters, see e.g. \cite{SU12}.

Differential equations have long been important in developing and understanding optimization schemes \cite{N83}, and investigating the connections between differential equations and optimization is still an active area of research \cite{S18,S14,S16,W16}. In the context of iterative ensemble methods, continuum limit analyses arise from considering small length-steps and have been developed primarily in the context of EKI-type algorithms \cite{BSWW19,SS18}. While the derivative-based algorithms that motivate the ensemble methods result in an ODE continuum limit, the ensemble versions lead to a system of SDEs. Continuum limit analyses are useful in at least three ways. First, they unveil the gradient structure of the optimization schemes.  Second, viewing optimization schemes as arising from discretizations of SDEs lends itself to design of algorithms that are easy to tune: the length-step is chosen to be small and the algorithms are run until statistical equilibrium is reached. Third, a simple linear-case analysis of the SDEs may be used to develop new algorithms that satisfy certain desirable properties. Our new iterative ensemble Kalman methods will be designed following these observations.

While our work advocates the study of continuum limits as a useful tool to design ensemble methods, continuum limits cannot fully capture the full richness and flexibility of discrete-based implementable algorithms, since different algorithms may result in the same SDE continuum limit. This insight suggests that it is not only the study of differential equations, but also their \emph{discretizations}, that may contribute to the design of iterative ensemble Kalman algorithms.

\subsection{Main Contributions and Outline}
In addition to providing a unified perspective of the existing literature, this paper contains several original contributions. We highlight some of them in the following outline and refer to Table \ref{AlgorithmsSummary} for a summary of the algorithms considered in this paper.
\begin{itemize}
	\item In Section \ref{sec:background} we review three iterative derivative-based methods for nonlinear least-squares optimization. The ensemble-based algorithms studied in subsequent sections can be interpreted as ensemble-based approximations of the derivative-based methods described in this section. We also derive informally ODE continuum limits for each method, which unveils their gradient flow structure. 
	\item In Section \ref{sec:EnsembleKalmanLearning} we describe the idea of statistical linearization. We review three subfamilies of iterative ensemble methods, each of which has an update formula analogous to one of the derivate-based methods in Section \ref{sec:background}. We analyze the methods when $h(u) = Hu$ is linear by formally deriving SDE continuum limits that unveil their gradient structure. A novelty in this section is the introduction of the IEKF method, which is similar to, but different from, the iterative ensemble method introduced in \cite{SU12}.  
	\item The material in Section \ref{sec:EnsembleKalmanLearningNew} is novel to the best of our knowledge. We introduce new variants of the iterative ensemble Kalman methods discussed in Section \ref{sec:EnsembleKalmanLearning} and formally derive their SDE continuum limit. We analyze the resulting SDEs when $h(u)=Hu$ is linear. The proposed methods are designed to ensure that (i) no parameter tuning or careful stopping criteria are needed; and (ii) the ensemble covariance contains meaningful information of the uncertainty in the reconstruction in the linear case, avoiding the ensemble collapse of some existing methods. 
	\item In Section \ref{sec:num} we include an in-depth empirical comparison of the performance of the iterative ensemble Kalman methods discussed in Sections \ref{sec:EnsembleKalmanLearning} and \ref{sec:EnsembleKalmanLearningNew}. We consider four problem settings motivated by applications in Bayesian inverse problems, data assimilation and machine learning. Our results illustrate the different behavior of some methods in small noise regimes and the benefits of avoiding ensemble collapse. 
	\item Section \ref{sec:conclusions} concludes and suggests some open directions for further research. 
\end{itemize}

\section{Derivative-based Optimization for Nonlinear Least-Squares}\label{sec:background}
In this section we review the Gauss-Newton and Levenberg-Marquardt methods, two derivative-based approaches for optimization of nonlinear least-squares objectives of the form
\begin{equation}\label{eq:obj}
	\J(u) = \frac12 |r(u)|^2.
\end{equation}
We derive closed formulae for the Gauss-Newton method applied to the Tikhonov-Phillips objective $\Jtp$, as well as for the Levenberg-Marquardt method  applied to the data-misfit objective $\Jdm$ and the Tikhonov-Phillips objective $\Jtp.$ These formulae are the basis for the ensemble, derivative-free methods considered in the next section. 

As we shall see, the search directions of Gauss-Newton and Levenberg-Marquardt methods are found by minimizing a linearization of the least-squares objective. It is thus instructive to consider first linear least-squares optimization before delving into the nonlinear setting.  The following well-known result, that we will use extensively, characterizes the minimizer $\mu$ of the Tikhonov-Phillips objective $\Jtp$ in the case of linear $h(u) = Hu$. 
\begin{lemma}\label{lemma:linear}
	It holds that 
	\begin{equation}\label{eq:linearTikhonov}
		\frac12| y - Hu|_R^2 + \frac12| u - m|^2_P 
		= \frac12 | u - \mu|_C^2 + \beta,
	\end{equation}
	where $\beta$ does not depend on $u,$ and 
	\begin{align}
		C^{-1} &= H^T R^{-1} H + P^{-1},  \label{eq:precision} \\ 
		C^{-1} \mu & = H^T R^{-1} y + P^{-1} m. \label{eq:precisionmean}
	\end{align}
	Equivalently, 
	\begin{align}
		\mu &= m + K(y - Hm), \label{eq:linearmean}\\ 
		C &= (I - K H)P,\label{eq:linearcovariance}
	\end{align}
	where $K$ is the Kalman gain matrix given by
	\begin{equation}\label{eq:kalmangain}
		K = P H^T (HPH^T + R)^{-1} =C H^T R^{-1}.
	\end{equation} 
\end{lemma}
\begin{proof}
	The  formulae  \eqref{eq:precision} and \eqref{eq:precisionmean}  follows by matching linear and quadratic coefficients in $u$ between  
	\begin{equation}
		\frac{1}{2} |u-\mu|^2_C  \quad \quad \text{and}  \quad \quad \frac12|u-m|^2_P + \frac12|y-h(u)|^2_R. 
	\end{equation}
	The formulae  \eqref{eq:linearmean} and \eqref{eq:linearcovariance} as well as the equivalent expressions for the Kalman gain $K$ in Equation \eqref{eq:kalmangain} can be obtained using the matrix inversion lemma \cite{SST19}.
\end{proof}

\paragraph{Bayesian Interpretation}
	Lemma \ref{lemma:linear} has a natural statistical interpretation. Consider a statistical model defined by likelihood $y|u \sim \Nc\bigl( Hu, R  \bigr)$ and prior $u\sim \Nc(m,P).$ Then Equation \eqref{eq:linearTikhonov} shows that the posterior distribution is Gaussian, $u|y\sim \Nc(\mu,C),$ and Equations \eqref{eq:precision}-\eqref{eq:precisionmean} characterize the posterior mean and precision (inverse covariance). We interpret Equation \eqref{eq:linearmean} as providing a closed formula for the posterior \emph{mode}, known as the \emph{maximum a posteriori} (MAP) estimator.
	
More generally, the generative model 
\begin{align}\label{eq:probabilisticmodel}
\begin{split}
u &\sim \Nc(m,P),\\
y|u  &\sim \Nc(h(u), R),
\end{split}
\end{align}
gives rise to a posterior distribution on $u|y$ with density proportional to $\exp\bigl( -\Jtp(u) \bigr). $ Thus, minimization of $\Jtp(u)$ corresponds to maximizing the posterior density under the model \eqref{eq:probabilisticmodel}.

\subsection{Gauss-Newton Optimization of Tikhonov-Phillips Objective}\label{ssec:GNTP}
In this subsection we introduce two ways of writing the  Gauss-Newton update applied to the Tikhonov-Phillips objective $\Jtp$. We recall that the Gauss-Newton method applied to a general least-squares objective $\J(u) = \frac12 |r(u)|^2$ is a line-search method which, starting from an initialization $u_0,$ sets
\begin{align*}
	u_{i+1} = u_i + \alpha_i v_i, \quad \quad i  = 0, 1,\ldots
\end{align*}
where $v_i$ is a search direction defined by
%\begin{equation*}
%x_{i+1} = x_i -  \alpha_i \Bigl(  r'(x_i)^T r'(x_i) \Bigr)^{-1} r'(x_i)^T r(x_i),
%\end{equation*}
\begin{align}\label{eq:definitionJl}
	v_i = \arg\min_v \Jl(v), \quad \quad \Jl(v):= \frac12|r'(u_i)v + r(u_i) |^2,
\end{align}
where $\alpha_i>0$ is a length-step parameter whose choice will be discussed later. 
In order to apply the Gauss-Newton method to the Tikhonov-Phillips objective, we write $\Jtp$  in the standard nonlinear least-squares form \eqref{eq:obj}. Note that
\begin{align*}
	\Jtp(u) &= \frac12|u-m|^2_P + \frac12|y-h(u)|^2_R \\
	&= \frac12|z-g(u)|^2_Q,
\end{align*}
where
$$z := \begin{bmatrix}
	y\\ m
\end{bmatrix},  \quad \quad \quad 
g(u) := \begin{bmatrix}
	h(u)\\ u
\end{bmatrix},
\quad \quad \quad
Q := 
\begin{bmatrix}
	R & 0 \\ 0 & P
\end{bmatrix}.
$$
Therefore we have
\begin{align}\label{eq:objectiveGN}
	\Jtp(u) = \frac12|\rtp(u)|^2, \quad \quad \rtp(u) :=Q^{-1/2}\bigl( z - g(u) \bigr).
\end{align}

%We start by considering the GN method with constant step-length $\alpha_i =1$ and then discuss adaptive choice of length-step, stopping criteria and continuum limits. 

The following result is a direct consequence of Lemma \ref{lemma:linear}. 
\begin{lemma}[\cite{BC93}]
	The Gauss-Newton method applied to the Tikhonov-Phillips objective $\Jtp$ admits the characterizations: 
	\begin{equation}
		u_{i+1}  = u_i +  \alpha_i C_i\Bigl\{ H_i^T R^{-1} \bigl( y - h(u_i) \bigr)  + P^{-1}(m - u_i)  \Bigr\}\label{GNTP1},
	\end{equation}
	and 
	\begin{equation}
		u_{i+1}  =  u_i  + \alpha_i  \Bigl\{K_i \bigl( y - h(u_i) \bigr) +  (I- K_i H_i) (m - u_i)   \Bigl\}, \label{GNTP2} 
	\end{equation}
 where $H_i = h'(u_i)$ and 
	\begin{align*}
		K_i &= P H_i^T (H_iPH_i^T + R)^{-1},  \\
		C_i & = (I - K_i H_i) P.
	\end{align*}
\end{lemma}
\begin{proof}
	The search direction $v_i$  of Gauss-Newton for the objective $\Jtp$ is given by 
	\begin{align}
		v_i &= \arg \min_v \Jtpl (v)  \\
		&= \arg \min_v  \frac12 \bigl|\rtp'(u_i)v + \rtp(u_i) \bigr|^2 \\
		& = \arg \min_v   \frac12 \bigl| z - g(u_i) - g'(u_i) v \bigr|_Q      \\
		& = \arg \min_v  \biggl\{  \frac12 \bigl| y - h(u_i) - h'(u_i) v \bigr|^2_R + \frac12 \bigl|v - (m - u_i) \bigr|_P^2 \biggr\} . \label{eq:GN_obj}
	\end{align}
	Applying Lemma \ref{lemma:linear},  using formulae \eqref{eq:precisionmean} and \eqref{eq:linearcovariance},  we deduce that
	\begin{align*}
		v_i =  C_i\Bigl\{ H_i^T R^{-1} \bigl( y - h(u_i) \bigr)  + P^{-1}(m - u_i)  \Bigr\},
	\end{align*}
	which establishes the characterization \eqref{GNTP1}.
	The equivalence between \eqref{GNTP1} and \eqref{GNTP2} follows from the identity \eqref{eq:kalmangain}, which implies that $C_i H_i^T R^{-1} = K_i$ and $C_iP^{-1} = I-K_iH_i.$ 
\end{proof}

%In this subsection we describe iterative extended Kalman filters \cite{AHJ07}.
%Lemma \ref{lemma:linear} shows that if $h$ is linear, then $\mu$ given by \eqref{eq:linearmean} is the minimizer of the objective \eqref{eq:TPobjective}. It then seems natural to  replace $H$ by $h'(m)$ in the nonlinear case to obtain an approximation to the minimizer of \eqref{eq:TPobjective}. This first approximation may be improved by iterating the same idea, leading to the following algorithm. 

We refer to the Gauss-Newton method  with constant length-step $\alpha_i = \alpha$  applied to $\Jtp$ as the Iterative Extended Kalman Filter (IExKF) algorithm. IExKF was developed in the control theory literature \cite{AHJ07} without reference to the Gauss-Newton optimization method; the agreement between both methods was established in \cite{BC93}. In order to compare IExKF with an ensemble-based method in Section \ref{sec:EnsembleKalmanLearning}, we summarize it here.

\FloatBarrier
\begin{algorithm}
	\caption{Iterative Extended Kalman Filter (IExKF) \label{itexKF}}
	\vspace{0.1in}
	\STATE {\bf Input}: Initialization $u_0 = m$, length-step $\alpha$.  \\
	\STATE For $i = 0, 1, \ldots$ do:
	\begin{enumerate}
		\item Set $K_i = P H_i^T (H_i P H_i^T + R)^{-1}, \quad \quad H_i = h'(u_i). $
		\item  Set 
		\begin{equation}\label{eq:updateiexKF}
			u_{i+1} = u_i + \alpha \Bigl\{ K_i \big( y - h(u_i) \big) + (I - K_i H_i)(m-u_i)   \Bigr\},
		\end{equation}
		or, equivalently, 
		\begin{equation}\label{eq:updateiexKFR}
			u_{i+1} = u_i + \alpha C_i\Bigl\{  H_i^T R^{-1}( y - h(u_i) ) + P^{-1}(m - u_i)  \Bigr\}.
		\end{equation}
	\end{enumerate}
	\STATE{\bf Output}: $u_1, u_2, \ldots$
\end{algorithm}
\FloatBarrier
The next proposition shows that in the linear case,  if $\alpha$ is set to 1, IExKF finds the minimizer of the objective \eqref{eq:TPobjective} in one iteration, and further iterations still agree with the minimizer.
\begin{proposition}\label{proposition:IExKF}
	Suppose that $h(u) = Hu$ is linear  and $\alpha = 1$. Then the output of Algorithm \ref{itexKF} satisfies
	$$ u_i = \mu, \quad \quad   i  = 1, 2, \ldots $$
	where $\mu$ is the minimizer of the Tikhonov-Phillips objective \eqref{eq:TPobjective}.
\end{proposition}
\begin{proof}
	In the linear case we have 
	$$H_i = H, \quad \quad K_i = K = PH^T (H P H^T + R)^{-1}, \quad \quad i  = 0, 1,\ldots$$
	Therefore,  update \eqref{eq:updateiexKF} simplifies as 
	$$u_{i+1} = m + K(y - Hm), \quad i  = 0, 1,\ldots$$
	This implies that, for all $i\ge 1,$   it holds that $u_i=\mu$ with $\mu$ defined in Equation \eqref{eq:linearmean}. 
\end{proof}

\paragraph{Choice of Length-Step}
When implementing Gauss-Newton methods, it is standard practice to perform a line search in the direction of $v_i$ to adaptively choose the length-step $\alpha_i.$ For instance, a common strategy is to guarantee that the Wolfe conditions are satisfied \cite{DS96,NW06}. In this paper we will instead simply set $\alpha_i = \alpha$ for some fixed value of $\alpha,$ and we will follow a similar approach for all the derivative-based and ensemble-based algorithms we consider. There are two main motivations for doing so. First, it is appealing from a practical viewpoint to avoid performing a line search for ensemble-based algorithms. Second, when $\alpha$ is small each derivative-based algorithms we consider can be interpreted as a discretization of an ODE system, while the ensemble-based methods arise as discretizations of SDE systems. These ODEs and SDEs allow us to compare and gain transparent understanding of the gradient structure of the algorithms. They will also allow us to propose some new variants of existing ensemble Kalman methods.  We next describe the continuum limit structure of IExKF.

\paragraph{Continuum Limit}
It is not hard to check that the term in brackets in the update \eqref{eq:updateiexKFR}
$$
H_i^T R^{-1}(y - h(u_i)) + P^{-1}(m - u_i)
$$
is the negative gradient of $\Jtp(u)$, which reveals the following gradient flow structure in the limit of small length-step $\alpha:$
\begin{align}
\begin{split}
	\dot{u} &= C(t) \Bigl\{   h'\bigl(u(t) \bigr)^T R^{-1} \bigl( y - h\big( u(t) \big) \bigr)  + P^{-1}\bigl( m - u(t)\bigr)\Bigr\} \\
	& = - C(t)  \Jtp'\bigr( u(t)  \bigl), \label{gn_cont_limit}
	\end{split}
\end{align}
with  preconditioner
$$C(t) :=\biggl( h'\bigl(u(t)\bigr)^T R^{-1} h'\bigl(u(t) \bigr) + P^{-1}  \biggr)^{-1}.$$
We remark that in the linear case, $C(t) \equiv C,$ where  $C$ is the posterior covariance given by \eqref{eq:linearcovariance}, which agrees with the  inverse of the  Hessian of the Tikhonov Phillips objective.

\subsection{Levenberg-Marquardt Optimization of Data-Misfit Objective}\label{ssec:LMDM}
In this subsection we introduce the Levenberg-Marquardt algorithm and describe its application to the data misfit objective $\Jdm$. We recall that the Levenberg-Marquardt method applied to a general least-squares objective $\J(u) = \frac12 |r(u)|^2$ is a trust region method which, starting from an initialization $u_0,$ sets
\begin{align*}
	u_{i+1} = u_i +  v_i, \quad \quad i  = 0, 1,\ldots
\end{align*}
where
\begin{equation*}
	v_i =  \arg\min_{v} \Jl(v), \quad \text{s.t.} \,\, |v|_P^2 \le \delta_i,  \quad \quad \Jl(v):= \frac12|r'(u_i)v + r(u_i) |^2.
\end{equation*}
Similar to Gauss-Newton methods, the increment $v_i$ is defined as the minimizer of a linearized objective, but now the minimization is constrained to a ball $\{ |v|_P^2 \le \delta_i\}$ in which we \emph{trust} that the objective can be replaced by its linearization. The increment can also be written as
\begin{equation*}
	v_i = \arg\min_v  \Juc \hspace{-0.2cm}(v) ,
\end{equation*}
where
\begin{equation}\label{eq:LMobjectivewithlengthstep}
	\Juc \hspace{-0.2cm}(v)  =  \Jl(v)   + \frac{1}{2\alpha_i} | v |_P^2 .
\end{equation}
The parameter $\alpha_i>0$ plays an analogous role to the length-step in Gauss-Newton methods. Note that the Levenberg-Marquardt increment is the  unconstrained  minimizer of a \emph{regularized} objective. It is for this reason that we say that Levenberg-Marquardt provides an implicit regularization. 

We next consider application of the Levenberg-Marquardt method to the data-misfit objective $\Jdm,$ which we write in standard nonlinear least-squares form:
\begin{align}\label{eq:objectiveLM}
	 \Jdm(u) = \frac12|\rdm(u)|^2, \quad \quad \rdm(u) :=R^{-1/2}\bigl(y - h(u) \bigr). 
\end{align}

\begin{lemma}
	The Levenberg-Marquardt method applied to the data misfit objective $\Jdm$ admits the following characterization:
	\begin{equation}
		u_{i+1} = u_i + K_i \Bigl\{y - h(u_i) \Bigr\},
	\end{equation}
	where $$K_i = \alpha_i P H_i^T (\alpha_i H_i P H_i^T + R)^{-1}, \quad \quad H_i = h'(u_i).$$
\end{lemma} 
\begin{proof}
	Note that the increment $v_i$ is defined as the unconstrained minimizer of
	\begin{align}
		\begin{split}
			\Jucdm(v) &=   \frac12|\rdm'(u_i)v + \rdm(u_i) |^2  + \frac{1}{2\alpha_i} |v|^2_P  \\
			& = \frac12| y - h(u_i) - h'(u_i)v |^2_R + \frac{1}{2\alpha_i} |v|^2_P.\label{eq:LMDM}
		\end{split}
	\end{align}
	The result follows from Lemma \ref{lemma:linear}. 
\end{proof}
Similar to the previous section,  we will focus on implementations with constant length-step $\alpha_i = \alpha$, which leads to the following algorithm.

\FloatBarrier
\begin{algorithm}
	\caption{Iterative Levenberg-Marquardt with Data Misfit (ILM-DM) \label{ILMDM}}
	\vspace{0.1in}
	\STATE {\bf Input}: Initialization $u_0 = m$, length-step $\alpha$.  \\
	\STATE For $i = 0, 1, \ldots$ do:
	\begin{enumerate}
		\item Set  $K_i = \alpha P H_i^T (\alpha H_i P H_i^T + R)^{-1}, \quad \quad H_i = h'(u_i).$
		\item Set 
		\begin{equation}\label{eq:ILMDM}
			u_{i+1} = u_i + K_i \Bigl\{ y - h(u_i)  \Bigr\}.
		\end{equation}
	\end{enumerate}
	\STATE{\bf Output}: $u_1, u_2, \ldots$
\end{algorithm}
\FloatBarrier
When $\alpha = 1$, the following linear-case result shows that ILM-DM reaches the minimizer of $\Jtp$ in one iteration. However, in contrast to IExKF, further iterations of ILM-DM  will typically worsen the optimization of $\Jtp$, and start moving towards minimizers of $\Jdm.$
\begin{proposition}\label{proposition:ILMDM}
	Suppose that $h(u) = Hu$ is linear  and $\alpha = 1$. Then the output of Algorithm \ref{ILMDM} satisfies
	$$ u_1= \arg\min_u \Jtp(u),$$
	where $\Jtp$ is the Tikhonov-Phillips objective \eqref{eq:TPobjective}.
\end{proposition}
\begin{proof}
	The proof is identical to that of Proposition \ref{proposition:IExKF}, noting that in the linear case
	$u_{i+1} = u_i + K(y - Hu_i).$
\end{proof}

\begin{example}[Convergence of ILM-DM with invertible observation map]\label{example:invertible}
Suppose that $H \in \R^{d\times d}$ is invertible and $\alpha = 1.$ Then, writing 
$$u_{i+1} = (I-KH) u_i + Ky$$
and noting that $\rho(I-KH)<1$ \cite{AM79}, it follows that $u_i \to u^*$, where $u^*$ is the unique solution to $y = Hu.$ That is, the iterates of ILM-DM converge to the unique minimizer of the data misfit objective $\Jdm.$
\end{example}

\paragraph{Choice of Length-Step}
When implementing Levenberg-Marquardt algorithms, the length-step parameter $\alpha_i$ is often chosen adaptively, based on the objective. However, similar to  Section \ref{ssec:GNTP}, we fix $\alpha_i=\alpha$ to be a small value which leads to an ODE continuum limit. % so that we can gain better theoretical understandings from its continuous ODE limit.
%We now investigate the role of the length-step parameter $\alpha_i$ in Equation \eqref{eq:LMobjectivewithlengthstep}. The update formulae can be derived as above, replacing $P$ with $\alpha_i P,$ leading to a Kalman gain 
%\begin{equation}
%K_i = \alpha_i P H_i^T (\alpha_i H_i P H_i^T + R)^{-1} 
%\end{equation}
%and update
%\begin{equation}
%u_{i+1} = u_i + \alpha_i P H_i^T (\alpha_i H_i P H_i^T + R)^{-1} \Bigl\{ y - h(u_i) \Bigr\}.
%\end{equation}
%\paragraph{Stopping Criteria}
%
\paragraph{Continuum Limit}
We notice that, in the limit of small length-step $\alpha$, update \eqref{eq:ILMDM} can be written as
\begin{equation*}
	u_{i+1} = u_i + \alpha P H_i^T R^{-1} \Bigl\{ y - h(u_i)  \Bigr\}.
\end{equation*}
The term $H_i^T R^{-1} \Bigl\{ y - h(u_i)  \Bigr\}$ is the negative gradient of $\Jdm(u)$, which reveals the following gradient flow structure
\begin{align}
\begin{split}
	\dot{u} &= P h'\bigl( u(t) \bigr)^T R^{-1} \Bigl\{ y - h \bigl(u(t)\bigr) \Bigr\} \\
	& = - P \Jdm'(u), \label{eq:lm_cont_limit}
	\end{split}
\end{align}
where the preconditioner $P$ is interpreted as the prior covariance in the Bayesian framework.\nc

\subsection{Levenberg-Marquardt Optimization of Tikhonov-Phillips Objective}\label{ssec:LMTP}
In this subsection we describe the application of the Levenberg-Marquardt algorithm to the Tikhonov-Phillips objective $\Jtp$. %Notice that most of the arguments are similar to what we discussed in Section \ref{ssec:LMDM}, since we can replace $H_i$ by $G_i$.

\begin{lemma}
	The Levenberg-Marquardt method applied to the data misfit objective $\Jtp$ admits the following characterization:
	\begin{align*}
		u_{i+1} = u_i + K_i  \Bigl\{ z - g(u_i) \Bigr\},
	\end{align*}
	where 
	$$K_i = \alpha_i P G_i^T ( \alpha_i G_i P G_i^T + Q )^{-1}, \quad \quad G_i = g'(u_i).$$
\end{lemma} 
\begin{proof}
	Note that the increment $v_i$ is defined as the unconstrained minimizer of
	\begin{align}
		\Juctp(v) &=  \Jtpl(v)  + \frac{1}{2\alpha_i}|v|^2_P \\
		& = \frac12 |z - g(u_i) -g'(u_i)v |_Q^2 + \frac{1}{2\alpha_i}|v|^2_P, \label{eq:LMTP_obj}
	\end{align}
	which has the same form as Equation \eqref{eq:LMDM} replacing $y$ with $z,$  $h$ with $g,$ and $R$ with $Q.$ 
\end{proof}

Setting $\alpha_i = \alpha$ leads to the following algorithm.
\FloatBarrier
\begin{algorithm}
	\caption{Iterative Levenberg-Marquardt with Tikhonov-Phillips (ILM-TP) \label{ILMTP}}
	\vspace{0.1in}
	\STATE {\bf Input}: Initialization $u_0 = m$, length-step $\alpha$.  \\
	\STATE For $i = 0, 1, \ldots$ do:
	\begin{enumerate}
		\item Set $K_i = \alpha P G_i^T (\alpha G_i P G_i^T + Q)^{-1}, \quad \quad G_i = g'(u_i).$
		\item Set 
		\begin{equation}\label{eq:ILMTP}
			u_{i+1} = u_i + K_i \Bigl\{ z - g(u_i)  \Bigr\}.
		\end{equation}
	\end{enumerate}
	\STATE{\bf Output}: $u_1, u_2, \ldots$
\end{algorithm}
\FloatBarrier

%\blue (Keep the following?) \nc
\begin{proposition}\label{proposition:ILMTP}
	Suppose that $h(u) = Hu$ is linear and $\alpha = 1$. The output of Algorithm \ref{ILMTP} satisfies
	\begin{equation}
		u_1= \arg \min_u \Bigl\{ \Jtp(u) + \frac12|u - m|^2   \Bigr\}.
	\end{equation} 
\end{proposition}
\begin{proof}
	The result is a corollary of Proposition \ref{proposition:ILMDM}. To see this, note that $\Jtp(u) = \frac12 |z - g(u)|^2_Q$ can be viewed as a data-misfit objective with data $z,$ forward model $g(u)$ and observation matrix $Q.$ Then,  ILM-TP can be interpreted as applying ILM-DM to $\Jtp,$ and the objective $\Jtp(u) + \frac12|u-m|^2$ as its Tikhonov-Phillips regularization.
\end{proof}
\begin{example}[Convergence of ILM-TP with linear invertible observation map.]
It is again instructive to consider the case where $H \in \R^{d\times d}$ is invertible. Then, following the same reasoning as in Example \ref{example:invertible}  we deduce that the iterates of ILM-TP converge to the unique minimizer of $\Jtp.$
\end{example}

\paragraph{Continuum Limits} In the limit of small length-step $\alpha$, we can derive the gradient flow structure of update \eqref{eq:ILMTP}. This is similar to Section \ref{ssec:LMDM}, with $\Jdm$ replaced by $\Jtp$. Precisely, the gradient flow of ILM-TP is given by
\begin{align*}
	\dot{u} &= P g'\bigl( u(t) \bigr)^T Q^{-1} \Bigl\{ z - g\bigl(u(t)\bigr) \Bigr\} \\
	& = -P \Jtp'(u).
\end{align*}

%\begin{table}
%	\begin{center}
%		\begin{tabular}{ | c | c |c|c|}
%			\hline
%			 Objective & Optimization & Continuum Limit \\ \hline
%			IExKF &   $\Jtp$ & GN & $\dot{u} = -\mathcal{C}(t) \Jtp'(u) $  \\ \hline
%			ILM-DM  &$\Jdm$ & LM &  $\dot{u} = - P \Jdm'(u) $       \\ \hline
%			ILM-TP   &$\Jtp$ & LM  &   $\dot{u} = - P \Jtp'(u) $\\ \hline
%		\end{tabular}
%		\bigskip
%		\caption{Comparison of continuum limits for different iterative derivative-based methods.}		
%		\label{Updateswithgradients}
%	\end{center}
%\end{table}
%\FloatBarrier

%%%%%%%%%%%%%%%%%%%%%%%%%%%%%%%SECTION 3%%%%%%%%%%%%%%%%%%%%%%%%%%%%%%%%%%%%%%%%%%%%%%%%%%%%%%%%%%%%%%

\section{Ensemble-based Optimization for Nonlinear Least-Squares}
\label{sec:EnsembleKalmanLearning}
In this section we review three subfamilies of iterative methods that update an ensemble $\{ u_i^{(n)} \}_{n=1}^N$ employing Kalman-based formulae, where $i = 0, 1, \ldots$ denotes the iteration index and $N$ is a fixed ensemble size. Each ensemble member $u_i^{(n)}$ is updated by optimizing a (random) objective $\J_i^{(n)}$ defined using the current ensemble $\{ u_i^{(n)} \}_{n=1}^N$ and/or the initial ensemble $\{ u_0^{(n)} \}_{n=1}^N$. The optimization is performed without evaluating derivatives by invoking a \emph{statistical linearization} of a Gauss-Newton or Levenberg-Marquardt algorithm. In analogy with the previous section, the three subfamilies of ensemble methods we consider differ in the choice of the objective and in the choice of the optimization algorithm. Table \ref{Updateswithgradients} summarizes the derivative and ensemble methods considered in the previous and the current section.

\FloatBarrier
\begin{table}
	\begin{center}
		\begin{tabular}{ | c | c |c|c|}
			\hline
			Objective & Optimization & Derivative Method & Ensemble Method    \\ \hline
		 $\Jtp$ & GN &  IExKF & IEKF  \\ \hline
			 $\Jdm$ & LM  & ILM-DM  & EKI    \\ \hline
			 $\Jtp$ & LM & ILM-TP   & TEKI  \\ \hline
		\end{tabular}
		\bigskip
		\caption{Summary of the main algorithms in Sections \ref{sec:background} and \ref{sec:EnsembleKalmanLearning}.}		
		\label{Updateswithgradients}
	\end{center}
\end{table}

Given an ensemble $\{u_i^{(n)} \}_{n=1}^N$ we use the following notation for ensemble empirical means 
\begin{align*}
m_i &= \frac{1}{N} \sum_{n=1}^N u_i^{(n)}, \quad \quad h_i = \frac{1}{N} \sum_{n=1}^N h(u_i^{(n)}),
\end{align*}
and empirical covariances
\begin{align*}
P_i^{uu}  &= \frac{1}{N} \sum_{n=1}^N (u_i^{(n)} - m_i)  (u_i^{(n)} - m_i)^T,  \quad \quad  \\
P_i^{uy} &= \frac{1}{N} \sum_{n=1}^N \bigl( u_i^{(n)} - m_i  \bigr)   \bigl( h(u_i^{(n)}) - h_i  \bigr)^T,  \\
P_i^{yy} &= \frac{1}{N} \sum_{n=1}^N \bigl(h(u_i^{(n)}) - h_i \bigr)  \bigl(h(u_i^{(n)}) - h_i \bigr)^T.
\end{align*}

Two overarching themes that underlie the derivation and analysis of the ensemble methods studied in this section and the following one are the use of a statistical linearization to avoid evaluation of derivatives, and the study of continuum limits. We next introduce these two ideas. 

\paragraph{Statistical Linearization} If $h(u) = Hu$ is linear, we have
\begin{equation*}
P_i^{uy} = P_i^{uu} H^T,
\end{equation*}
which motivates the approximation in the general nonlinear case
\begin{equation}\label{eq:stat_linearization}
h'(\uin) \approx (P_i^{uy})^T (P_i^{uu})^{-1} =: H_i,  \quad \quad n = 1,\dots,N,
\end{equation}
where here and in what follows $(P_i^{uu})^{-1}$ denotes the pseudoinverse of $P_i^{uu}.$
Notice that \eqref{eq:stat_linearization} can be regarded as a linear least-squares fit of pairs $\big\{ \big(\uin, h(\uin) \big) \big\}_{n=1}^N$  normalized around their corresponding empirical means $m_i$ and $h_i$. We remark that in order for the approximation in \eqref{eq:stat_linearization} to be accurate, the ensemble size $N$ should not be much smaller than the input dimension $d$.

\paragraph{Continuum Limit}
We will gain theoretical understanding by studying continuum limits. Specifically, each algorithm includes a length-step parameter $\alpha>0,$ and the evolution of the ensemble for small $\alpha$ can be interpreted as a discretization of an SDE system. We denote by  $\{u^{(n)}(t)\}_{n=1}^N$ the sample paths of the underlying SDE. For each $1 \le n \le N$, we have $u_0^{(n)} = u^{(n)}(0)$, and we view $\uin$ as an approximation of $u^{(n)}(t)$ for $t=\alpha i$. Similarly as above, we define $P^{uu}(t), P^{uy}(t), P^{yy}(t)$ as the corresponding empirical covariances at time $t \ge 0$.

\subsection{Ensemble Gauss-Newton Optimization of Tikhonov-Phillips Objective}\label{ssec:IenKFSL}
\subsubsection{Iterative Ensemble Kalman Filter}
Given an ensemble $\{u_i^{(n)} \}_{n=1}^N$, consider the following Gauss-Newton update for each $n$:
\begin{equation}\label{eq:GNupdateensemble}
u_{i+1}^{(n)} = u_i^{(n)} + \alpha v_i^{(n)},
\end{equation}
where $\alpha>0$ is the length-step, and $\vin$ is the minimizer of the following (linearized) Tikhonov-Phillips objective  (cf. equation \eqref{eq:GN_obj})
\begin{equation}\label{eq:IEKF_new_C}
\Jtpn(v) = \frac12 \bigl| y_i^{(n)} - h(u_i^{(n)} )  - H_i v \bigr|^2_R + \frac12 \bigl| u_0^{(n)} - \uin - v \bigr|^2_{P_0^{uu}}, \quad \quad y_i^{(n)} \sim \Nc(y,\alpha^{-1} R).
\end{equation}
Notice that we adopt the statistical linearization  \eqref{eq:stat_linearization} in the above formulation. Applying Lemma \ref{lemma:linear}, the minimizer $\vin$ can be calculated as
\begin{equation}\label{eq:IEKF_step1}
v_i^{(n)} = C_i \Bigl\{ H_i^T R^{-1} \big( \yin - h(\uin) \big) + (P_0^{uu})^{-1} \big( u_0^{(n)} - \uin \big) \Bigr\},
\end{equation}
or, in an equivalent form, 
\begin{equation}\label{eq:IEKF_step2}
v_i^{(n)} = K_i \big( \yin - h(\uin) \big) + (I - K_i H_i) (u_0^{(n)} - u_i^{(n)}), 
\end{equation}                      
where
\begin{align*}
C_i &= \big( H_i^T R^{-1} H_i + (P_0^{uu})^{-1} \big)^{-1}, \\
K_i &= P_0^{uu} H_i^T (H_i P_0^{uu} H_i^T + R)^{-1}.
\end{align*}
%GN minimizer of this objective is
%\begin{equation}\label{eq:GNminimizer}
%u^{\rm {\tiny{GN}}} = u_0^{(n)} + K_i \Bigl\{ y_i^{(n)} - h(u_i^{(n)}) - H_i(m-u_i^{(n)})   \Bigr\},
%\end{equation}
%where 
%\begin{equation}\label{eq:kalmanGN}
%K_i = P_0^{uu} H_i^T (H_i P_0^{uu}  H_i^T + R)^{-1}, \quad \quad H_i = h'(u_i).
%\end{equation}
%We propose use of GN update \eqref{eq:GNminimizer} replacing $h'(u_i)$ in \eqref{eq:kalmanGN} with a statistical linearization. Precisely, we note that if $h(u) = Hu$ is linear, then 
%\begin{equation}
%\text{Cov} \bigl(u,h(u)\bigr) = \text{Cov}(u,u) H^T,
%\end{equation}
%which motivates setting 
%\begin{equation}\label{eq:sensitivity}
%H_i := (P_i^{uy})^T (P_i^{uu})^{-1},
%\end{equation}
%where $P_i^{uy}$ and $P_i^{uu}$ are empirical covariances defined in Algorithm \ref{itenKF}.
%\FloatBarrier
Combining \eqref{eq:GNupdateensemble} and \eqref{eq:IEKF_step2}  leads to the Iterative Ensemble Kalman Filter (IEKF) algorithm.
\begin{algorithm}
\caption{Iterative Ensemble Kalman Filter (IEKF) \label{itenKF}} %by Statistical Linearization 
\vspace{0.1in}
\STATE {\bf Input}: Initial ensemble $\{ u_0^{(n)}\}_{n=1}^N$  sampled independently from the prior, length-step $\alpha$.  \\
\STATE For $i = 0, 1, \ldots$ do:
\begin{enumerate}
\item Set $K_i = P_0^{uu} H_i^T (H_i P_0^{uu} H_i^T + R)^{-1}, \quad \quad H_i = (P_i^{uy})^T (P_i^{uu})^{-1}.  $
\item Draw $y_i^{(n)} \sim \Nc(y,\alpha^{-1} R)$ and set
\begin{equation}\label{eq:IEKF_update}
u_{i+1}^{(n)} = \uin + \alpha  \Bigl\{ K_i \big( y_i^{(n)}  - h(u_i^{(n)}) \big) + (I - K_i H_i)  \big( u_0^{(n)} - \uin \big)   \Bigr\}, \quad \quad 1 \le n \le N. 
\end{equation}
\end{enumerate}
\STATE{\bf Output}: Ensemble means $m_1, m_2, \ldots$
\end{algorithm}
\FloatBarrier

 We highlight that IEKF is a natural ensemble-based version of the derivative-based IExKF Algorithm \ref{itexKF}  with update \eqref{eq:updateiexKF}.  Algorithm \ref{itenKF} is a slight modification of the iterative ensemble Kalman algorithm proposed in 
\cite{SU12}. The difference is that \cite{SU12} sets  $H_i = (P_i^{uy})^T (P_0^{uu})^{-1}$ rather than $H_i = (P_i^{uy})^T (P_i^{uu})^{-1}$. Our modification guarantees that Algorithm \ref{itenKF} is well-balanced in the sense that if $\alpha =1$, $u_0^{(n)} \sim \Nc(m,P)$ and  $h(u) = Hu$ is linear, 
then the output of Algorithm \ref{itenKF} satisfies that, as $N\to \infty,$
$$m_i \to \mu, \quad \quad  i  = 1, 2, \ldots$$
where $\mu$ is the minimizer of $\Jtp(u)$ given in Equation \eqref{eq:TPobjective}. This is analogous to Proposition \ref{proposition:IExKF} for IExKF.  A detailed explanation is included in Section \ref{sec:analysis_IEKF} below. 
%\begin{proposition}\label{prop:itEnKF}
%Suppose that $h(u) = Hu$ is linear. Suppose further that $u_0^{(n)} \sim \Nc(m,P).$ 
%Then the output of Algorithm \ref{itenKF} satisfies that, as $N\to \infty,$
%$$m_i \to \mu, \quad \quad  i \ge 1,$$
%where $\mu$ is the minimizer of $\Jtp(u)$ given in Equation \eqref{eq:TPobjective}.
%\begin{proof}
%In the linear case we have
%$$P_i^{uy} = P_i^{uu} H^T, \quad i \ge 0.$$
%Hence, from the definition of $H_i$ in Equation \eqref{eq:sensitivity} and using that $P_i^{uu}$ is symmetric,
%$$H_i = (P_i^{uy})^T (P_i^{uu})^{-1} = H, \quad i \ge 0.$$
%Moreover, since $P_0^{uu} \to P$ as $N\to \infty,$ 
%\begin{equation}\label{eq:auxKalman}
%K_i = P_0^{uu} H_i^T (H_i P_0^{uu} H_i^T + R)^{-1} \to P H^T (H P_0^{uu} H^T + R)^{-1} =:K, \quad i \ge 0.
%\end{equation}
%Therefore, for large $N,$
%\begin{align}\label{eq:auxproof}
%u_{i+1}^{(n)}  &= u_0^{(n)}+ K_i\Bigl\{ y^{(n)} - Hu_i^{(n)} - H(u_0^{(n)} -u_i^{(n)}  )    \Bigr\} \\
%&\approx u_0^{(n)}+ K\Bigl\{ y^{(n)} -  Hu_0^{(n)} \Bigr\}, \quad i \ge 0.
%\end{align}
%This shows that, for all $i \ge 0,$  $u_i^{(n)}$ converges to the minimizer of the randomized objective \eqref{eq:randomobj}. The result follows.
%\end{proof}
%\end{proposition}

Other statistical linearizations and approximations of the Gauss-Newton scheme are possible. We next give a high-level description of the method proposed in \cite{RRZL06}, one of the earliest applications of iterative ensemble Kalman methods for inversion in the petroleum engineering literature. Consider the alternative characterization of the Gauss-Newton update \eqref{eq:IEKF_step1}. However, instead of using a different preconditioner $C_i$ for each step, \cite{RRZL06} uses a fixed preconditioner $C_* = P_0^{uu} - P_0^{uy}(R + P_0^{yy})^{-1}(P_0^{uy})^T$. Note that $C_*$ can be viewed as an approximation of $C_0:$
\begin{equation*}
\begin{split}
C_* &\approx P_0^{uu} - P_0^{uu} H_0^T (R + H_0 P_0^{uu} H_0^T)^{-1} H_0 P_0^{uu} \\
	&=\big( H_0^T R^{-1} H_0 + (P_0^{uu})^{-1} \big)^{-1} = C_0.
\end{split}
\end{equation*}
This leads to the following algorithm. 
%The idea is to approximate by stochastic linearization the alternative characterization of the GN update given in \eqref{GNTP2}. The preconditioner is kept fixed, and the sensibilities are only updated to approximate the gradient term. After describing the algorithm and showing that it is well-balanced we will provide a comparison with Algorithm \ref{itenKF}.

\FloatBarrier
\begin{algorithm}
\caption{Iterative Ensemble Kalman Filter  (IEKF-RZL)\label{itenKFalt}}
\vspace{0.1in}
\STATE {\bf Input}: Initial ensemble $\{ u_0^{(n)}\}_{n=1}^N$ sampled independently from the prior, length-step $\alpha$.  
\begin{enumerate}
\item Set $C_* = P_0^{uu} - P_0^{uy}\bigl( R + P_0^{yy} \bigr)^{-1}  (P_0^{uy})^T . $
\end{enumerate}
\STATE For $i = 0, 1, \ldots$ do:
\begin{enumerate}
\item Set $H_i = (P_i^{uy})^T (P_i^{uu})^{-1}.$
\item Draw $y_i^{(n)} \sim \Nc(y, \alpha^{-1} R)$ and set
$$u_{i+1}^{(n)} = u_i^{(n)} + \alpha\, C_* \Bigl\{ H_i^T R^{-1} \big( \yin - h(\uin) \big) + (P_0^{uu})^{-1} \big( u_0^{(n)} - \uin \big) \Bigr\}, \quad \quad 1 \le n \le N. $$
\end{enumerate}
\STATE{\bf Output}: Ensemble means $m_1, m_2, \ldots$
\end{algorithm}
\FloatBarrier

 We note that IEKF-RZL is a natural ensemble-based version of the derivative-based IExKF Algorithm \ref{itexKF}  with update \eqref{eq:updateiexKFR}.  We have empirically observed in a wide range of numerical experiments that Algorithm \ref{itenKF} is more stable than Algorithm \ref{itenKFalt}, and we now give a heuristic argument for the advantage of Algorithm \ref{itenKF} in small noise regimes. 
\begin{itemize}
\item The update formula in Algorithm \ref{itenKF} is motivated by the large $N$ approximation 
$$P_0^{uu} H_i^T (H_i P_0^{uu} H_i^T + R)^{-1} \approx P H_i^T (H_i P H_i^T + R)^{-1},$$ which only requires that $P_0^{uu}$ is a good approximation of $P$.
\item The update formula in Algorithm \ref{itenKFalt} may be derived by invoking a large $N$ approximation of several terms. In particular, the error arising from the approximation 
$$C_* H_i^T R^{-1} \approx  C_0 H_i^T R^{-1},$$ 
 gets amplified when $R$ is small.
\end{itemize}

Empirical evidence of the instability of Algorithm \ref{itenKFalt} will be given in Section \ref{ex:elliptic2d}.

\subsubsection{Analysis of IEKF}\label{sec:analysis_IEKF}
In the literature \cite{RRZL06, SU12}, the length-step $\alpha$ is sometimes set to be 1. Here we state a simple observation about Algorithm \ref{itenKF} when $\alpha=1$ and $h(u) = Hu$ is linear. We further assume that $H_i \equiv H$ for all $i$. Then $K_i \equiv K_0 := P_0^{uu} H^T (H P_0^{uu} H^T + R)^{-1}$, and the update \eqref{eq:IEKF_update} can be simplified as
\begin{equation*}
\begin{split}
u_{i+1}^{(n)} &= \uin + K_0(\yin - H \uin) + (I - K_0H) (u_0^{(n)} - \uin) \\
		  &= u_0^{(n)} + K_0(\yin - H u_0^{(n)}),
\end{split}
\end{equation*}
where $\yin \sim \Nc(y, R)$. If $\{ u_0^{(n)}\}_{n=1}^N$ are sampled independently from the prior $\Nc(m, P)$ and we let $N\rightarrow \infty$, we have $K_0 \rightarrow K = PH^T(HPH^T + R)^{-1}$ and, by the law of large numbers, for any $i \ge 1$,
\begin{equation*}
\begin{split}
\frac{1}{N}\sum_{n=1}^N \uin &= (I - K_0 H) \cdot \frac{1}{N}\sum_{n=1}^N u_0^{(n)} + K_0 \cdot \frac{1}{N}\sum_{n=1}^N \yin \\
							 &\rightarrow (I - KH) m + K y,
\end{split}
\end{equation*}
which is the posterior mean (and mode) in the linear setting.  In other words, in the large ensemble limit, the ensemble mean recovers the posterior mean \emph{after one iteration.} However, while the choice $\alpha = 1$  may be effective in low dimensional (nonlinear) inverse problems, we do not recommend it when the dimensionality is high, as $H_i$ might not be a good approximation of $H$. Thus, we introduce Algorithms \ref{itenKF} and \ref{itenKFalt} with a choice of length-step $\alpha,$ resembling the derivative-based Gauss-Newton method discussed in Section \ref{ssec:GNTP}. Further analysis of a new variant of Algorithm \ref{itenKF} will be conducted in a continuum limit setting in Section \ref{sec:EnsembleKalmanLearningNew}.

\begin{remark}
Some remarks:
\begin{enumerate}
%\item In general, we do not require that $\{u_0^{(n)}\}_{n=1}^N$ be sampled i.i.d. from the prior, for the algorithms to work. However, the outputs may not match the desired posterior, even in the simpliest linear case, as we no longer have $P_0^{uu} \rightarrow P$ and $\frac{1}{N}\sum_{n=1}^N u_0^{(n)} \rightarrow m$ when $N \rightarrow \infty$.
\item Although this will not be the focus of our paper, the IEKF Algorithm \ref{itenKF} (together with the EKI Algorithm \ref{algorithm:EKI_step} and TEKI Algorithm \ref{algorithm:TEKI_step} to be discussed later) enjoys the `initial subspace property' studied in previous works \cite{ILS13, SS17, CST20} by which, for any $i$ and any initialization of $\{u_0^{(n)}\}_{n=1}^N,$
\begin{equation*}
\text{span} \bigl(\{u_i^{(n)}\}_{n=1}^N\bigr) \subset \text{span} \bigl(\{u_0^{(n)}\}_{n=1}^N\bigr).
\end{equation*} This can be shown easily for the IEKF Algorithm \ref{itenKF} by expanding the $P_0^{uu}$ term in $K_i$ in the update formula \eqref{eq:IEKF_update}. 
%This property can also be generalized to the continuum limits setting as well.
\item Since we assume a Gaussian prior $\Nc(m, P)$ on $u$, a natural idea is to replace $u_0^{(n)}$ by $m$ and $P_0^{uu}$ by $P$ in the update formula \eqref{eq:IEKF_update}. We pursue this idea in Section \ref{sec:IEKS}, where we introduce a new variant of Algorithm \ref{itenKF} and analyze it in the continuum limit setting. While the initial subspace property breaks down, this new variant is numerically promising, as shown in Section \ref{sec:num}.
\end{enumerate}
\end{remark}

%\begin{equation}\label{eq:IEKF_cont}
%du^{(n)} = C(t) \Big( H(t)^T R^{-1} \big( y - h(u^{(n)}) \big) + ( P_0^{uu} )^{-1}(u_0^{(n)} - u^{(n)}) \Big) dt + C(t) H(t) R^{-1/2} dW^{(n)}
%\end{equation}

\subsection{Ensemble Levenberg-Marquardt Optimization of Data-Misfit Objective}\label{ssec:EKI}
\subsubsection{Ensemble Kalman Inversion}
Given an ensemble $\{u_i^{(n)} \}_{n=1}^N$, consider the following Levenberg-Marquardt update for each $n$:
\begin{equation}\label{eq:updateensembleLM}
u_{i+1}^{(n)} = u_i^{(n)} + v_i^{(n)},
\end{equation}
where $v_i^{(n)}$ is the minimizer of the following regularized (linearized) data-misfit objective  (cf. equation \eqref{eq:LMDM}) 
\begin{equation}\label{eq:EKI_new_C}
\Jucdmn \hspace{-0.2cm}(v) = \frac12 \bigl| y_i^{(n)} - h(u_i^{(n)} )  - H_i v \bigr|^2_R  + \frac{1}{2\alpha} \bigl|v\bigr|^2_{P_i^{uu}}, \quad \quad y_i^{(n)} \sim \mathcal{N}(y, \alpha^{-1} R),
\end{equation}
and $\alpha>0$ will be regarded as a length-step. Notice that we adopt the statistical linearization \eqref{eq:stat_linearization} in the above formulation. Applying Lemma \ref{lemma:linear}, we can calculate the minimizer $v_i^{(n)}$ explicitly:
\begin{equation}\label{eq:EKI_step1}
v_i^{(n)} = (H_i^T R^{-1} H_i + \alpha^{-1} (P_i^{uu})^{-1})^{-1} H_i^T R^{-1} \big( \yin - h(\uin) \big),
\end{equation}
or, in an equivalent form, 
\begin{equation}\label{eq:EKI_step2}
v_i^{(n)} = P_i^{uu} H_i^T (H_i P_i^{uu} H_i^T + \alpha^{-1} R)^{-1} \big( \yin - h(\uin) \big).
\end{equation}
We combine \eqref{eq:updateensembleLM} and \eqref{eq:EKI_step2}, substitute $P_i^{uu} H_i^T = P_i^{uy}$, and make another level of approximation $H_i  P_i^{uy} \approx P_i^{yy}$. This leads to the Ensemble Kalman Inversion (EKI) method \cite{ILS13}.
\begin{algorithm}
	\caption{Ensemble Kalman Inversion (EKI) \label{algorithm:EKI_step}}
	\vspace{0.1in}
	\STATE {\bf Input}: Initial ensemble $\{ u_0^{(n)}\}_{n=1}^N$, length-step $\alpha$.  \\ 
	\STATE For $i = 0, 1, \ldots$ do:
	\begin{enumerate}
		\item Set $K_i = P_i^{uy} (P_{i}^{yy} + \alpha^{-1} R)^{-1}. $
		\item Draw $y_i^{(n)} \sim \Nc(y, \alpha^{-1} R)$ and set
		\begin{equation}\label{eq:EKI_update}
		u_{i+1}^{(n)} = u_i^{(n)}+ K_i \Bigl\{ y_i^{(n)}  - h(u_i^{(n)})     \Bigr\} ,\quad \quad 1 \le n \le N. 
		\end{equation}
	\end{enumerate}
	\STATE{\bf Output}: Ensemble means $m_1, m_2, \ldots$
\end{algorithm}
\FloatBarrier
 We note that EKI is a natural ensemble-based version of the derivative-based ILM-DM Algorithm \ref{ILMDM}. However, an important difference is that the Kalman gain in ILM-DM only uses the iterates to update $H_i$ and $P$ is kept fixed. In contrast, the ensemble is used in EKI to update $P_i^{uy}$ and $P_i^{yy}$.  

\subsubsection{Analysis of EKI}
In view of the definition of $K_i$, we can rewrite the update \eqref{eq:EKI_update} as a time-stepping scheme:
\begin{equation}\label{eq:EKI_update_cont}
\begin{split}
u_{i+1}^{(n)} &= u_i^{(n)}+ \alpha P_i^{uy} (\alpha P_i^{yy} +R)^{-1} \bigl( y + \alpha^{-1/2} R^{1/2} \xiin - h(u_i^{(n)}) \bigr)\\
&= u_i^{(n)}+ \alpha P_i^{uy} (\alpha P_i^{yy} +R)^{-1} \bigl( y - h(u_i^{(n)}) \bigr) + \alpha^{1/2}P_i^{uy}(\alpha P_i^{yy} +R)^{-1} R^{1/2} \xiin,
\end{split}
\end{equation}
where $\xiin \sim \Nc(0, I)$ are independent.
Taking the limit $\alpha \rightarrow 0$, we interpret \eqref{eq:EKI_update_cont} as a discretization of the SDE system
\begin{equation}\label{eq:EKI_cont}
\diff u^{(n)} = P^{uy}(t) R^{-1} \bigl( y - h(u^{(n)}) \bigr) \diff t + P^{uy}(t) R^{-1/2} \diff W^{(n)}.
\end{equation}
If $h(u) = Hu$ is linear, the SDE system \eqref{eq:EKI_cont} turns into
\begin{equation}\label{eq:EKI_cont_linear}
\diff u^{(n)} = P^{uu}(t) H^T R^{-1} \big( y - Hu^{(n)} \big) \diff t + P^{uu}(t) H^T R^{-1/2} \diff W^{(n)} .
\end{equation}
\begin{proposition} \label{prop:EKI}
	For the SDE system \eqref{eq:EKI_cont}, assume $h(u) = Hu$ is linear, and suppose that the initial ensemble $\{ u_0^{(n)}\}_{n=1}^N$ is drawn independently from a continuous distribution with finite second moments. Then, in the large ensemble size limit $N \rightarrow \infty$ (mean-field),  the distribution of $ u^{(n)}(t)$ has mean $\mf(t)$ and covariance $\Cf(t)$, which satisfy
	\begin{align}
	\frac{\diff \mf(t)}{\diff t} &= \Cf(t) H^T R^{-1} \bigl(y - H \mf(t)\bigr), \label{eq:EKI_mean_evolve}\\
	\frac{\diff \Cf(t)}{\diff t} &= -\Cf(t) H^T R^{-1} H \Cf(t). \label{eq:EKI_cov_evolve}
	\end{align}
	Furthermore, the solution can be computed analytically:
	\begin{align}
	\mf(t) &= \Big( \Cf(0)^{-1} + t H^T R^{-1} H \Big)^{-1} \big( \Cf(0)^{-1} \mf(0) + t H^T R^{-1} y \big), \label{eq:EKI_mean_sol}\\
	\Cf(t) &= \Big( \Cf(0)^{-1} + t H^T R^{-1} H \Big)^{-1}. \label{eq:EKI_cov_sol}
	\end{align}
	 In particular, if $H \in \R^{k \times d}$ has full column rank (i.e., $d\le k$, ${\emph{rank}(}H) = d$), then, as $t \rightarrow \infty$, 
	\begin{align}
	\mf(t) &\rightarrow (H^T R^{-1} H)^{-1} (H^T R^{-1} y), \label{eq:EKI_mean_limit}\\
	\Cf(t) &\rightarrow 0. \label{eq:EKI_cov_limit}
	\end{align}
	If $H \in \R^{k \times d}$ has full row rank (i.e., $d\ge k$, $\emph{rank}(H) = k$), then, as $t \rightarrow \infty,$
	\begin{align}
	H \mf(t) &\rightarrow y, \label{eq:EKI_mean_limit2}\\
	H \Cf(t) H^T &\rightarrow 0. \label{eq:EKI_cov_limit2}
	\end{align}
\end{proposition}
\begin{proof}
The proof technique is similar to \cite{GHLS19}.  We will use that
\begin{align*}
\mf(t) &= \lim_{N\to \infty} \Expect \bigl[ u^{(n)}(t) \bigr], \\
\Cf(t) & = \lim_{N\to \infty}   \Expect \bigl[ e^{(n)}(t) \otimes e^{(n)}(t) \bigr],
\end{align*}
where $e^{(n)}(t) := u^{(n)}(t) - \mf(t).$
First, note that  \eqref{eq:EKI_mean_evolve} follows directly from \eqref{eq:EKI_cont_linear} using that in the mean field limit $P^{uu}(t)$ can be replaced by $\Cf(t)$. To obtain the evolution of $\Cf(t)$, note that
\begin{equation*}
	\diff \Cf(t) = \lim_{N\to \infty}   \Expect \Bigl[ \diff e^{(n)} \otimes e^{(n)} + e^{(n)} \otimes \diff e^{(n)} + \diff e^{(n)} \otimes \diff e^{(n)}\Bigr],
\end{equation*}
where the last term accounts for the It\^{o} correction. This simplifies as
\begin{equation*}
\begin{split}
	\frac{\diff \Cf(t)}{\diff t} &= \lim_{N\to \infty}   \Expect \Bigl[ -\Cf(t) H^T R^{-1} H (e^{(n)} \otimes e^{(n)}) - (e^{(n)} \otimes e^{(n)}) H^T R^{-1} H \Cf(t) + \Cf(t) H^T R^{-1} H \Cf(t) \Bigr]\\
								 &= -\Cf(t) H^T R^{-1} H \Cf(t),
\end{split}
\end{equation*}
which gives Equation \eqref{eq:EKI_cov_evolve}. To derive exact formulas for $\mf(t)$ and $\Cf(t)$, we notice that
\begin{equation*}
\frac{\diff \Cf(t)^{-1}}{\diff t} = - \Cf(t)^{-1} \frac{\diff \Cf(t)}{\diff t} \Cf(t)^{-1} = H^T R^{-1} H,
\end{equation*}
and
\begin{equation*}
\frac{\diff \big( \Cf(t)^{-1} \mf(t) \big)}{\diff t} = \frac{\diff \Cf(t)^{-1}}{\diff t} \mf(t) + \Cf(t)^{-1} \frac{\diff \mf(t)}{\diff t} = H^T R^{-1} y.
\end{equation*}
Then \eqref{eq:EKI_mean_sol} and \eqref{eq:EKI_cov_sol} follow easily.

 If $H$ has full column rank, then $H^T R^{-1} H$ is invertible and therefore, as $t\rightarrow \infty$, $$\Cf(t) = t^{-1} \big( t^{-1} \Cf(0)^{-1} + H^T R^{-1} H \big)^{-1} \rightarrow  0$$
by continuity of the matrix inverse function. The limit of $\mf (t)$, \eqref{eq:EKI_mean_limit}, follows immediately from \eqref{eq:EKI_mean_sol}. 

If $H$ has full row rank, we make the following substitutions
\begin{equation*}
	\widetilde{\mf}(t) = R^{-1/2} H \mf(t), \quad\quad\quad \widetilde{\Cf}(t) = R^{-1/2} H \Cf(t) H^T R^{-1/2}.
\end{equation*}
Then \eqref{eq:EKI_mean_evolve} and \eqref{eq:EKI_cov_evolve} can be transformed into
\begin{align}
	\frac{\diff \widetilde{\mf}(t)}{\diff t} &= \widetilde{\Cf}(t)  \bigl(R^{-1/2} y -  \widetilde{\mf}(t) \bigr), \label{eq:EKI_tr_mean_evolve}\\
	\frac{\diff \widetilde{\Cf}(t)}{\diff t} &= -\widetilde{\Cf}(t)^2. \label{eq:EKI_tr_cov_evolve}
\end{align}
Using the fact that $\widetilde{\Cf}(0) =R^{-1/2} H \Cf(0) H^T R^{-1/2}$ is invertible, we can solve these using the same technique as in the previous case:
\begin{equation*}
	\widetilde{\Cf}(t) = \big( \widetilde{\Cf}(0)^{-1} + t \big)^{-1}, \quad\quad\quad \widetilde{\mf}(t) = \big( \widetilde{\Cf}(0)^{-1} + t \big)^{-1} \bigl(\widetilde{\Cf}(0)^{-1} \widetilde{\mf}(0) + t R^{-1/2}y\bigr).
\end{equation*}
As $t\rightarrow \infty$, we have $\widetilde{\mf}(t)\rightarrow R^{-1/2}y$ and $\widetilde{\Cf}(t) \rightarrow 0$, which lead to \eqref{eq:EKI_mean_limit2} and \eqref{eq:EKI_cov_limit2}.
\end{proof}

\begin{remark} Some remarks:
	\begin{enumerate}
	 \item In fact, \eqref{eq:EKI_cov_limit2} always holds, without rank constraints on $H$. The proof follows the similar idea as in \cite{SS17}. This can be viewed from Equation \eqref{eq:EKI_tr_cov_evolve}, where we can perform eigenvalue decomposition $\widetilde{\Cf}(0) = X \Lambda(0) X^T$, and show that $\widetilde{\Cf}(t) = X \Lambda(t) X^T$, where $\Lambda(t)$ are diagonal matrices and $\Lambda(t) \rightarrow 0$ as $t \rightarrow \infty$. The statement that $H \Cf(t) H^T \rightarrow 0$ (or $\Cf(t) \rightarrow 0$ if $H$ has full column rank) is referred to as `ensemble collapse'. This can be interpreted as `the images of all the particles under $H$ collapse to a single point as time evolves'. Our numerical results in Section \ref{sec:num} show that the ensemble collapse phenomenon of Algorithm \ref{algorithm:EKI_step} is also observed in a variety of nonlinear examples and outside the mean-field limit, with moderate ensemble size $N$.
	These empirical results  justify the practical significance of the linear continuum analysis in Proposition \ref{prop:EKI}.
%		\item \sout{We will not justify the closeness between \eqref{eq:EKI_update_cont} and \eqref{eq:EKI_cont}. In fact, when the noise covariance $R$ has a small scale, intuitively $\alpha$ has to be made `even smaller' in order to safely remove the $\alpha P_i^{yy}$ term in \eqref{eq:EKI_update_cont}. However, in practice $\alpha$ is often chosen to be a fixed value regardless of noise, or chosen adaptively and increasing with time \cite{KS19}. Therefore \eqref{eq:EKI_update_cont}, or equivalently \eqref{eq:EKI_update}, may not be an accurate approximation of \eqref{eq:EKI_cont} in small noise regimes.}
%		However, for fair comparison, we will follow the literature and use \eqref{eq:EKI_update} as the update equation for EKI in numerical examples. %Most of the existing literature have their analysis based on \eqref{eq:EKI_cont} while conducting their experiments using a fixed $\alpha$ or an adaptive $\alpha$ that is increasing, but we believe \eqref{eq:EKI_update_cont} is not accurate enough to represent \eqref{eq:EKI_update_cont}.
%		\item \stkout{Regarding the above issue, it is possible to consider a time-stepping scheme of \eqref{eq:EKI_cont} that is slightly different from \eqref{eq:EKI_update}, namely}
%		\begin{equation*}
%		\stkout{u_{i+1}^{(n)} = \uin + \alpha P_i^{uy} R^{-1} \big( \yin - h(\uin) \big), \quad \quad y_i^{(n)} \sim \Nc(y, \alpha^{-1} R).}
%		\end{equation*}
%		\stkout{However, as we found in our numerical experiments, this is also problematic when $R$ has a small scale, leading to a blow-up of the particles within the first few iterations.}
		\item Under the same setting of Proposition \ref{prop:EKI}, if we further require that the initial ensemble $\{ u_0^{(n)}\}_{n=1}^N$ is drawn independently from the prior distribution $\Nc(m, P)$, then in the mean-field limit $N \rightarrow \infty$ we have $\mf(0)=m$ and $\Cf(0)=P$, leading to
		\begin{equation*}
		\mf(1) = (P^{-1} + H^T R^{-1} H)^{-1} (P^{-1} m + H^T R^{-1} y), \quad \quad \Cf(1) = (P^{-1} + H^T R^{-1} H)^{-1},
		\end{equation*}
		which are the true posterior mean and covariance, respectively. However, we have observed in a variety of numerical examples (not reported here) that in nonlinear problems it is often necessary to run EKI up to times larger than $1$ to obtain adequate approximation of the posterior mean and covariance. Providing a suitable stopping criteria for EKI is a topic of current research \cite{SS18,IY20} beyond the scope of our work.  \label{rm: Bayesian_interpret}
		
%		\item \stkout{Under the setting of Proposition \ref{prop:EKI}, if $H^T R^{-1} H$ is invertible, then as $t \rightarrow \infty$ all of the ensemble members converge to a single point, which is the unique minimizer of the data misfit function $\Jdm$. This is a phenomenon referred to as `ensemble collapse'. If $H^T R^{-1} H$ is not invertible, then we can only say that the ensemble mean converges to one of the minimizers of $\Jdm$, with a potentially nontrivial covariance. Analysis of this convergence is out of the scope of this paper. In practice, early stoppings are often adopted, in view of the previous remark. }
%		%this is the case for some problems of interest, for example, when the observation space dimension is smaller than the input space dimension. 
		
%		\item An important drawback of Proposition \ref{prop:EKI} is the mean-field assumption $N \rightarrow \infty$. In practice, one of the main advantages of ensemble Kalman methods over alternative sampling or optimization methods is their ability to  perform well with a moderate ensemble size $N,$ often smaller than the dimension of the unknown parameter $u.$ In such cases the mean-field approximation may not be justified and may not capture the essence of the method. 
	\end{enumerate}
\end{remark}

\subsection{Ensemble Levenberg-Marquardt Optimization of Tikhonov-Phillips Objective}\label{ssec:TEKI}
\subsubsection{Tikhonov Ensemble Kalman Inversion}
Recall that we define
$$z := \begin{bmatrix}
y\\ m
\end{bmatrix},  \quad \quad \quad 
g(u) := \begin{bmatrix}
h(u)\\ u
\end{bmatrix},
\quad \quad \quad
Q := 
\begin{bmatrix}
R & 0 \\ 0 & P
\end{bmatrix}.
$$
Then, given an ensemble $\{u_i^{(n)} \}_{n=1}^N$, we can define $$g_i = \frac{1}{N} \sum_{n=1}^N g(u_i^{(n)}),$$ and empirical covariances
\begin{align*}
P_i^{zz} &= \frac{1}{N} \sum_{n=1}^N \bigl( g(u_i^{(n)}) - g_i  \bigr)   \bigl( g(u_i^{(n)}) - g_i  \bigr)^T,\\
P_i^{uz} &= \frac{1}{N} \sum_{n=1}^N \bigl( u_i^{(n)} - m_i  \bigr)   \bigl( g(u_i^{(n)}) - g_i  \bigr)^T.
\end{align*}
Furthermore, we define the statistical linearization $G_i$:
\begin{equation}
g'(\uin)
\approx
(P_i^{uz})^T (P_i^{uu})^{-1}
=: G_i.
\end{equation}
It is not hard to check that
\begin{equation*}
G_i = \begin{bmatrix}
H_i \\ I
\end{bmatrix},
\end{equation*}
with $H_i$ defined in \eqref{eq:stat_linearization}. 

Given an ensemble $\{u_i^{(n)} \}_{n=1}^N$, consider the following Levenberg-Marquardt update for each $n$:
\begin{equation*}
u_{i+1}^{(n)} = u_i^{(n)} + v_i^{(n)},
\end{equation*}
where $v_i^{(n)}$ is the minimizer of the following regularized (linearized) Tikhonov-Phillips objective  (cf. equation \eqref{eq:LMTP_obj}) 
\begin{equation}\label{eq:TEKI_new_C}
\Juctpn \hspace{-0.2cm}(v) = \frac12 \bigl| z_i^{(n)} - g(u_i^{(n)} )  - G_i v \bigr|^2_Q  + \frac{1}{2\alpha} \bigl|v\bigr|^2_{P_i^{uu}}, \quad \quad z_i^{(n)} \sim \mathcal{N}(z, \alpha^{-1} Q),
\end{equation}
and $\alpha>0$ will be regarded as a length-step. We can calculate the minimizer $v_i^{(n)}$ explicitly, applying Lemma \ref{lemma:linear}:
\begin{equation}\label{eq:TEKI_step1}
v_i^{(n)} = (G_i^T Q^{-1} G_i + \alpha^{-1} (P_i^{uu})^{-1})^{-1} G_i^T Q^{-1} \big( \zin - g(\uin) \big),
\end{equation}
or, in an equivalent form, 
\begin{equation}\label{eq:TEKI_step2}
v_i^{(n)} = P_i^{uu} G_i^T (G_i P_i^{uu} G_i^T + \alpha^{-1} Q)^{-1} \big( \zin - g(\uin) \big).
\end{equation}
Similar to EKI, in Equation \eqref{eq:TEKI_step2} we substitute $P_i^{uu} G_i^T = P_i^{uz}$, and make the approximation $G_i P_i^{uz} \approx P_i^{zz}$. This leads to Tikhonov Ensemble Kalman Inversion (TEKI), described in Algorithm \ref{algorithm:TEKI_step}.
\begin{algorithm}
	\caption{Tikhonov Ensemble Kalman Inversion (TEKI) \label{algorithm:TEKI_step}}
	\vspace{0.1in}
	\STATE {\bf Input}: Initial ensemble $\{ u_0^{(n)}\}_{n=1}^N$, length-step $\alpha$.  \\ 
	\STATE For $i = 0, 1, \ldots$ do:
	\begin{enumerate}
		\item Set $K_i = P_i^{uz} (P_{i}^{zz} + \alpha^{-1} Q)^{-1}. $
		\item Draw $z_i^{(n)} \sim \Nc(z, \alpha^{-1} Q)$ and set
		\begin{equation}\label{eq:TEKI_update}
		u_{i+1}^{(n)} = u_i^{(n)}+ K_i \Bigl\{ z_i^{(n)}  - g(u_i^{(n)})     \Bigr\} ,\quad \quad 1 \le n \le N. 
		\end{equation}
	\end{enumerate}
	\STATE{\bf Output}: Ensemble means $m_1, m_2, \ldots$
\end{algorithm}
\FloatBarrier
 We note that TEKI is a natural ensemble-based version of the derivative-based ILM-TP Algorithm \ref{ILMTP}. However,  the Kalman gain in ILM-TP keeps $P$ fixed, while in TEKI  $P_i^{uz}$ and $P_i^{zz}$ are updated using the ensemble.  

\subsubsection{Analysis of TEKI}
When $\alpha$ is small, we can rewrite the update \eqref{eq:TEKI_update} as a time-stepping scheme:
\begin{equation}\label{eq:TEKI_update_cont}
\begin{split}
u_{i+1}^{(n)} &= u_i^{(n)}+ \alpha P_i^{uz} (\alpha P_i^{zz} +Q)^{-1} \bigl( z + \alpha^{-1/2} Q^{1/2} \xiin - g(u_i^{(n)}) \bigr)\\
&= u_i^{(n)}+ \alpha P_i^{uz} (\alpha P_i^{zz} +Q)^{-1} \bigl( z - g(u_i^{(n)}) \bigr) + \alpha^{1/2}P_i^{uz}(\alpha P_i^{zz} +Q)^{-1} Q^{1/2} \xiin,
\end{split}
\end{equation}
where $\xiin \sim \Nc(0, I_{d+k})$ are independent. Taking the limit $\alpha \rightarrow 0$, we can interpret \eqref{eq:TEKI_update_cont} as a discretization of the SDE system
\begin{equation}\label{eq:TEKI_cont}
\diff u^{(n)} = P^{uz}(t) Q^{-1} \bigl( z - g(u^{(n)}) \bigr) \diff t + P^{uz}(t) Q^{-1/2} \diff W^{(n)}.
\end{equation}
If $h(u) = Hu$ is linear, $g(u) = Gu$ is also linear and the SDE system \eqref{eq:TEKI_cont} can be rewritten as
\begin{equation}\label{eq:TEKI_cont_linear}
\begin{split}
\diff u^{(n)} &= P^{uz}(t) Q^{-1} \big( z - Gu^{(n)} \big) \diff t + P^{uz}(t) Q^{-1/2} \diff W^{(n)} \\
&= P^{uu}(t) G^T Q^{-1} \big( z - Gu^{(n)} \big) \diff t + P^{uu}(t) G^T Q^{-1/2} \diff W^{(n)} . 
\end{split}
\end{equation}
\begin{proposition}
	For the SDE system \eqref{eq:TEKI_cont}, assume $h(u) = Hu$ is linear, and that the initial ensemble $\{ u_0^{(n)}\}_{n=1}^N$ is made of independent samples from a distribution with finite second moments. Then, in the large ensemble limit (mean-field),  the distribution of $ u^{(n)}(t) $ has mean $\mf(t)$ and covariance $\Cf(t)$, which satisfy:
	\begin{align}
	\frac{\diff \mf(t)}{\diff t} &= \Cf(t) G^T Q^{-1} (z - G \mf(t)), \label{eq:TEKI_mean_evolve}\\
	\frac{\diff \Cf(t)}{\diff t} &= -\Cf(t) G^T Q^{-1} G \Cf(t). \label{eq:TEKI_cov_evolve}
	\end{align}
	Furthermore, the solution can be computed analytically:
	\begin{align}
	\mf(t) &= \Big( \Cf(0)^{-1} + t G^T Q^{-1} G \Big)^{-1} \big( \Cf(0)^{-1} \mf(0) + t G^T Q^{-1} z \big), \label{eq:TEKI_mean_sol}\\
	\Cf(t) &= \Big( \Cf(0)^{-1} + t G^T Q^{-1} G \Big)^{-1}. \label{eq:TEKI_cov_sol}
	\end{align}
%	The second equation can be transformed to:
%	\begin{equation*}
%	\frac{\diff \Cf(t)^{-1}}{\diff t} = -G^T Q^{-1} G
%	\end{equation*}
	In particular, as $t \rightarrow \infty$, 
	\begin{align}
		\begin{split}
		\mf(t) &\rightarrow (G^T Q^{-1} G)^{-1} (G^T Q^{-1} z) \label{eq:TEKI_mean_limit}\\
			   &= (H^T R^{-1} H  + P)^{-1} (H^T R^{-1} y  + P ^{-1} m),  
		\end{split}\\
	\Cf(t) &\rightarrow 0. \label{eq:TEKI_cov_limit}
	\end{align}
	Notice that $\mf(t)$ converges to the true posterior mean.
\end{proposition}
\begin{proof}
Equations \eqref{eq:TEKI_mean_evolve}-\eqref{eq:TEKI_cov_sol} can be derived similarly as in the proof of Proposition \ref{prop:EKI}, replacing $H$ by $G$, $R$ by $Q$, and $y$ by $z$. Now since $G^T Q^{-1} G = H^T R^{-1} H + P$ is always invertible we have that,  as $t \rightarrow \infty,$
$$\Cf(t) = t^{-1} \big( t^{-1} \Cf(0)^{-1} + G^T Q^{-1} G \big)^{-1} \rightarrow  0$$
by continuity of the matrix inverse function. The limit of $\mf(t)$ in \eqref{eq:TEKI_mean_limit} follows directly from \eqref{eq:TEKI_mean_sol}.
\end{proof}
%\begin{remark}
%	Some remarks:
%	\begin{enumerate}
%		\item
%	\end{enumerate}
%\end{remark}

%%%%%%%%%%%%%%%%%%%%%%%%%%%%%%%SECTION 4%%%%%%%%%%%%%%%%%%%%%%%%%%%%%%%%%%%%%%%%%%%%%%%%%%%%%%%%%%%%%%

\section{Ensemble Kalman methods: New Variants}
\label{sec:EnsembleKalmanLearningNew}
In the previous section we discussed three popular subfamilies of iterative ensemble Kalman methods, analogous to the derivative-based algorithms in Section \ref{sec:background}. The aim of this section is to introduce two new iterative ensemble Kalman methods which are inspired by the SDE continuum limit structure of the algorithms in Section \ref{sec:EnsembleKalmanLearning}. The two new methods that we introduce have in common that they rely on statistical linearization, and that the long-time limit of the ensemble covariance recovers the posterior covariance in a linear setting. Subsection \ref{sec:IEKS} contains a new variant of IEKF which in addition to recovering the posterior covariance, it also recovers the posterior mean in the long-time limit. Subsection \ref{sec:EKIN} introduces a new variant of the EKI method. Finally, Subsection \ref{ssec:gradientstructure} highlights the gradient structure of the algorithms in this and the previous section, shows that our new variant of IEKF can also be interpreted as a modified TEKI algorithm, and sets our proposed new methods into the broader literature. 

\subsection{Iterative Ensemble Kalman Filter with Statistical Linearization}\label{sec:IEKS}
In some of the literature on iterative ensemble Kalman methods \cite{RRZL06, SU12}, the length-step ($\alpha$ in Algorithms \ref{itenKF} and \ref{itenKFalt}) is set to be 1. Although this choice of length-step allows to recover the true posterior mean in the linear case after one iteration (see Section \ref{sec:analysis_IEKF}), it leads to numerical instability in complex nonlinear models. Alternative ways to set $\alpha$ include performing a line-search that satisfies Wolfe's condition, or using other ad-hoc line-search criteria \cite{GO07}. These methods allow $\alpha$ to be adaptively chosen throughout the iterations, but they introduce other hyperparameters that need to be selected manually.

Our idea here is to slightly modify Algorithm \ref{itenKF} so that in the linear case its continuum limit has the true posterior as its invariant distribution. In this way we can simply choose a small enough $\alpha$ and run the algorithm until the iterates reach a statistical equilibrium, avoiding the need to specify suitable hyperparameters and stopping criteria. Our empirical results show that this approach also performs well in the nonlinear case.

In the update Equation \eqref{eq:IEKF_update}, we replace each of the $u_0^{(n)}$ by a perturbation of the prior mean $m$, and we replace $P_0^{uu}$ by the prior covariance matrix $P$ in the definition of $K_i$. Details can be found below in our modified algorithm, which we call IEKF-SL.
\begin{algorithm}
	\caption{Iterative Ensemble Kalman Filter with Statistical Linearization (IEKF-SL) \label{algorithm:IEKF_new}}
	\vspace{0.1in}
	\STATE {\bf Input}: Initial ensemble $\{ u_0^{(n)}\}_{n=1}^N$, step size $\alpha$.  \\ 
	\STATE For $i = 0, 1, \ldots$ do:
	\begin{enumerate}
		\item Set $K_i = P H_i^T (H_i P H_i^T + R)^{-1}, \quad \quad H_i = (P_i^{uy})^T (P_i^{uu})^{-1}$.
		\item Draw $y_i^{(n)} \sim \Nc(y, 2\alpha^{-1} R)$, $m_i^{(n)} \sim \Nc(m, 2\alpha^{-1} P)$ and set
		\begin{equation}\label{eq:IEKF_new_update}
			u_{i+1}^{(n)} = u_i^{(n)}+ \alpha \Bigl\{ K_i \big( y_i^{(n)}  - h(u_i^{(n)}) \big) + (I - K_i H_i)  \big( m_i^{(n)} - \uin \big)   \Bigr\} ,\quad \quad 1 \le n \le N. 
		\end{equation}
	\end{enumerate}
	\STATE{\bf Output}: Ensemble means $m_1, m_2, \ldots$
\end{algorithm}
\FloatBarrier

It is natural to regard the update \eqref{eq:IEKF_new_update} as a time-stepping scheme. We rewrite it in an alternative form, in analogy to \eqref{eq:IEKF_step1}:
\begin{equation}\label{eq:IEKF_new_update_cont}
	\begin{split}
		u_{i+1}^{(n)} &= \uin + \alpha C_i \Bigl\{ H_i^T R^{-1} \big( \yin - h(\uin) \big) + P^{-1} \big( m_i^{(n)} - \uin \big) \Bigr\} \\
		&= \uin + \alpha C_i \Bigl\{ H_i^T R^{-1} \big( y - h(\uin) \big) + P^{-1} \big( m - \uin \big) + (H_i^T R^{-1} \zeta + P^{-1}\eta) \Bigr\},
	\end{split}
\end{equation}
where 
\begin{equation*}
	C_i = \big( H_i^T R^{-1} H_i + P^{-1} \big)^{-1}, \quad \zeta \sim \Nc(0, 2\alpha^{-1} R), \quad \eta \sim \Nc(0, 2\alpha^{-1} P).
\end{equation*}
We interpret Equation \eqref{eq:IEKF_new_update_cont}  as a discretization of the SDE system
\begin{equation}\label{eq:IEKF_new_cont}
	\diff u^{(n)} = C(t) \Big( H(t)^T R^{-1} \big( y - h(u^{(n)}) \big) + P^{-1} \big( m - u^{(n)} \big) \Big) \diff t + \sqrt{2 C(t)} \diff W^{(n)},
\end{equation}
where
\begin{align*}
	H(t) &= \big( P^{uy}(t) \big)^T \big( P^{uu}(t) \big)^{-1}, \\
	C(t) &= \big(  H(t)^T R^{-1} H(t) + P^{-1} \big)^{-1}.
\end{align*}
The diffusion term can be derived using the fact that
\begin{equation*}
	C_i(H_i^T R^{-1} \zeta + P^{-1}\eta) \sim \Nc\big(0, 2\alpha^{-1} C_i (H_i^T R^{-1} H_i + P^{-1}) C_i \big)  = \Nc (0, 2\alpha^{-1} C_i).
\end{equation*}

If $h(u)=Hu$ is linear and the empirical covariance $P^{uu}(t)$ has full rank for all $t$, then $H(t) \equiv H$,  $C(t) \equiv C = (P^{-1} + H^T R^{-1} H)^{-1}$, and the SDE system can be decoupled and further simplified:
\begin{equation}\label{eq:IEKF_new_cont_linear}
	\begin{split}
		\diff u^{(n)} &= C \Big( H^T R^{-1} \big( y - Hu^{(n)} \big) + P^{-1} (m - u^{(n)}) \Big)\diff t + \sqrt{2 C} \diff W\\
		&= \Big( -u^{(n)} + C(H^T R^{-1} y + P^{-1}m) \Big) \diff t + \sqrt{2 C} \diff W .
	\end{split}
\end{equation}
\begin{proposition}
	For the SDE system \eqref{eq:IEKF_new_cont}, assume $h(u)=Hu$ is linear and $H(t) \equiv H$ holds. Assume that the initial ensemble $\{ u_0^{(n)}\}_{n=1}^N$ is made of independent samples from a continuous distribution with finite second moments. Then, for $1\le n \le N$, the mean $\mf(t)$ and covariance $\Cf(t)$  of $u^{(n)}(t)$ satisfy
	\begin{align}
		\frac{\diff \mf(t)}{\diff t} &= -\mf(t) + C ( H^T R^{-1} y + P^{-1} m), \label{eq:IEKF_mean_evolve}\\
		\frac{\diff \Cf(t)}{\diff t} &= - 2 \Cf(t) + 2 C. \label{eq:IEKF_cov_evolve}
	\end{align}
	Furthermore, as $t\rightarrow \infty$,
	\begin{align}
		\begin{split}
			\mf(t) &\rightarrow C (H^T R^{-1} y + P^{-1} m) \label{eq:IEKF_mean_limit}\\
			&= (P^{-1} + H^T R^{-1} H)^{-1} (H^T R^{-1} y + P^{-1} m),  
		\end{split}\\
		\Cf(t) &\rightarrow C = (P^{-1} + H^T R^{-1} H)^{-1}. \label{eq:IEKF_cov_limit}
	\end{align}
	In other words, $\mf(t)$ and $\Cf(t)$ converge to the true posterior mean and covariance, respectively.
\end{proposition}
\begin{proof}
	It is clear that, for fixed $t>0,$ $\{u^{(n)}(t)\}_{n=1}^N$ are independent and identically distributed. The evolution  of the mean follows directly from \eqref{eq:IEKF_new_cont_linear}, and the evolution of the covariance follows from \eqref{eq:IEKF_new_cont_linear} by applying It\^{o}'s formula. It is then straightforward to derive \eqref{eq:IEKF_mean_limit} and \eqref{eq:IEKF_cov_limit} from \eqref{eq:IEKF_mean_evolve} and \eqref{eq:IEKF_cov_evolve}, respectively. 
\end{proof}
%\begin{remark}
%Some remarks:
%\begin{enumerate}
%\item
%\end{enumerate}
%\end{remark}

\subsection{Ensemble Kalman Inversion with Statistical Linearization}\label{sec:EKIN}

Recall that in the formulation of EKI, we define a regularized (linearized) data-misfit objective \eqref{eq:EKI_new_C}, where we have a regularizer on $v$ with respect to the norm $|\cdot|_{P_i^{uu}}$. However, in view of Proposition \ref{prop:EKI}, under a linear forward model $h(u) = Hu$, the particles $\{u_i^{(n)}\}_{i=1}^n$ will `collapse', meaning that the empirical covariance of $\{H u_i^{(n)}\}_{n=1}^N$ will vanish in the large $i$ limit. One possible solution to this is `covariance inflation', namely to inject certain amount of random noise after each ensemble update. However, this requires ad-hoc tuning of additional hyperparameters. An alternative approach to avoid the ensemble collapse is to modify the regularization term in the Levenberg-Marquardt formulation \eqref{eq:EKI_new_C}. The rough idea is to consider another regularizer on $H_i v$ in the data space, as we describe in what follows. 

We define a new regularized data-misfit objective, slightly different from \eqref{eq:EKI_new_C}:
\begin{equation}\label{eq:EKI_newer_C}
	\Jucdmn \hspace{-0.2cm}(v) = \frac12 \bigl| y_i^{(n)} - h(u_i^{(n)} )  - H_i v \bigr|^2_R  + \frac{1}{2\alpha} \bigl|v\bigr|^2_{C_i} \quad \quad y_i^{(n)} \sim \mathcal{N}(y, 2 \alpha^{-1} R),
\end{equation}
where
\begin{equation}\label{eq:EKIN-Si}
	C_i = (P^{-1} + H_i^T R^{-1} H_i)^{-1},
\end{equation}
and $H_i$ is defined in \eqref{eq:stat_linearization}. The regularization term can be decomposed as
\begin{equation*}
	|v|_{C_i}^2 = |v|_P^2 + |H_i v|_R^2.
\end{equation*}
The first term can be regarded as a regularization on $v$ with respect to the prior covariance $P$. The second term can be regarded as a regularization on $H_i v$ with respect to the noise covariance $R$. %a regularization on $v$ with respect to Hessian $H_i^T R^{-1} H_i$ of the linearized data misfit objective $$\frac12 \bigl| y_i^{(n)} - h(u_i^{(n)} )  - H_i v \bigr|^2_R.$$
Applying Lemma \ref{lemma:linear}, we can calculate the minimizer of \eqref{eq:EKI_newer_C}:
\begin{equation}\label{eq:EKI_new_update_cont}
	\begin{split}
		\vin &= (H_i^T R^{-1} H_i + \alpha^{-1} C_i^{-1})^{-1} H_i^T R^{-1} \big( \yin - h(\uin) \big)\\
		&= \big( \alpha^{-1} P^{-1} + (1 + \alpha^{-1}) H_i^T R^{-1} H_i \big)^{-1} H_i^T R^{-1} \big( y_i^{(n)} - h(u_i^{(n)}) \big)\\
		&= \alpha P H_i^T \big( (1 + \alpha) H_i P H_i^T + R \big)^{-1} \big( y_i^{(n)} - h(u_i^{(n)}) \big),
	\end{split}
\end{equation}
where the second equality follows from the definition of $C_i$, and the third equality follows from the matrix identity \eqref{eq:kalmangain}. This leads to Algorithm \ref{algorithm:EKI_new}.
\begin{algorithm}
	\caption{Ensemble Kalman Inversion with Statistical Linearization (EKI-SL) \label{algorithm:EKI_new}}
	\vspace{0.1in}
	\STATE {\bf Input}: Initial ensemble $\{ u_0^{(n)}\}_{n=1}^N$, step size $\alpha$.  \\ 
	\STATE For $i = 0, 1, \ldots$ do:
	\begin{enumerate}
		\item Set $K_i = \alpha P H_i^T \big( (1 + \alpha) H_i P H_i^T + R \big)^{-1}, \quad \quad H_i = (P_i^{uy})^T (P_i^{uu})^{-1}$.
		\item Draw $y_i^{(n)} \sim \Nc(y, 2\alpha^{-1} R)$ and set
		\begin{equation}\label{eq:EKI_new_update}
			u_{i+1}^{(n)} = u_i^{(n)}+ K_i \Bigl\{ y_i^{(n)}  - h(u_i^{(n)})     \Bigr\} ,\quad \quad 1 \le n \le N. 
		\end{equation}
	\end{enumerate}
	\STATE{\bf Output}: Ensemble means $m_1, m_2, \ldots$
\end{algorithm}
\FloatBarrier
For small $\alpha > 0$, we interpret the update \eqref{eq:EKI_new_update} as a discretization of the coupled SDE system
\begin{equation}\label{eq:EKI_new_cont}
	\begin{split}
		\diff u^{(n)} &= P H(t)^T \big( H(t) P H(t)^T + R \big)^{-1} \Big( \big( y - h(u^{(n)}) \big) \diff t + \sqrt{2} R^{1/2} \diff W \Big) \\
		&= C(t) H(t)^T R^{-1} \big( y - h(u^{(n)}) \big) \diff t + \sqrt{2} C(t) H(t)^T R^{-1/2} \diff W,
	\end{split}
\end{equation}
where
\begin{align*}
	H(t) &= \big( P^{uy}(t) \big)^T \big( P^{uu}(t) \big)^{-1},\\
	C(t) &= \big( P^{-1} + H(t)^T R^{-1} H(t) \big)^{-1}.
\end{align*}
For our next result we will work under the assumption that $H(t) \equiv H$ is constant, which in particular requires that 
the empirical covariance $P^{uu}(t)$ has full rank for all $t.$
Importantly, under this assumption $C(t) \equiv C = (P^{-1} + H^T R^{-1} H)^{-1}$and the SDE system is decoupled:
\begin{equation}\label{eq:EKI_new_cont_linear}
	\diff u^{(n)} = C H^T R^{-1} \big( y - Hu^{(n)} \big) \diff t + \sqrt{2} C H^T R^{-1/2} \diff W^{(n)}.
\end{equation}
\begin{proposition}\label{prop:EKIN}
	For the SDE system \eqref{eq:EKI_new_cont}, assume $h(u)=Hu$ is linear and $H(t) \equiv H$ holds. Assume that the initial ensemble $\{ u_0^{(n)}\}_{n=1}^N$ is made of independent samples from a continuous distribution with finite second moments. Then   the mean $\mf(t)$ and covariance $\Cf(t)$ of  $u^{(n)}(t)$ satisfy
	\begin{align}
		\frac{\diff \mf(t)}{\diff t} &= C H^T R^{-1} (y - H \mf(t)), \label{eq:EKIN_mean_evolve} \\
		\frac{\diff \Cf(t)}{\diff t} &= - C H^T R^{-1} H \Cf(t) - \Cf(t) H^T R^{-1} H C + 2 C H^T R^{-1} H C. \label{eq:EKIN_cov_evolve}
	\end{align}
	 In particular, if $H \in \R^{k \times d}$ has full column rank (i.e., $d\le k$, $\emph{rank}(H) = d$), then, as $t \rightarrow \infty$, 
	\begin{align}
		\mf(t) &\rightarrow (H^T R^{-1} H)^{-1} (H^T R^{-1} y), \label{eq:EKIN_mean_limit}\\
		\Cf(t) &\rightarrow C = (P^{-1} + H^T R^{-1} H)^{-1}. \label{eq:EKIN_cov_limit}
	\end{align}
	If $H \in \R^{k \times d}$ has full row rank (i.e., $d\ge k$, $\emph{rank}(H) = k$), then, as $t \rightarrow \infty,$
	\begin{align}
		H \mf(t) &\rightarrow y, \label{eq:EKIN_mean_limit2}\\
		H \Cf(t) H^T &\rightarrow H C H^T. \label{eq:EKIN_cov_limit2}
	\end{align}
\end{proposition}
\begin{proof} Note that here, in contrast to the setting of Proposition \ref{prop:EKI}, the distribution of $u^{(n)}(t)$ does not depend on the ensemble size $N,$ and we have simply $\mf(t) = \Expect \bigl[u^{(n)}(t) \bigr]$ and $\Cf(t) = \Expect \big[ e^{(n)}(t) \otimes e^{(n)}(t) \big], $ with $e^{(n)}(t) := u^{(n)}(t) - \mf(t).$ The evolution of $\mf(t)$ follows directly from \eqref{eq:EKI_new_cont_linear}. To obtain the evolution of $\Cf(t)$, we use a similar technique as in the proof of Proposition \ref{prop:EKI}. Applying It\^{o}'s formula, 
	\begin{equation*}
		\begin{split}
			\frac{\diff \Cf(t)}{\diff t} &= \Expect \Big[ -C H^T R^{-1} H (e^{(n)} \otimes e^{(n)}) - (e^{(n)} \otimes e^{(n)}) H^T R^{-1} H C + 2C H^T R^{-1} H C \Big]\\
			&= -C H^T R^{-1} H \Cf(t) - \Cf(t) H^T R^{-1} H C + 2 C H^T R^{-1} H C,
		\end{split}
	\end{equation*}
	which recovers \eqref{eq:EKIN_cov_evolve}. 
	
	 If $H$ has full column rank, then $H^T R^{-1} H$ is invertible, and \eqref{eq:EKIN_mean_limit} follows immediately from \eqref{eq:EKIN_mean_evolve}. By setting the right-hand side of \eqref{eq:EKIN_cov_evolve} to 0, and using the fact that it has a unique solution, we derive \eqref{eq:EKIN_cov_limit}. 
	 
 If $H$ has full row rank, then \eqref{eq:EKIN_mean_limit2} follows immediately from \eqref{eq:EKIN_mean_evolve}. Next, the substitutions
	\begin{equation*}
		\widetilde{\Cf}(t) = R^{-1/2} H \Cf(t) H^T R^{-1/2}, \quad\quad\quad \widetilde{C} = R^{-1/2} H C H^T R^{-1/2}
	\end{equation*}
	allow to transform \eqref{eq:EKIN_cov_evolve}  into
	\begin{equation}
		\frac{\diff \widetilde{\Cf}(t)}{\diff t} = - \widetilde{C}\widetilde{\Cf}(t) - \widetilde{\Cf}(t)\widetilde{C} + 2\widetilde{C}^2. \label{eq:EKIN_tr_cov_evolve}
	\end{equation}
	Since $\widetilde{C}$ is invertible, by setting the right-hand side of \eqref{eq:EKIN_tr_cov_evolve} to 0 we deduce that $\widetilde{\Cf}(t) \rightarrow \widetilde{C}$ as $t \rightarrow \infty$, which recovers \eqref{eq:EKIN_cov_limit2}.
\end{proof}

\subsection{Gradient Structure and Discussion}\label{ssec:gradientstructure}
\paragraph{LM Algorithms in the Continuum Limit}
Levenberg-Marquardt algorithms have a natural gradient structure in the continuum limit. This was shown in Equation \eqref{eq:lm_cont_limit}, where the preconditioner $P$ corresponds to the regularizer $|\cdot|_P$ that is used in the Levenberg-Marquardt algorithm \eqref{eq:LMobjectivewithlengthstep}. Ensemble-based Levenberg-Marquardt algorithms also possess a gradient structure. To see this,  we consider an update $u_{i+1}^{(n)} = \uin + \vin$, where $\vin$ is the unconstrained minimizer of the same objective as \eqref{eq:EKI_newer_C}
\begin{equation*}
	\Jucdmn \hspace{-0.2cm}(v) = \frac12 \bigl| y_i^{(n)} - h(u_i^{(n)} )  - H_i v \bigr|^2_R  + \frac{1}{2\alpha} \bigl|v\bigr|^2_{S_i} \quad \quad y_i^{(n)} \sim \mathcal{N}(y, 2 \alpha^{-1} R),
\end{equation*}
except that we allow any positive (semi)definite matrix $S_i$ to act as a `regularizer'. The resulting continuum limit is given by the SDE system
\begin{equation}\label{eq:EKIN_cont_general}
	\diff u^{(n)} = S(t) H(t)^T R^{-1} \big( y - h(u^{(n)}) \big) \diff t + \sqrt{2} S(t) H(t)^T R^{-1/2} \diff W^{(n)} ,
\end{equation}
where $u^{(n)}(t), S(t), H(t)$ are continuous time analogs of $u_i^{(n)}, S_i, H_i$. Notice that $H(t)^T R^{-1} (y - h(u^{(n)}))$ is exactly $-\Jdm'(u^{(n)})$, so that we may rewrite \eqref{eq:EKIN_cont_general} as
\begin{equation}\label{eq:LM_cont}
	\dot{u}^{(n)} = - S(t) \Jdm'(u^{(n)}) + \sqrt{2} S(t) H(t)^T R^{-1/2} \dot{W}^{(n)} ,
\end{equation}
which is a perturbed gradient descent with preconditioner $S(t)$. We recall that $S(t) = P^{uu}(t)$ in EKI and $S(t) = C(t)= \big( P^{-1} + H(t)^T R^{-1} H(t) \big)^{-1}$ in EKI-SL. Other choices of $S(t)$ are possible and will be studied in future work.

\paragraph{Gauss-Newton Algorithms in the Continuum Limit}
Gauss-Newton algorithms also have a natural gradient structure in the continuum limit. As we shown in Equation \eqref{gn_cont_limit}, the Gauss-Newton method applied to a Tikhonov-Phillips objective can be regarded as a gradient flow with a preconditioner that is the inverse Hessian matrix of the objective. Ensemble-based Gauss-Newton methods also possess a similar gradient structure. Recall that in Equation \eqref{eq:IEKF_new_cont} we formulate the continuum limit of IEKF-SL:
\begin{equation}\label{eq:IEKF_new_cont2}
	\diff u^{(n)} = C(t) \Big( H(t)^T R^{-1} \big( y - h(u^{(n)}) \big) + P^{-1} \big( m - u^{(n)} \big) \Big) \diff t + \sqrt{2 C(t)} \diff W^{(n)}.
\end{equation}
Notice that the drift term is exactly $-C(t)\Jtp'(u^{(n)})$, so we may rewrite \eqref{eq:IEKF_new_cont2} as
\begin{equation}\label{eq:GN_cont}
	\dot{u}^{(n)} = - C(t)\Jtp'(u^{(n)}) + \sqrt{2 C(t)} \dot{W}^{(n)},
\end{equation}
which is a perturbed gradient descent with preconditioner $C(t)$, the inverse Hessian of $\Jtp$.

\paragraph{A Unified View of Levenberg-Marquardt and Gauss-Newton Algorithms in the Continuum Limit}
As a conclusion of above discussion, in the continuum limit Levenberg-Marquardt algorithms (e.g., EKI, EKI-SL) are (perturbed) gradient descent methods, with a preconditioner determined by the regularizer used in the Levenberg-Marquardt update step. Gauss-Newton algorithms (e.g., IEKF, IEKF-SL) are also (perturbed) gradient descent, with a preconditioner determined by the inverse Hessian of the objective. 

Interestingly, there are cases when the two types of algorithms coincide, even with the same amount of perturbation in the gradient descent step. A way to see this is to set the Levenberg-Marquardt regularizer to be the inverse Hessian of the objective. In equation \eqref{eq:LM_cont}, if we consider the Tikhonov-Phillips objective (i.e., replace $\Jdm$ by $\Jtp$, $H$ by $G$, $Q$ by $R$),  and set $S(t) = C(t)$, then \eqref{eq:LM_cont} and \eqref{eq:GN_cont} will coincide, leading to the same SDE system. This is the reason why we do not introduce a `TEKI-SL' algorithm, as it is identical to IEKF-SL.

\paragraph{Relationship to Ensemble Kalman Sampler (EKS)}
Another algorithm of interest is the Ensemble Kalman Sampler (EKS) \cite{GHLS19, DL19}. Although not discussed in this work, the EKS update is similar to \eqref{eq:IEKF_new_cont2}. In fact, if we replace $C(t)$ by the ensemble covariance $P^{uu}(t)$, and use the fact that $P^{uu}(t)H(t)^T=P^{uy}(t)$ by definition, we immediately recover the evolution equation of EKS:
\begin{equation*}
\diff u^{(n)} = \Big( P^{uy}(t)^T R^{-1} \big( y - h(u^{(n)}) \big)  + P^{-1} \big( m - u^{(n)} \big) \Big) \diff t + \sqrt{2 P^{uu}(t)} \diff W^{(n)}.
\end{equation*}
Then an Euler-Maruyama discretization is performed to compute the update formula in discrete form. However, as we find out in several numerical experiments, the length-step $\alpha$ needs to be chosen carefully. If the noise $R$ has a small scale, EKS often blows up in the first few iterations, while other algorithms with the same length-step do not. Also, if we consider a linear forward model $h(u) = Hu$, EKS still requires a mean field assumption $N \rightarrow \infty$ in order for it to converge to the posterior distribution, while IEKF-SL approximately needs $N\approx d$ particles where $d$ is the input dimension.

\section{Numerical Examples}
\label{sec:num}
In this section we provide a numerical comparison of the iterative ensemble Kalman methods introduced in Sections \ref{sec:EnsembleKalmanLearning} and \ref{sec:EnsembleKalmanLearningNew}. Our experiments highlight the variety of applications that have motivated the development of these methods, and illustrate their use in a wide range of settings. 

\subsection{Elliptic Boundary Value Problem}\label{ex:elliptic2d}
In this subsection we consider a simple nonlinear Bayesian inverse problem originally presented in \cite{ESS15}, where the unknown parameter is two dimensional and the forward map admits a closed formula. These features facilitate both the visualization and the interpretation of the solution.  
\subsubsection{Problem Setup}
Consider the elliptic boundary value problem
\begin{align}\label{eq:elliptic2d_forward}
\begin{split}
\frac{\diff}{\diff x} \bigg(\exp (u_1) \frac{\diff}{\diff x} p(x) \bigg) &= 1, \quad x \in (0,1), \\
p(0) &= 0,  \quad \quad p(1) = u_2.
\end{split}
\end{align}
  It can be checked that \eqref{eq:elliptic2d_forward} has an explicit solution
\begin{equation}\label{eq:elliptic2d_exact}
	p_u(x) = u_2 x -\frac{1}{2} \exp(-u_1)(x^2 - x).
\end{equation}

\begin{figure}[htbp]
	\centering
	\begin{subfigure}{.32\textwidth}
		\centering
		\includegraphics[width=5cm]{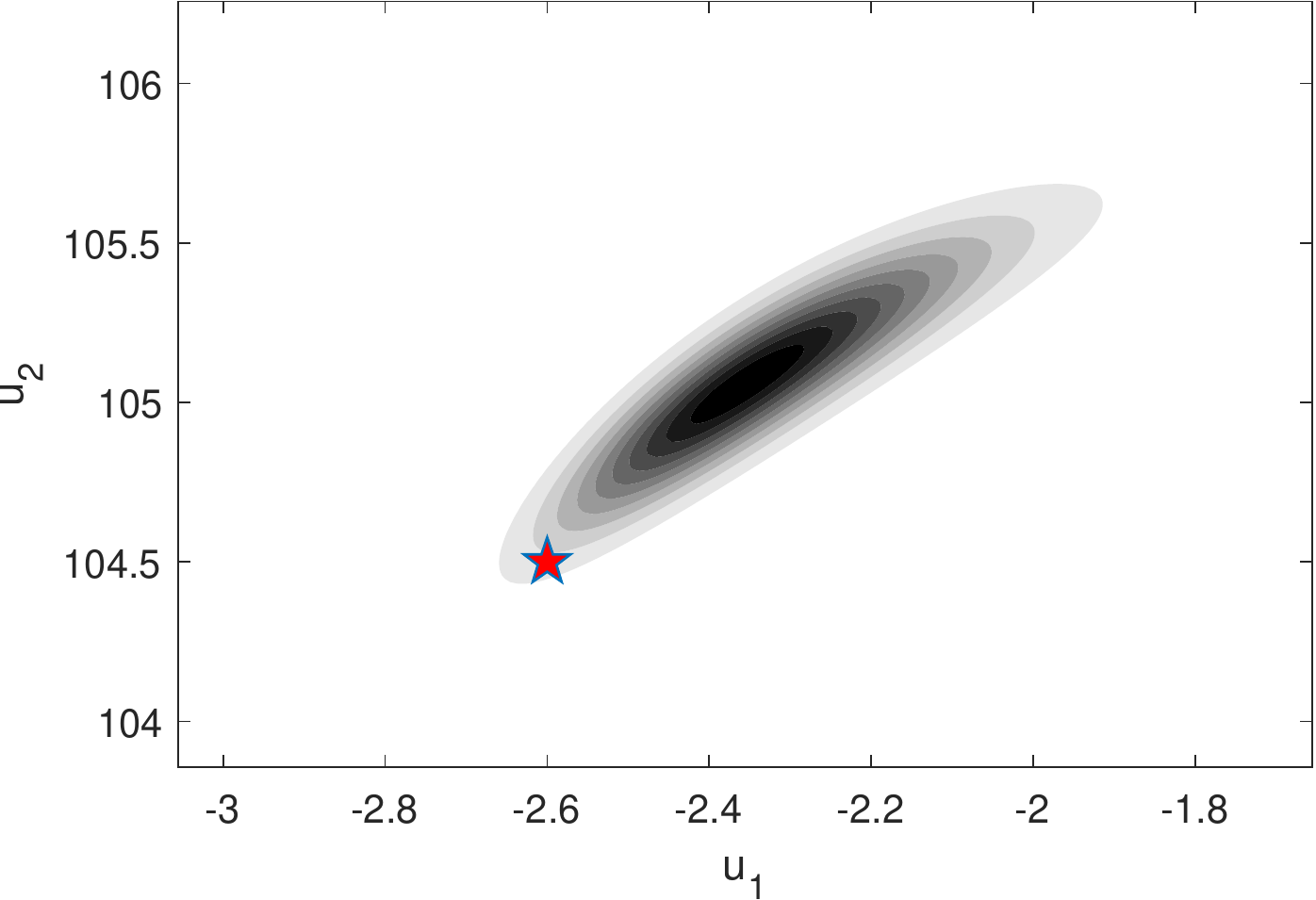}
		\caption{Truth $u^\dagger$.}
		 	%	\label{fig:i1}
	\end{subfigure}%
	%	\hfill
	\begin{subfigure}{.32\textwidth}
		\centering
		\includegraphics[width=5cm]{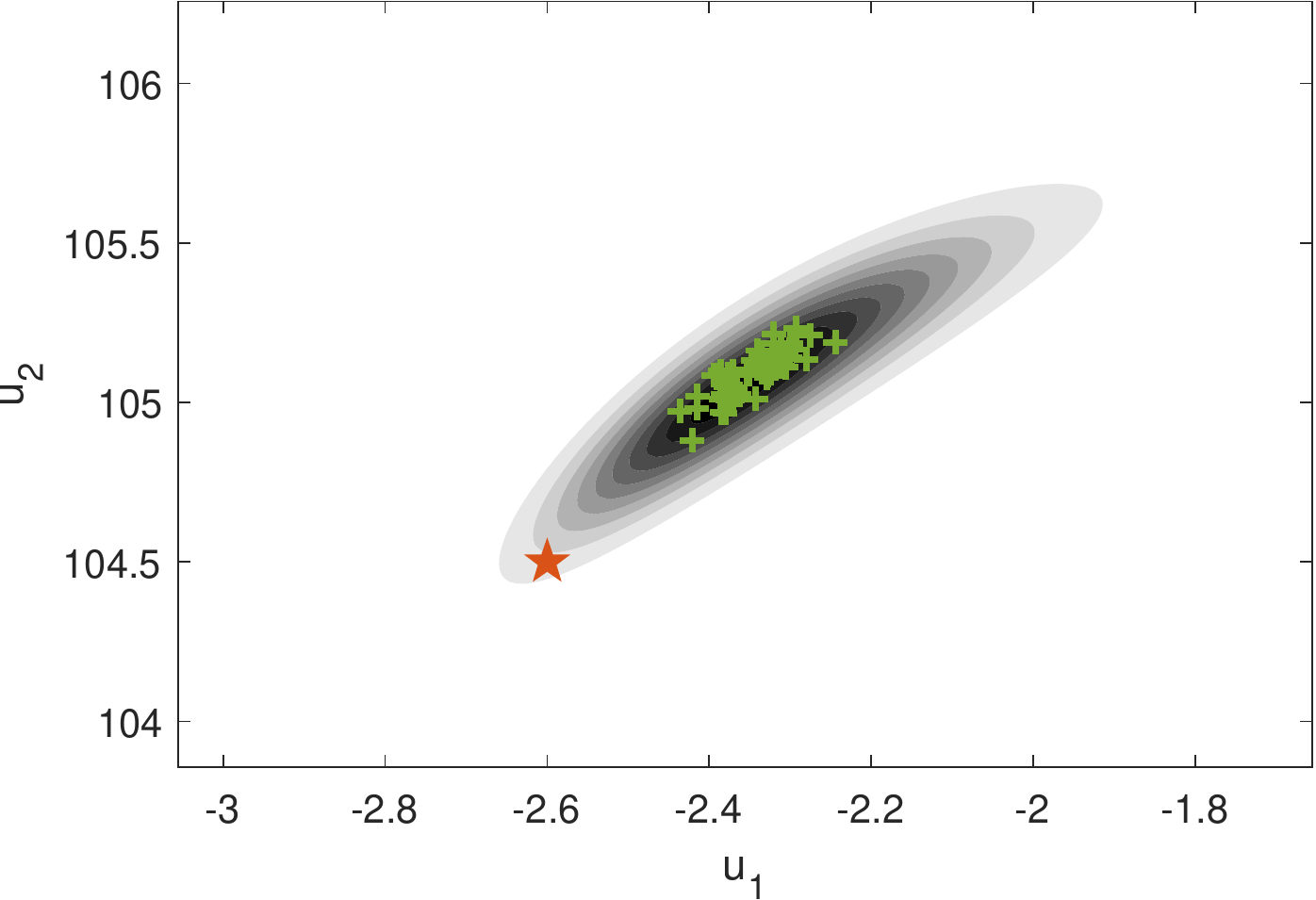}
		\caption{EKI.}
		% 			\label{fig:i1}
	\end{subfigure}%
	%	\hfill
	\begin{subfigure}{.32\textwidth}
		\centering
		\includegraphics[width=5cm]{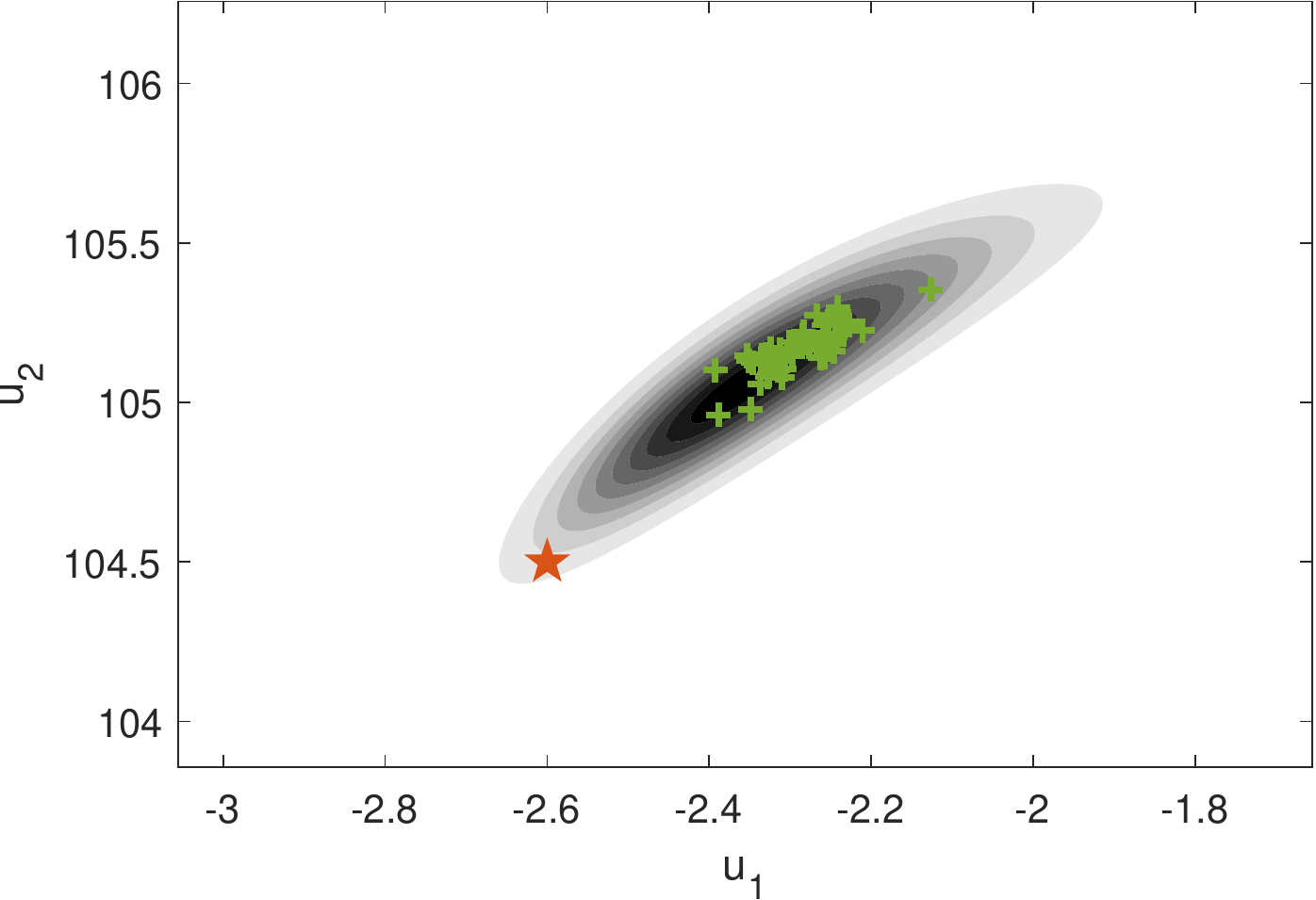}
		\caption{TEKI.}
		% 			\label{fig:i1}
	\end{subfigure}%
	\vskip\baselineskip
	\begin{subfigure}{.32\textwidth}
		\centering
		\includegraphics[width=5cm]{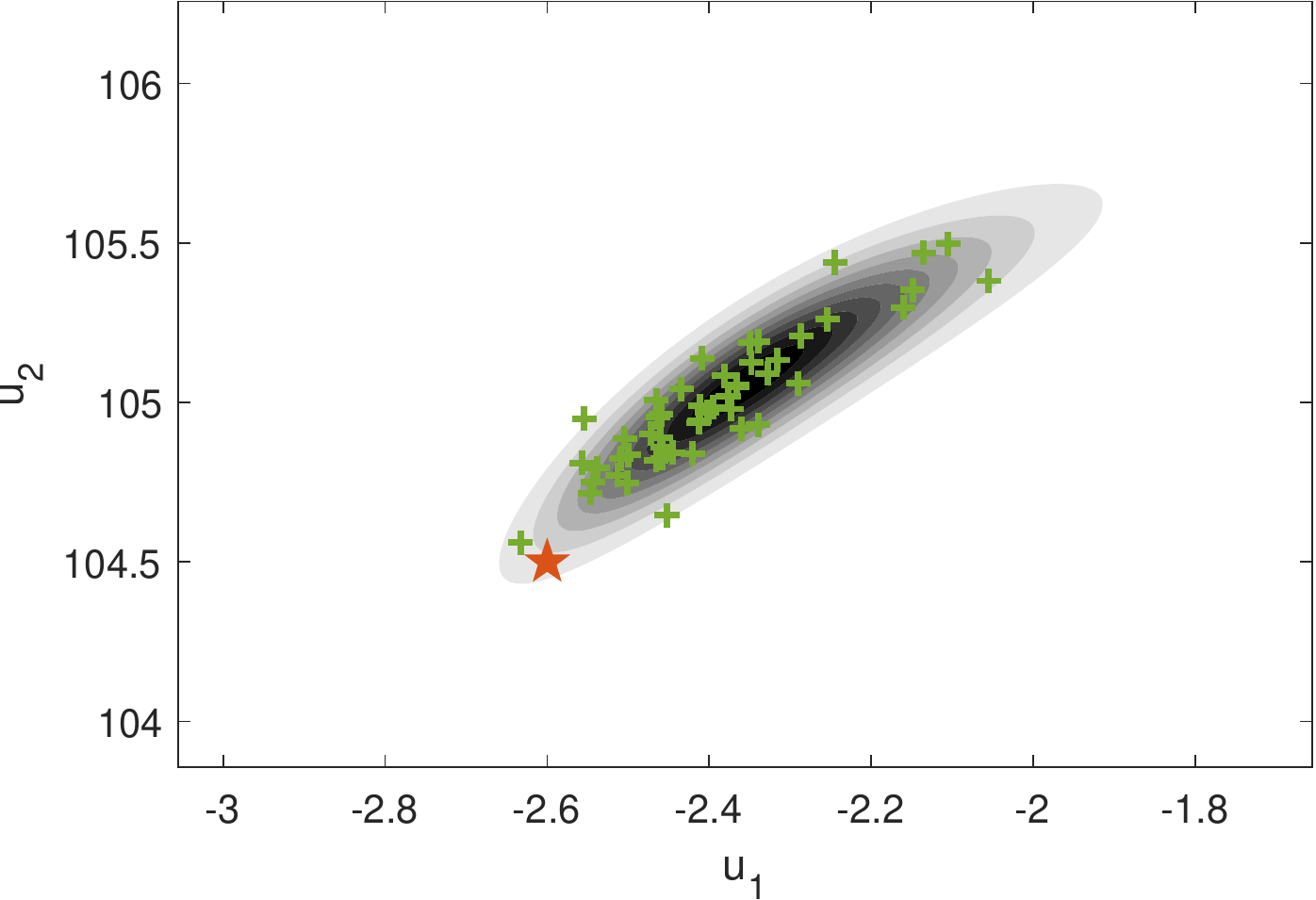}
		\caption{IEKF.}
		% 			\label{fig:i1}
	\end{subfigure}%
	%	\hfill
	\begin{subfigure}{.32\textwidth}
		\centering
		\includegraphics[width=5cm]{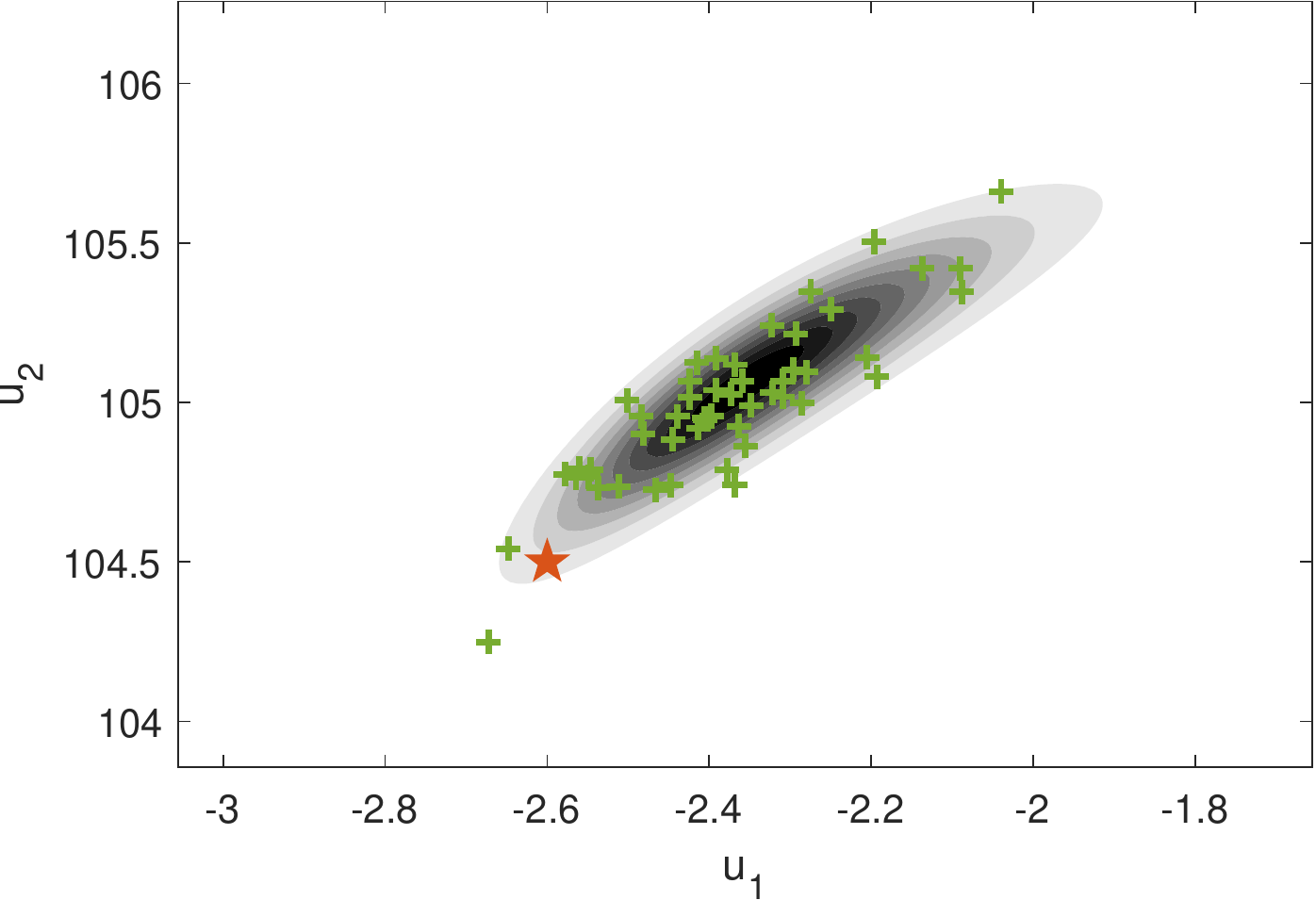}
		\caption{EKI-SL.}
		% 			\label{fig:i1}
	\end{subfigure}%
	\begin{subfigure}{.32\textwidth}
		\centering
		\includegraphics[width=5cm]{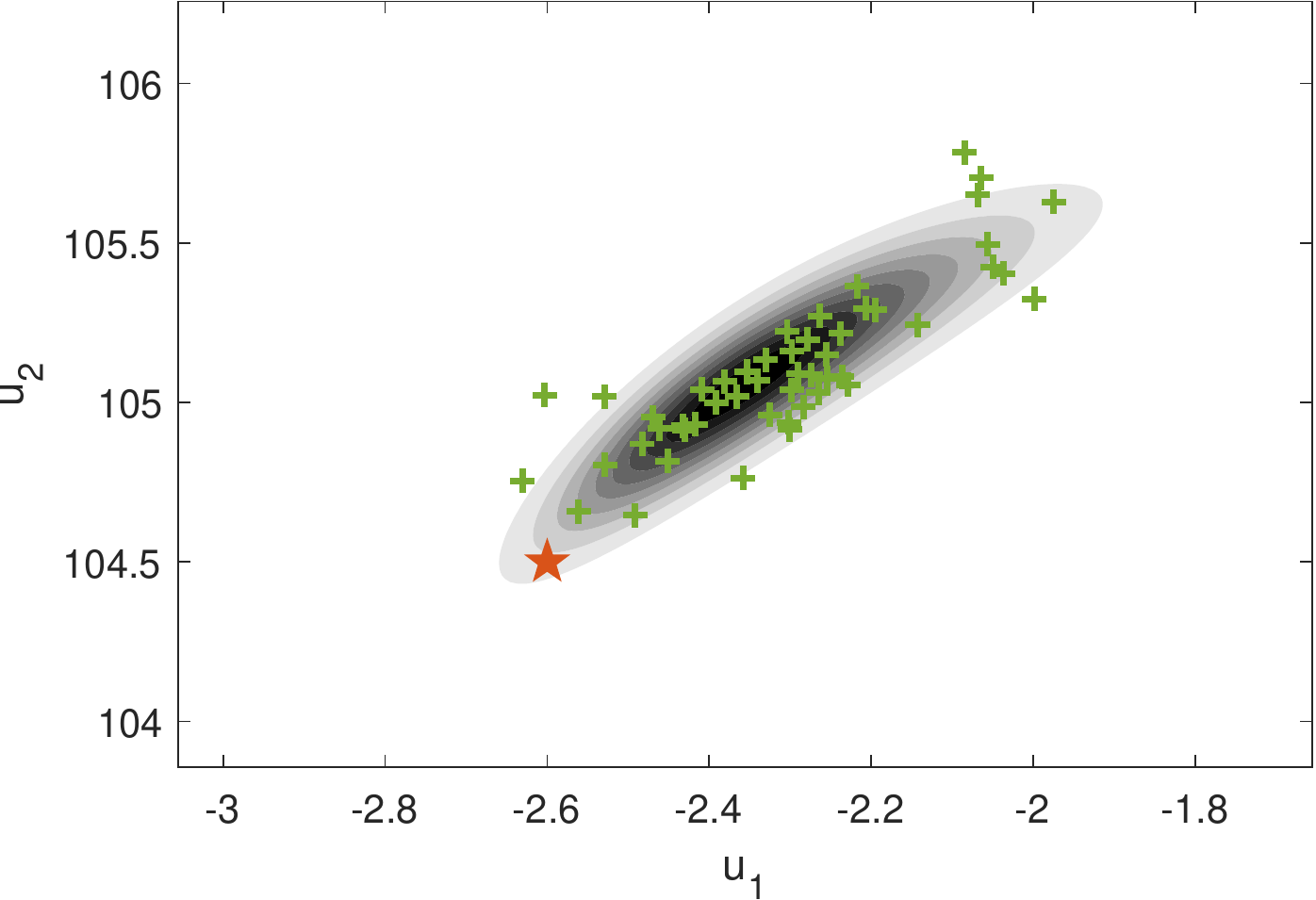}
		\caption{IEKF-SL.}
		% 			\label{fig:i1}
	\end{subfigure}%
	\caption{Ensemble members (green) after $100$ iterations, with truth $u^\dagger$ (red star) and contour plot of (unnormalized) posterior density.}
	 		\label{fig:e2d_ensem}
\end{figure}

We seek to recover $u = (u_1, u_2)^T$ from noisy observation of $p_u$ at two points $x_1 = 0.25$ and $x_2 = 0.75$. Precisely, we define $h(u) := \bigl( p_u(x_1), p_u(x_2) \bigr)^T$ and consider the inverse problem of recovering $u$ from data $y$ of the form
\begin{equation}\label{eq:elliptic2d_obs}
y = h( u) + \eta,
\end{equation} 
where $\eta\sim \Nc(0,  \gamma^2 I_2)$. We set a Gaussian prior on the unknown parameter  $u \sim \Nc(0, 1) \times \Nc(100, 16)$. In our numerical experiments we let the true parameter be $u^\dagger = (-2.6, 104.5)^T$ and use it to generate synthetic data $ y = h( u^\dagger) + \eta$ with noise level $\gamma = 0.1.$

\subsubsection{Implementation Details and Numerical Results}
We set the ensemble size to be $N =50$. The initial ensemble $\{ u_0^{(n)}\}_{n=1}^N$ is drawn independently from the prior. The length-step $\alpha$ is fixed to be 0.1 for all methods. We run each algorithm up to time $T = 10$, which corresponds to 100 iterations.

We plot the level curves of the posterior density of $u = (u_1, u_2)^T$ in Figure \ref{fig:e2d_ensem}. The forward model is approximately linear, as can be seen from the contour plots in Figure \ref{fig:e2d_ensem} or directly from Equation \eqref{eq:elliptic2d_exact}. Hence it can be used to validate the claims we made in previous sections.

\begin{figure}[htbp]
	\centering
	%\begin{subfigure}{.48\textwidth}
	\includegraphics[width=7cm]{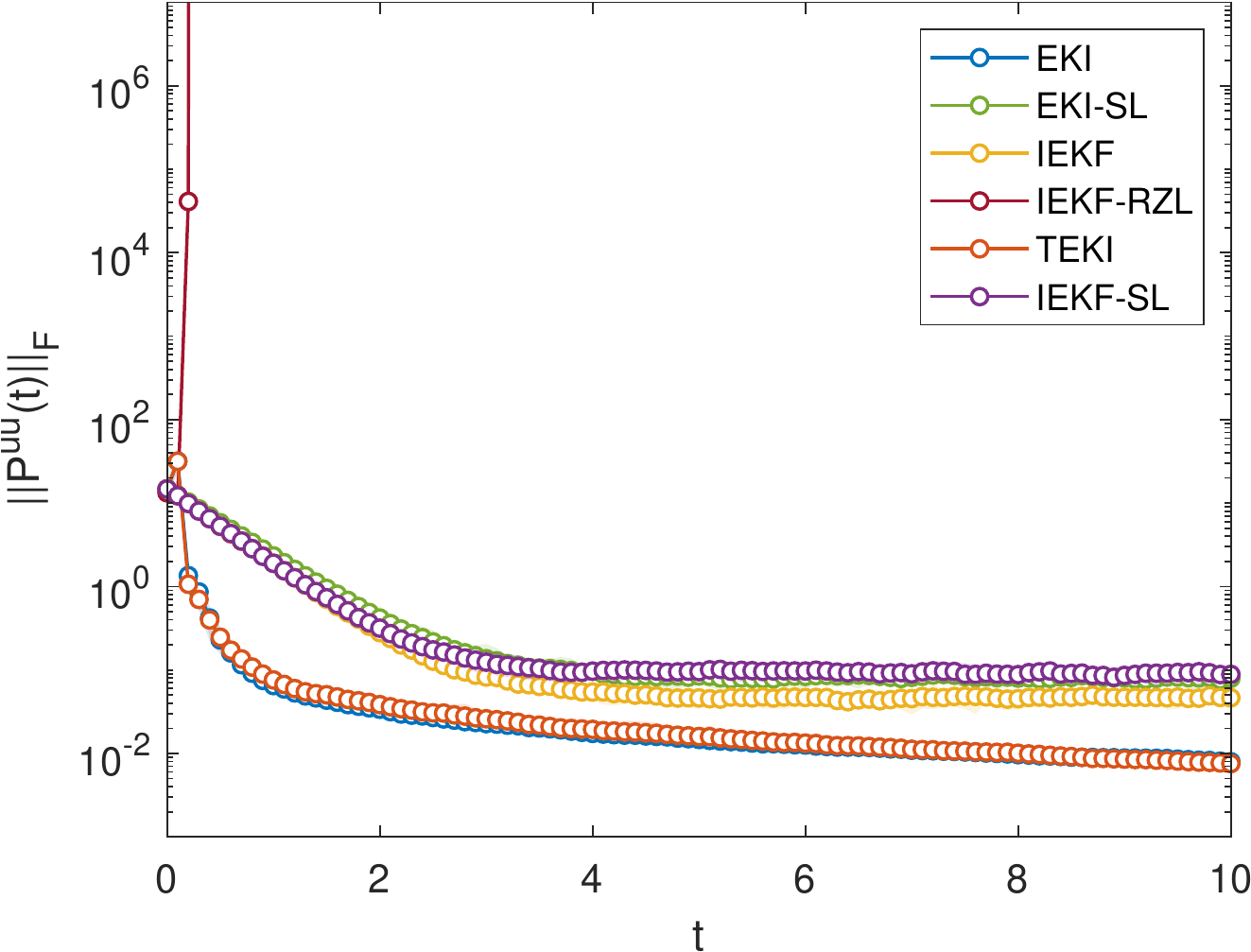}
	%\end{subfigure}
	\caption{Evolution of the Frobenius norm of the ensemble covariance $P^{uu}(t)$. The norm of IEKF-RZL blows up after a few iterations. The norms of the EKI and TEKI  are almost identical and monotonically decreasing. The norms of the new variants EKI-SL and IEKF-SL are similar and stabilize after around 40 iterations. The norm of IEKF lies between those of the old and new variants.}
	\label{fig:e2d_cov}
\end{figure}

\begin{figure}[ht]
	\centering
	\begin{subfigure}{.48\textwidth}
		\centering
		\includegraphics[width=7cm,height = 4.75cm]{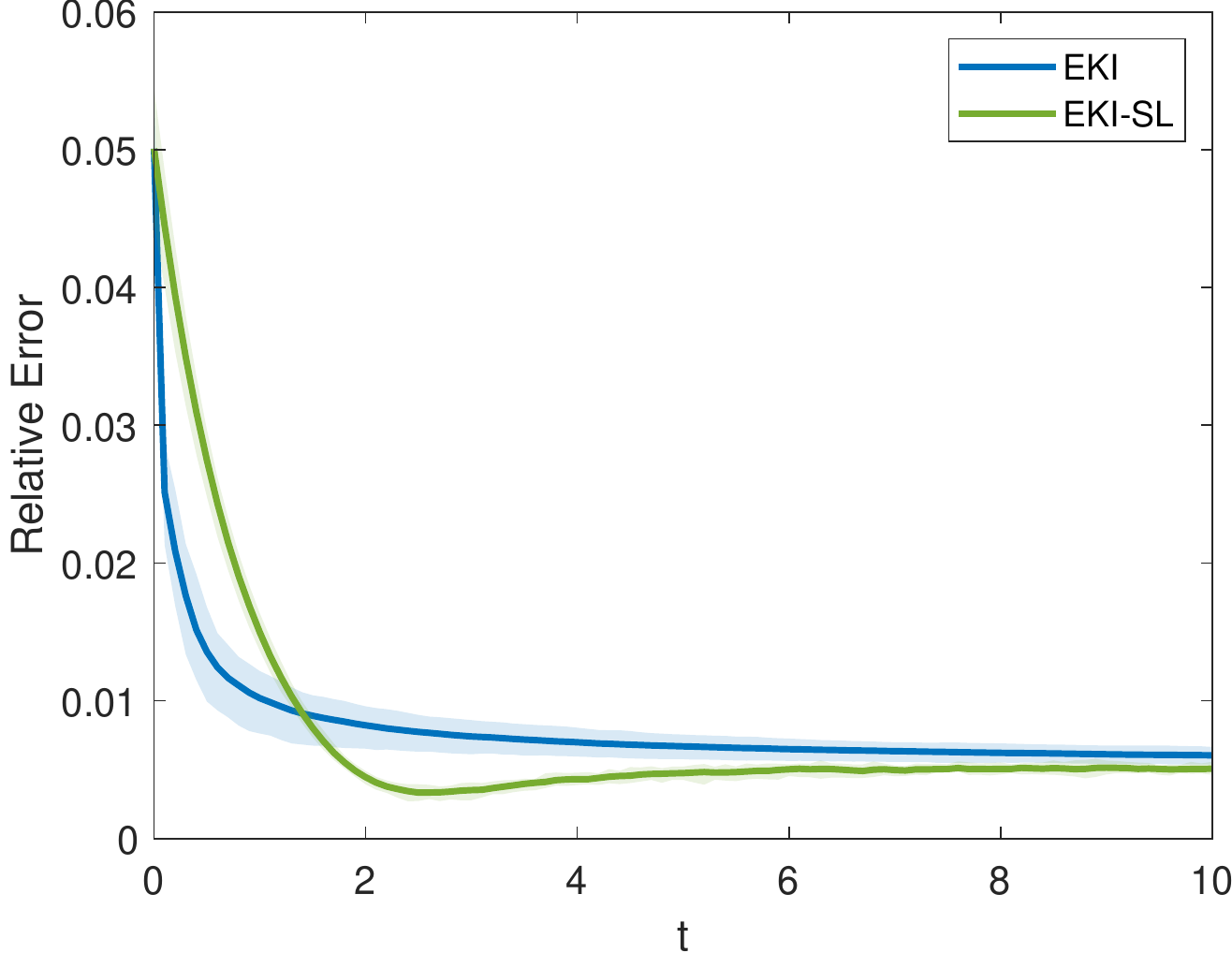}
		%		\caption{Truth $u^\dagger$}
		% 			\label{fig:i1}
	\end{subfigure}%
	%	\hfill
	\begin{subfigure}{.48\textwidth}
		\centering
		\includegraphics[width=7cm,height = 4.75cm]{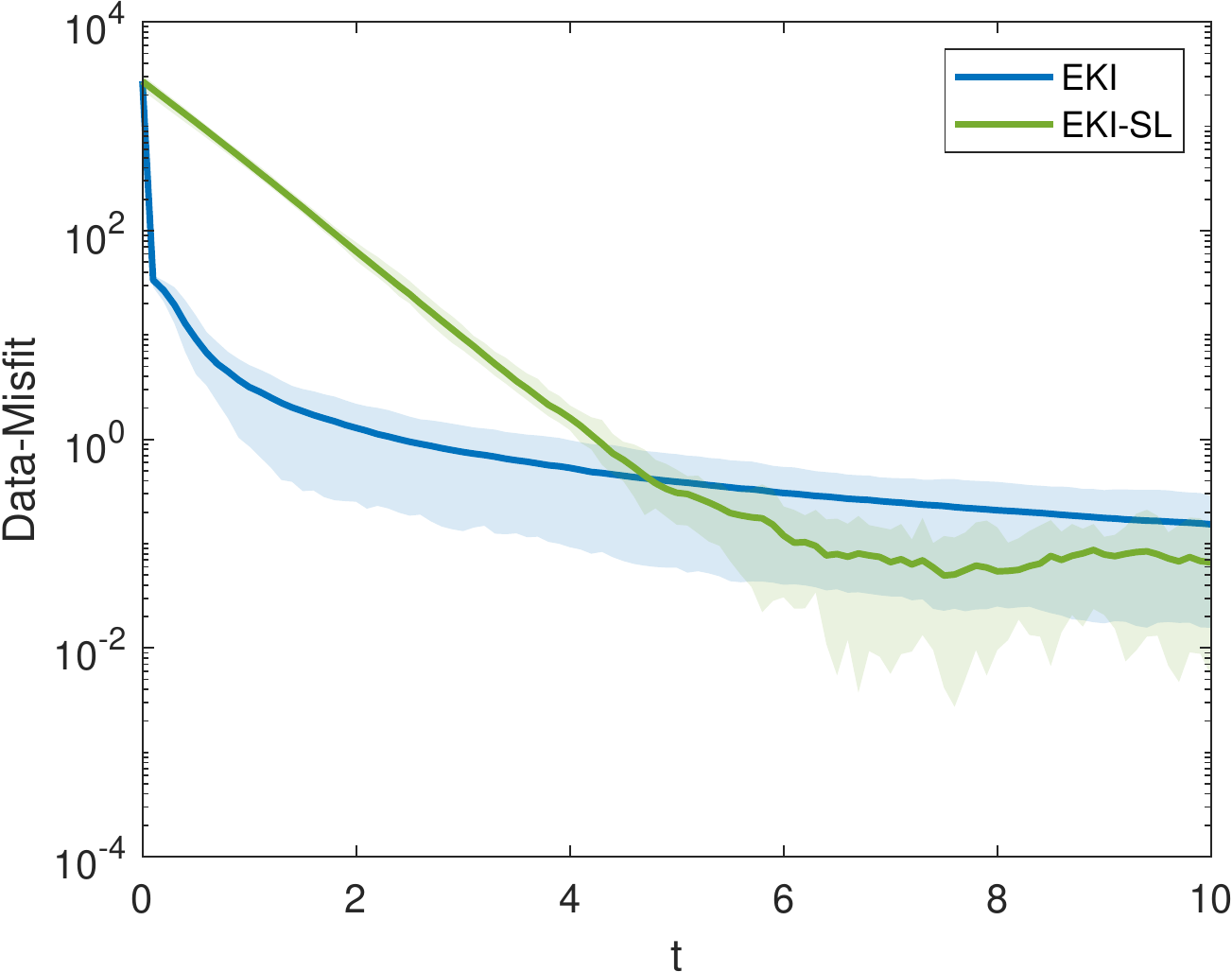}
		%		\caption{EKI}
		% 			\label{fig:i1}
	\end{subfigure}%
	\caption{EKI \& EKI-SL: Relative errors and data misfit  w.r.t time $t$.}
	 		\label{fig:e2d_dm}
\end{figure}

\begin{figure}[ht]
	\centering
	\begin{subfigure}{.48\textwidth}
		\centering
		\includegraphics[width=7cm,height = 4.75cm]{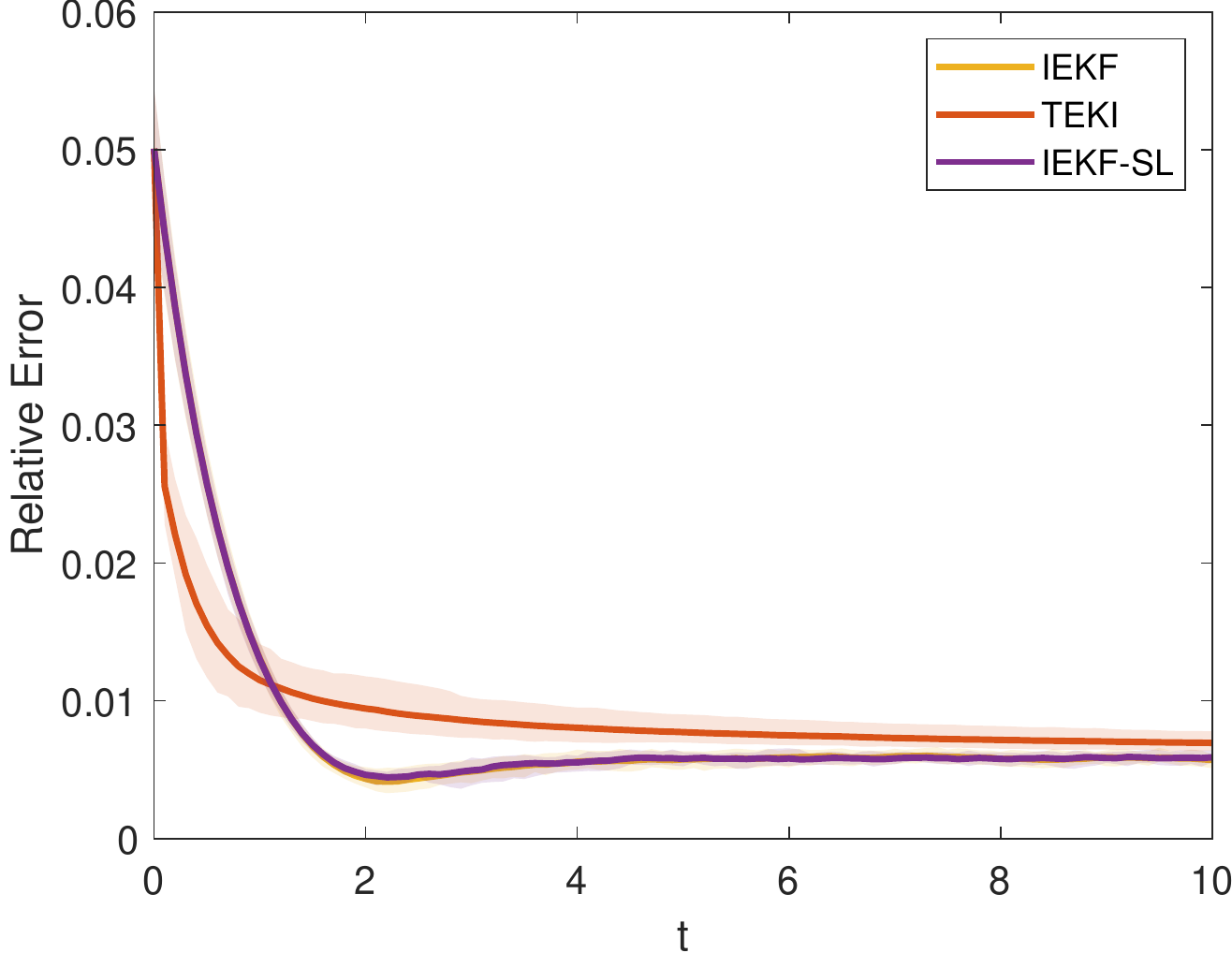}
		%		\caption{TEKI}
		% 			\label{fig:i1}
	\end{subfigure}%
	%	\hfill
	\begin{subfigure}{.48\textwidth}
		\centering
		\includegraphics[width=7cm,height = 4.75cm]{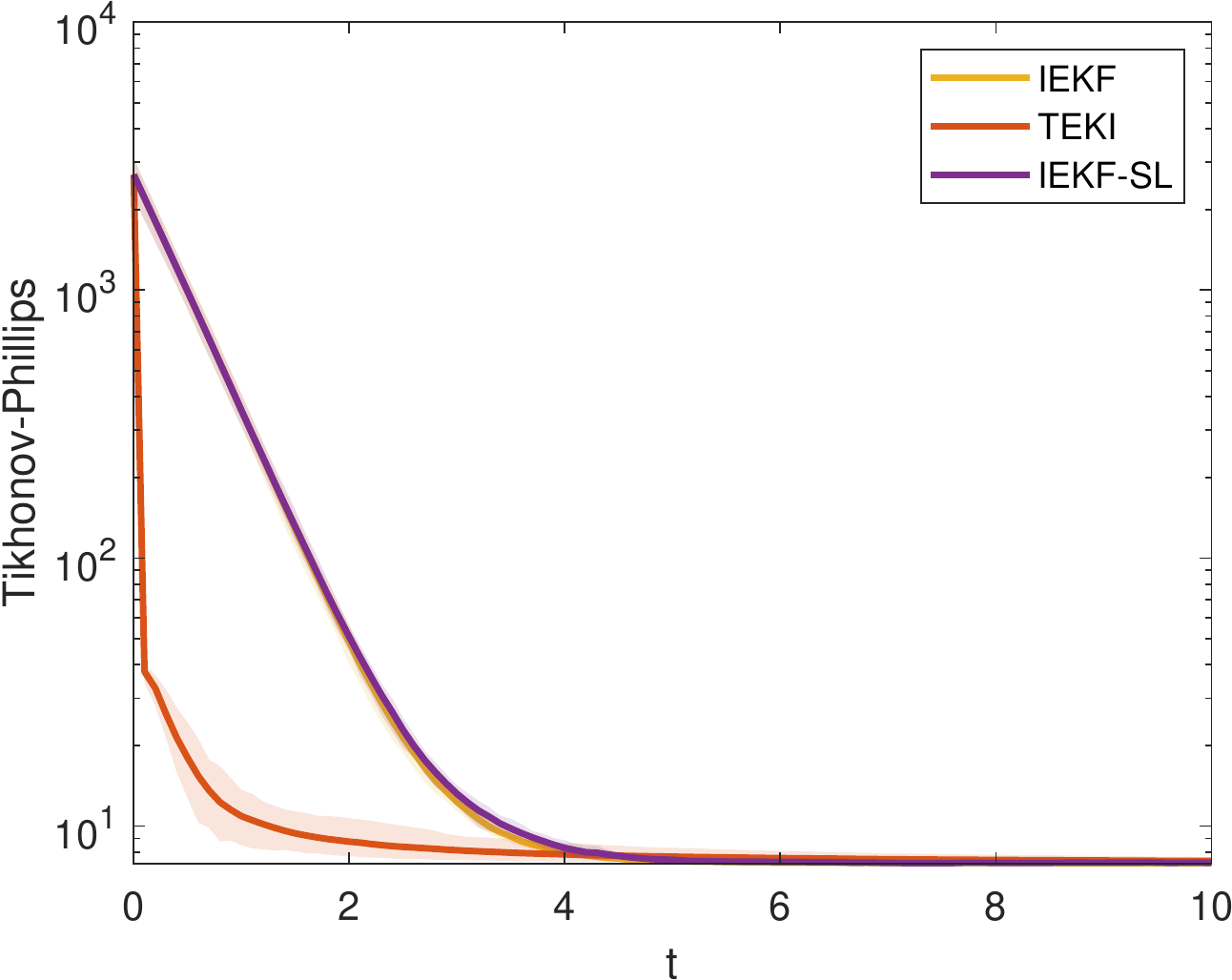}
		%		\caption{IEKF-SL}
%		 			\label{fig:e2d_tp}
	\end{subfigure}%
	\caption{TEKI, IEKF \& IEKF-SL: Relative errors and Tikhonov-Phillips objective w.r.t time $t$.}
	 		\label{fig:e2d_tp}
\end{figure}

Figure \ref{fig:e2d_cov} compares the time evolution of the empirical covariance of the different methods. The IEKF-RZL algorithm consistently blows up in the first few iterations due to the small noise which, as discussed in Section \ref{ssec:IenKFSL}, is an important drawback. Due to the poor performance of IEKF-RZL in small noise regimes, we will not include this method in subsequent comparisons. The EKI and TEKI plots show the `ensemble collapse' phenomenon discussed in Sections \ref{ssec:EKI} and \ref{ssec:TEKI}. Note that the size of the empirical covariances of EKI and TEKI decreases monotonically, and by the 100-th iteration their size is one order of magnitude smaller than that of the new variants IEKF-SL and EKI-SL. The ensemble collapse of EKI and TEKI can also be seen in Figure \ref{fig:e2d_ensem}, where we plot all the ensemble members after 100 iterations. 
 In contrast, Figure \ref{fig:e2d_cov}  shows that the size of the empirical  covariances of the new variants IEKF-SL and EKI-SL stabilizes after around 40 iterations, and Figure \ref{fig:e2d_ensem} suggests that the spread of the ensemble matches that of the posterior. These results are in agreement with the theoretical results derived in Sections \ref{sec:EKIN} and \ref{sec:IEKS} in a linear setting. Although we have not discussed the IEKF method when $\alpha$ is small, its ensemble covariance does not collapse in the linear setting, as can be seen in Figure \ref{fig:e2d_cov}. 

In Figures \ref{fig:e2d_dm} and \ref{fig:e2d_tp} we show the performance of the different methods along the full iteration sequence. Here and in subsequent numerical examples we use two performance assessments. First, the relative error, defined as $| m(t) - u^\dagger | / |u^\dagger |$ where $m(t)$ is the ensemble's empirical mean, which evaluates how well the ensemble mean approximates the truth. Second, we assess how each ensemble method performs in terms of its own optimization objective. Precisely, we report  the data-misfit objective $\Jdm\big(m(t)\big)$ for EKI and EKI-SL, and the Tikhonov-Phillips objective $\Jtp\big(m(t)\big)$ for IEKF, TEKI and IEKF-SL. We run 10 trials for each algorithm, using different ensemble initializations (drawn from the same prior), and generate the error bars accordingly. Since this is a simple toy problem, all methods perform well. However, we note that the first iterations of EKI and TEKI reduce the objective faster than other algorithms.

\subsection{High-Dimensional Linear Inverse Problem}\label{ssec:highd}
In this subsection we consider a linear Bayesian inverse problem from  \cite{ILS13}. This example illustrates the use of iterative ensemble Kalman methods in settings where both the size of the ensemble and the dimension of the data are significantly smaller than the dimension of the unknown parameter. 
\subsubsection{Problem setup}
 Consider the one dimensional elliptic equation
\begin{align}\label{eq:elliptic_forward}
\begin{split}
- \frac{\diff^2 p}{\diff x^2} + p &=u ,\quad    x  \in   (0,\pi), \\
						p(0 ) &= p( \pi) =0. 
\end{split}					
\end{align}
We seek to recover $u$ from noisy observation of $p$ at $k = 2^4 - 1$ equispaced points $x_j = \frac{j}{2^4} \pi.$ We assume that the data is generated from the model
\begin{equation}\label{eq:elliptic_obs}
y_j = p(x_j) + \eta_j, \quad j = 1,\dots,k,
\end{equation}
where $\eta_j \sim \Nc(0, \gamma^2)$ are independent. By defining $A = -\frac{\diff^2}{\diff x^2} + id$ and letting $\OO$ be the observation operator defined by $\big( \OO(p) \big)_j = p(x_j)$, we can rewrite \eqref{eq:elliptic_obs} as 
\begin{equation*}
y = h(u) + \eta, \quad \eta \sim \Nc(0, \gamma^2  I_k ),
\end{equation*}
where $h = \OO \circ A^{-1}$. The forward problem \eqref{eq:elliptic_forward} is solved on a uniform mesh with meshwidth $w=2^{-8}$ by a finite element method with continuous, piecewise linear basis functions. We assume that the unknown parameter $u$ has a Gaussian prior distribution, $u \sim \Nc(0, C_0)$ with covariance operator $C_0 = 10 (A - id)^{-1}$  with homogeneous Dirichlet boundary conditions.  This prior can be interpreted as the law of a Brownian bridge between $0$ and $\pi$. For computational purposes we view $u$ as a random vector in $\R^{2^8}$ and the linear map $h(u)$ is represented by a matrix $H \in \R^{(2^4 -1) \times 2^8}.$  The true parameter $u^\dagger$ is sampled from this prior (cf. Figure \ref{fig:e_ensem}), and is used to generate synthetic observation data $y = H u^\dagger + \eta$ with noise level $\gamma = 0.01$.

\subsubsection{Implementation Details and Numerical Results}
We set the ensemble size to be $N = 50$. The initial ensemble $\{ u_0^{(n)}\}_{n=1}^N$ is drawn independently from the prior. The length-step $\alpha$ is fixed to be 0.05 for all methods.  We run each algorithm up to time $T = 30$, which corresponds to 600 iterations.%is chosen based on the eigenvalues and eigenfunctions $\{\lambda_n, \phi_n\}_{n\in\N}$ of $C_0$. Specifically, we set $u_0^{(n)} = \sqrt{\lambda_n} \xi_n \phi_n$ with $\xi_n \sim \Nc(0, 1)$ for $n = 1,\dots,N$. The $n^{th}$ element can be regarded as the $n^{th}$ term in the KL expansion of the prior distribution.

We note that the linear inverse problem considered here has input dimension $2^8=256$, which is much larger than the ensemble size $N$. Combining the results from Figure \ref{fig:e_ensem}, \ref{fig:e_dm} and \ref{fig:e_tp}, EKI and EKI-SL clearly overfit the data. In contrast, TEKI over-regularizes the data, and we easily notice the ensemble collapse. The IEKF and IEKF-SL lie in between, and approximate the truth slightly better.
\begin{figure}[htbp]
	\centering
	\begin{subfigure}{.32\textwidth}
		\centering
		\includegraphics[width=5cm]{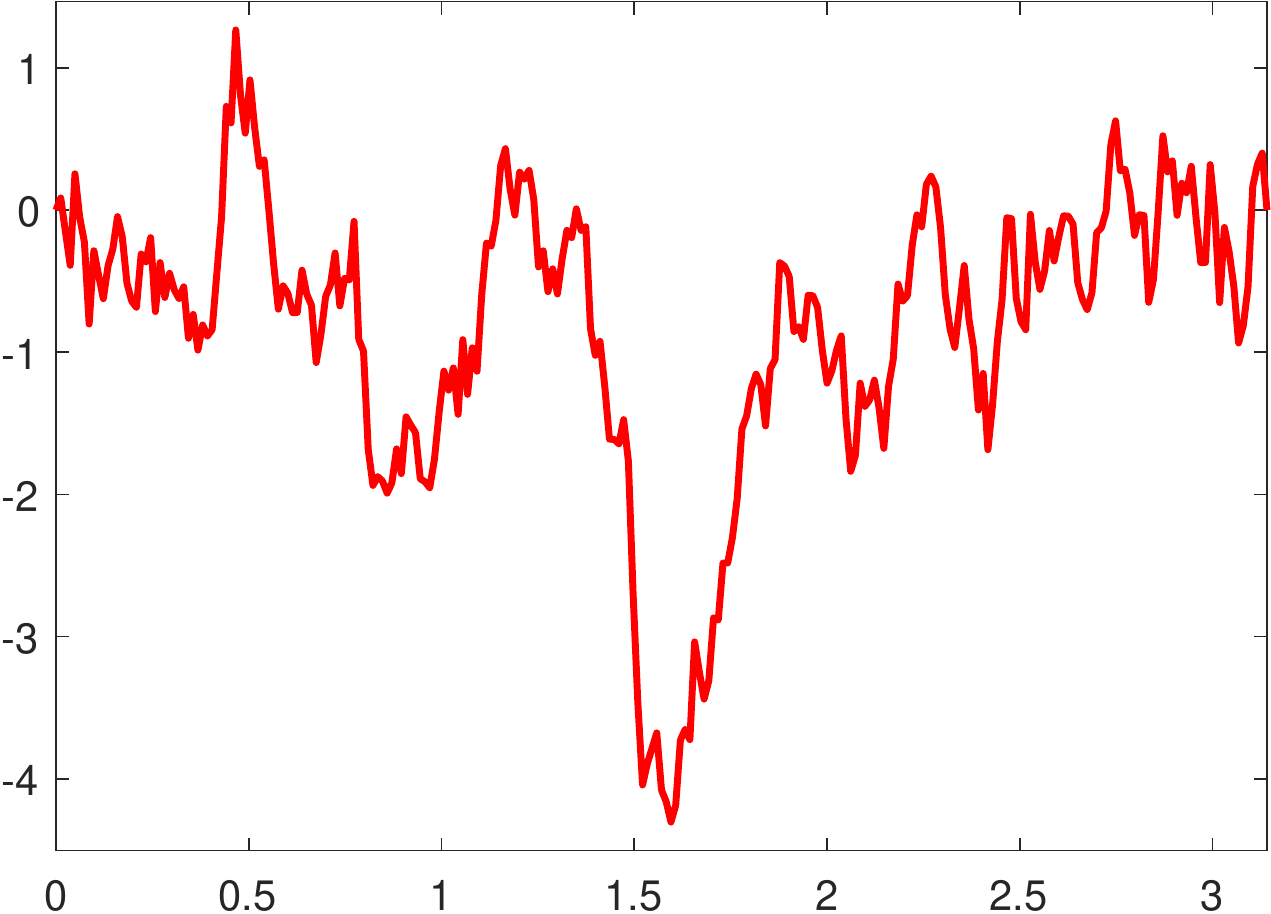}
		\caption{Truth $u^\dagger$}
		% 			\label{fig:i1}
	\end{subfigure}%
%	\hfill
	\begin{subfigure}{.32\textwidth}
		\centering
		\includegraphics[width=5cm]{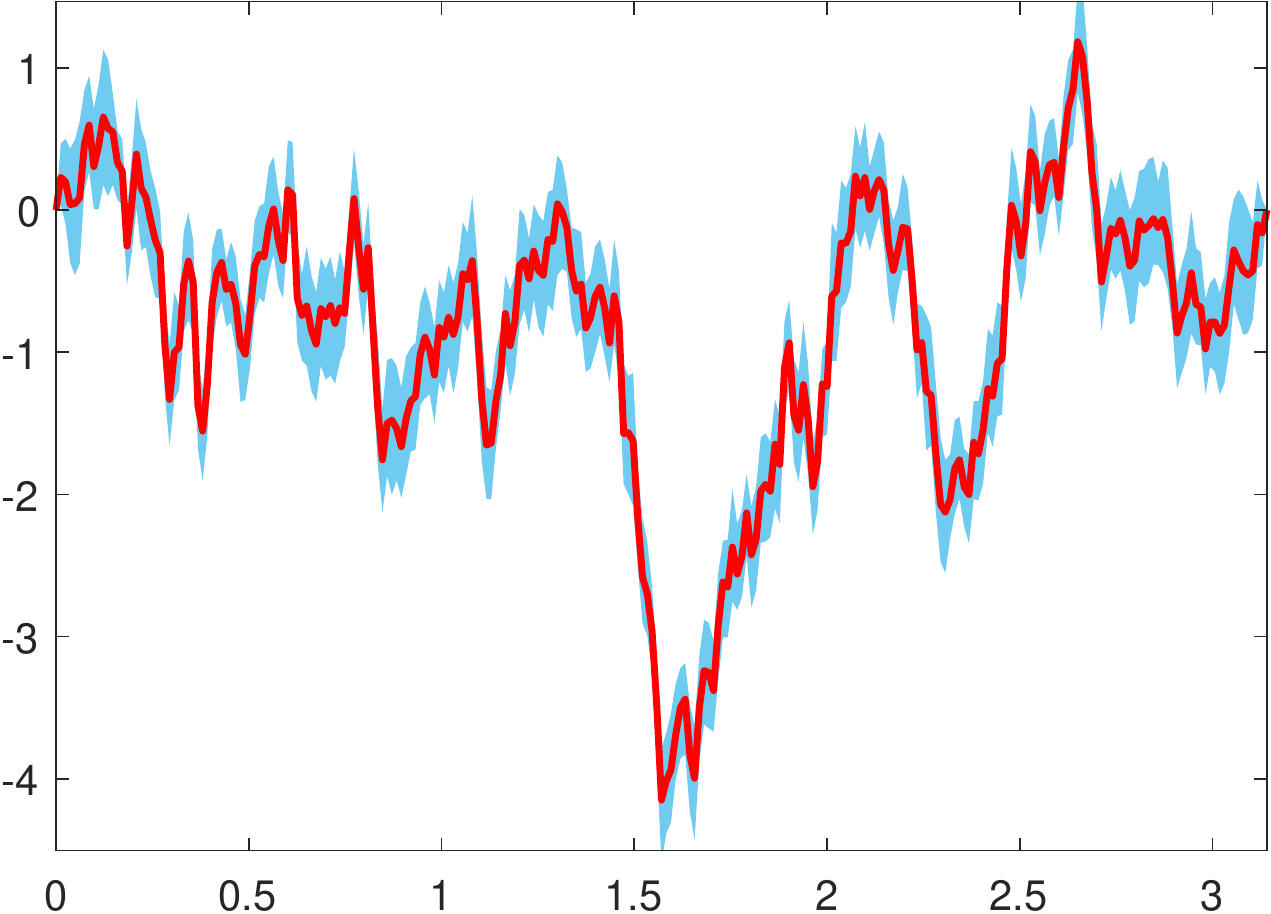}
		\caption{EKI}
		% 			\label{fig:i1}
	\end{subfigure}%
%	\hfill
		\begin{subfigure}{.32\textwidth}
		\centering
		\includegraphics[width=5cm]{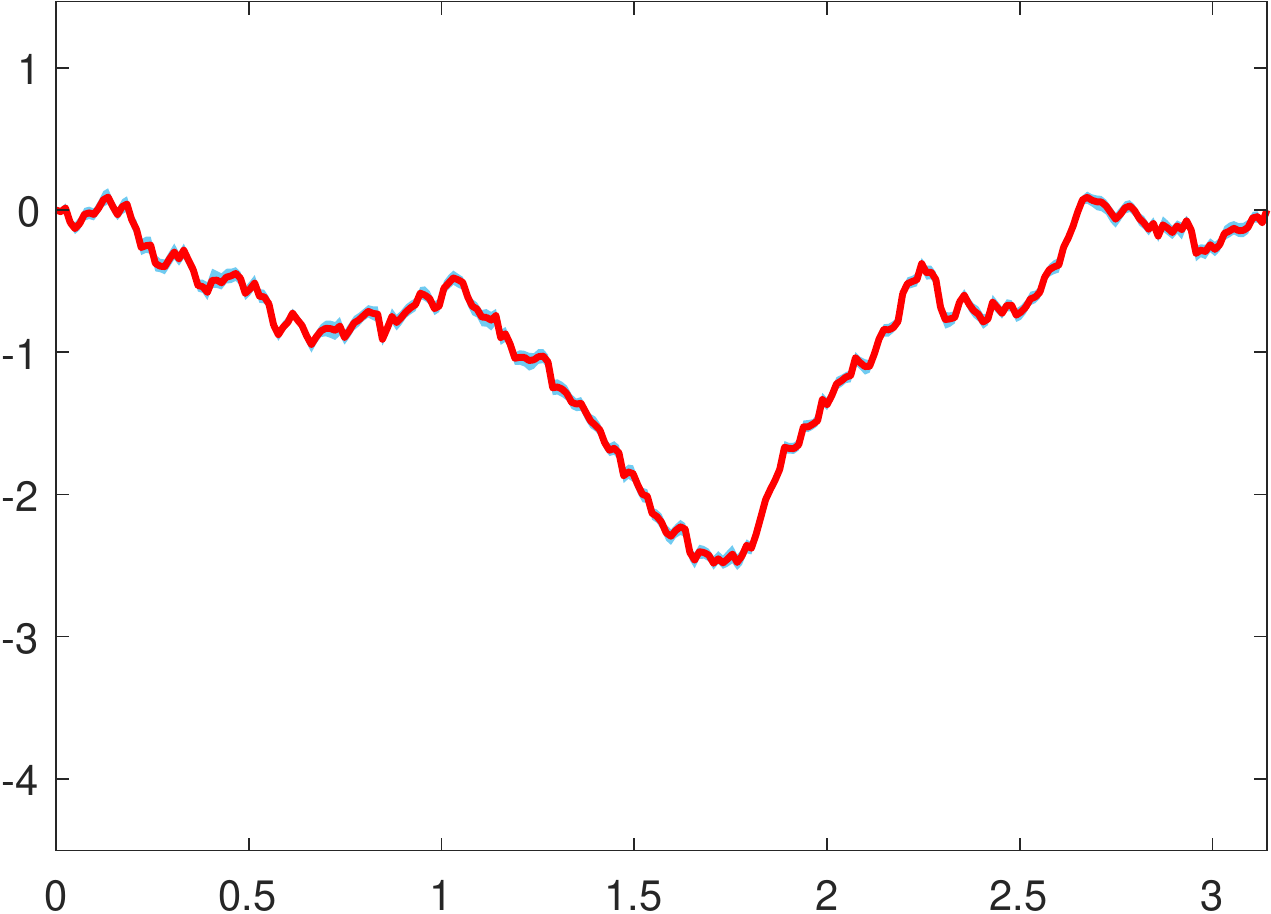}
		\caption{TEKI}
	% 			\label{fig:i1}
	\end{subfigure}%
	\vskip\baselineskip
	\begin{subfigure}{.32\textwidth}
		\centering
		\includegraphics[width=5cm]{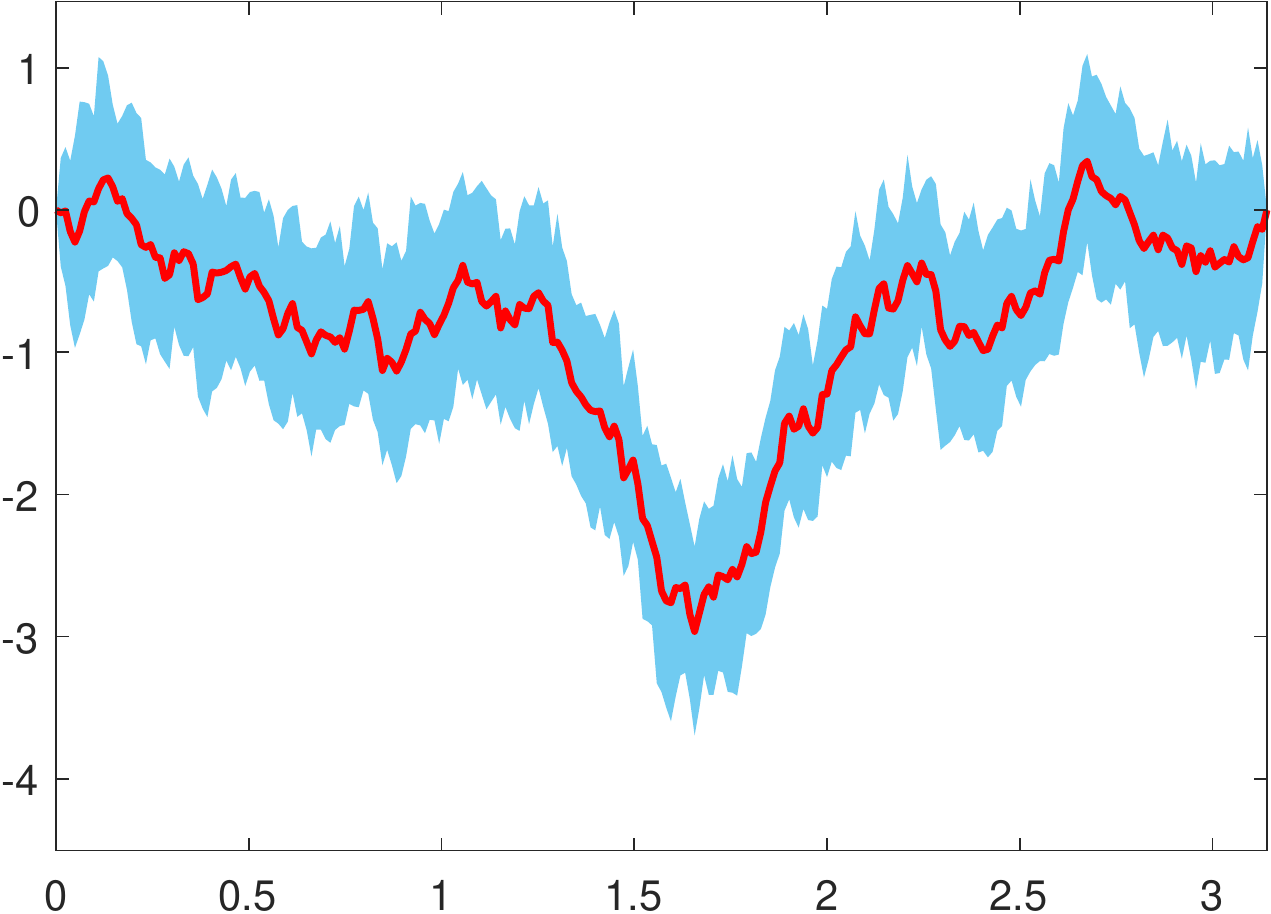}
		\caption{IEKF}
		% 			\label{fig:i1}
	\end{subfigure}%
%	\hfill
	\begin{subfigure}{.32\textwidth}
		\centering
		\includegraphics[width=5cm]{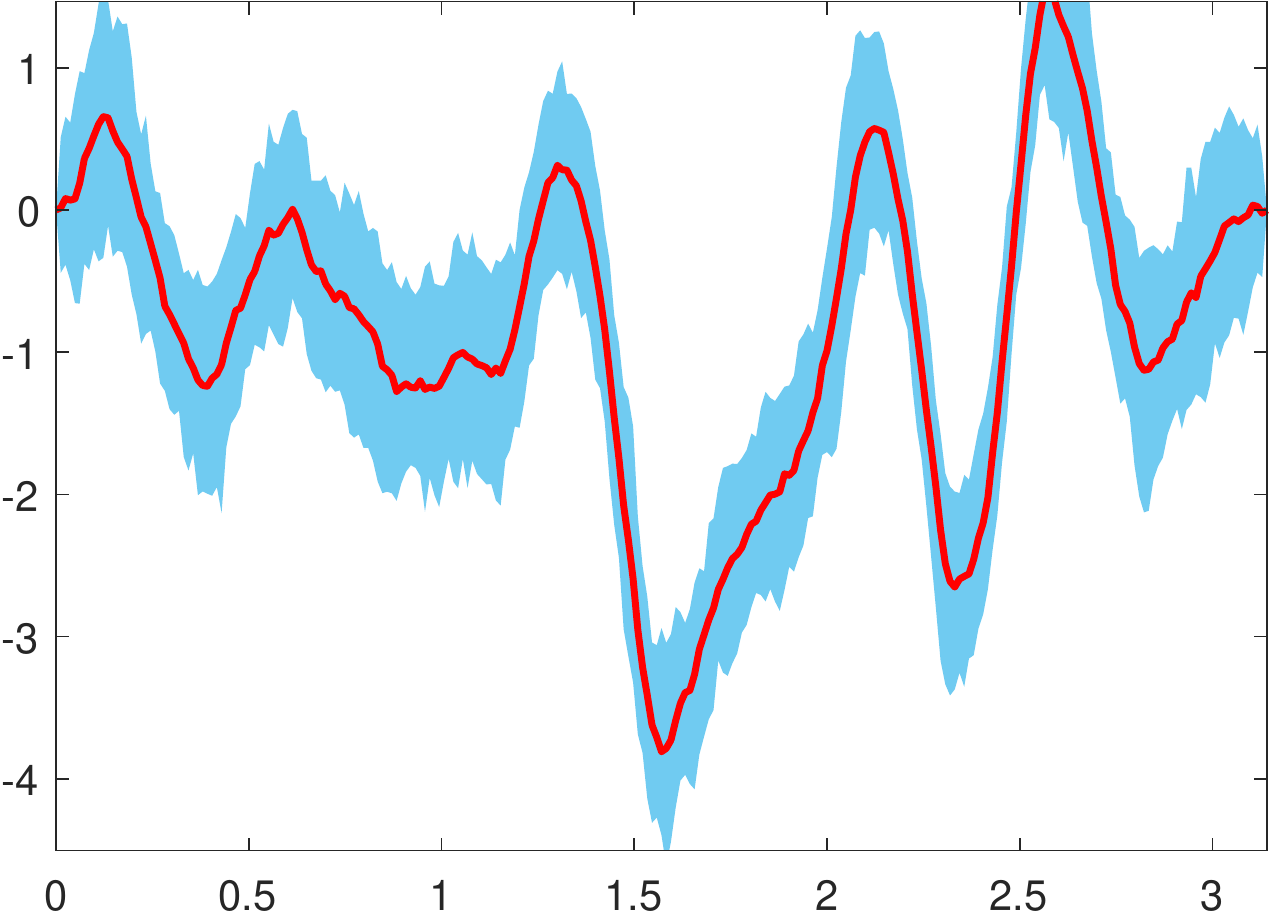}
		\caption{EKI-SL}
		% 			\label{fig:i1}
	\end{subfigure}%
%	\hfill
	\begin{subfigure}{.32\textwidth}
		\centering
		\includegraphics[width=5cm]{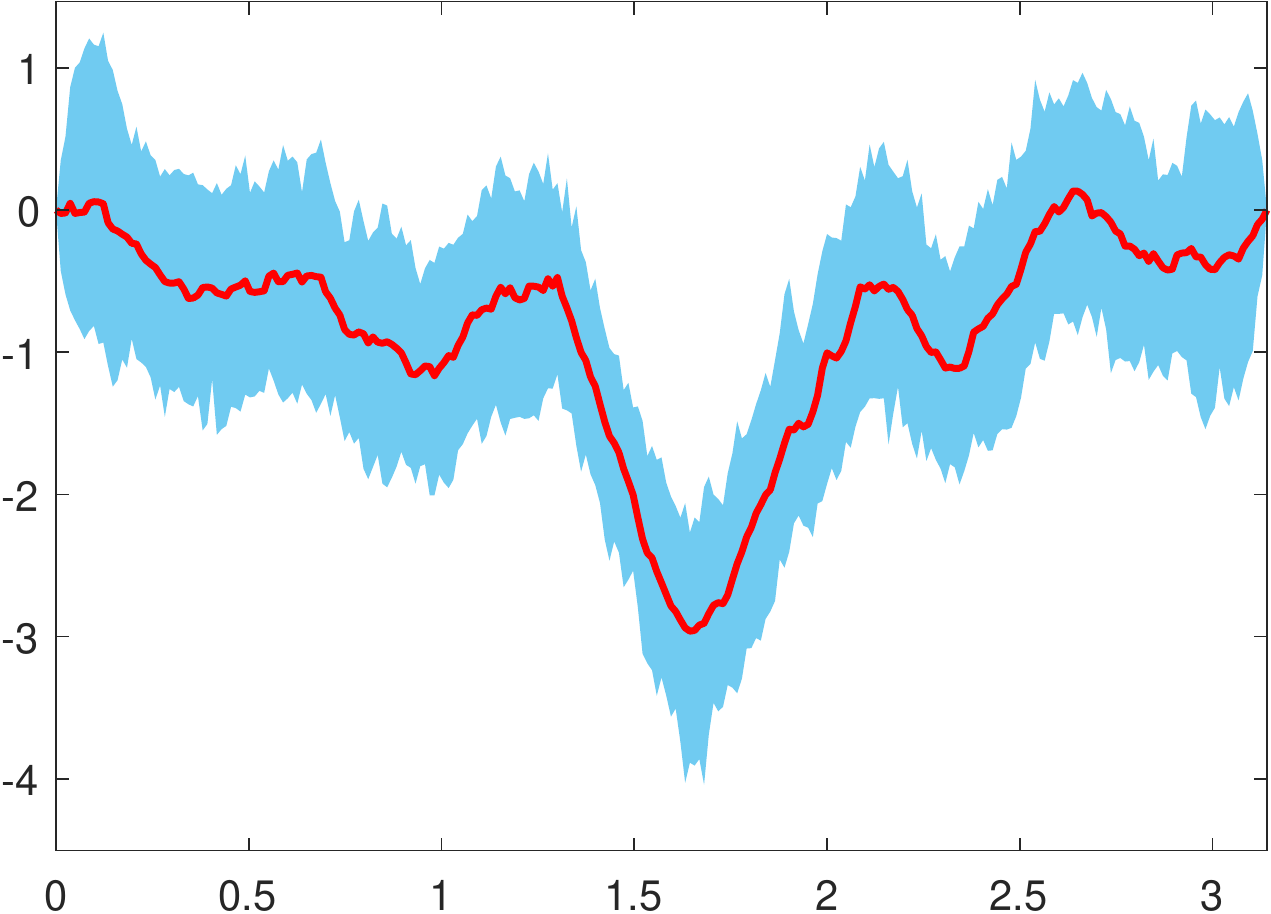}
		\caption{IEKF-SL}
		% 			\label{fig:i1}
	\end{subfigure}%
	 \caption{Ensemble mean (red) at the final iteration, with 10, 90-quantiles (blue).}
	 		\label{fig:e_ensem}
\end{figure}
%\FloatBarrier

\begin{figure}[htbp]
	\centering
	\begin{subfigure}{.48\textwidth}
		\centering
		\includegraphics[width=7cm]{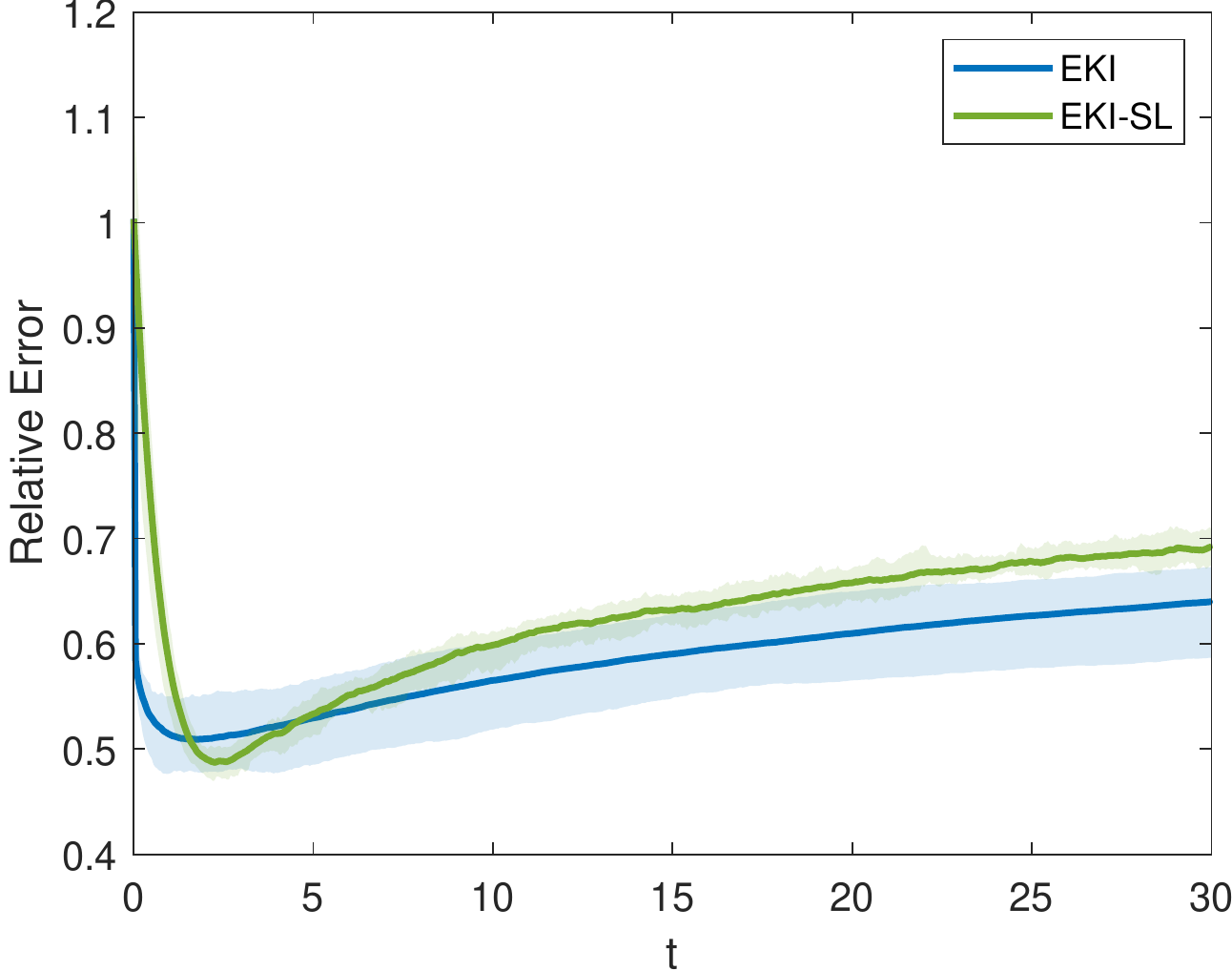}
%		\caption{Truth $u^\dagger$}
		% 			\label{fig:i1}
	\end{subfigure}%
	%	\hfill
	\begin{subfigure}{.48\textwidth}
		\centering
		\includegraphics[width=7cm]{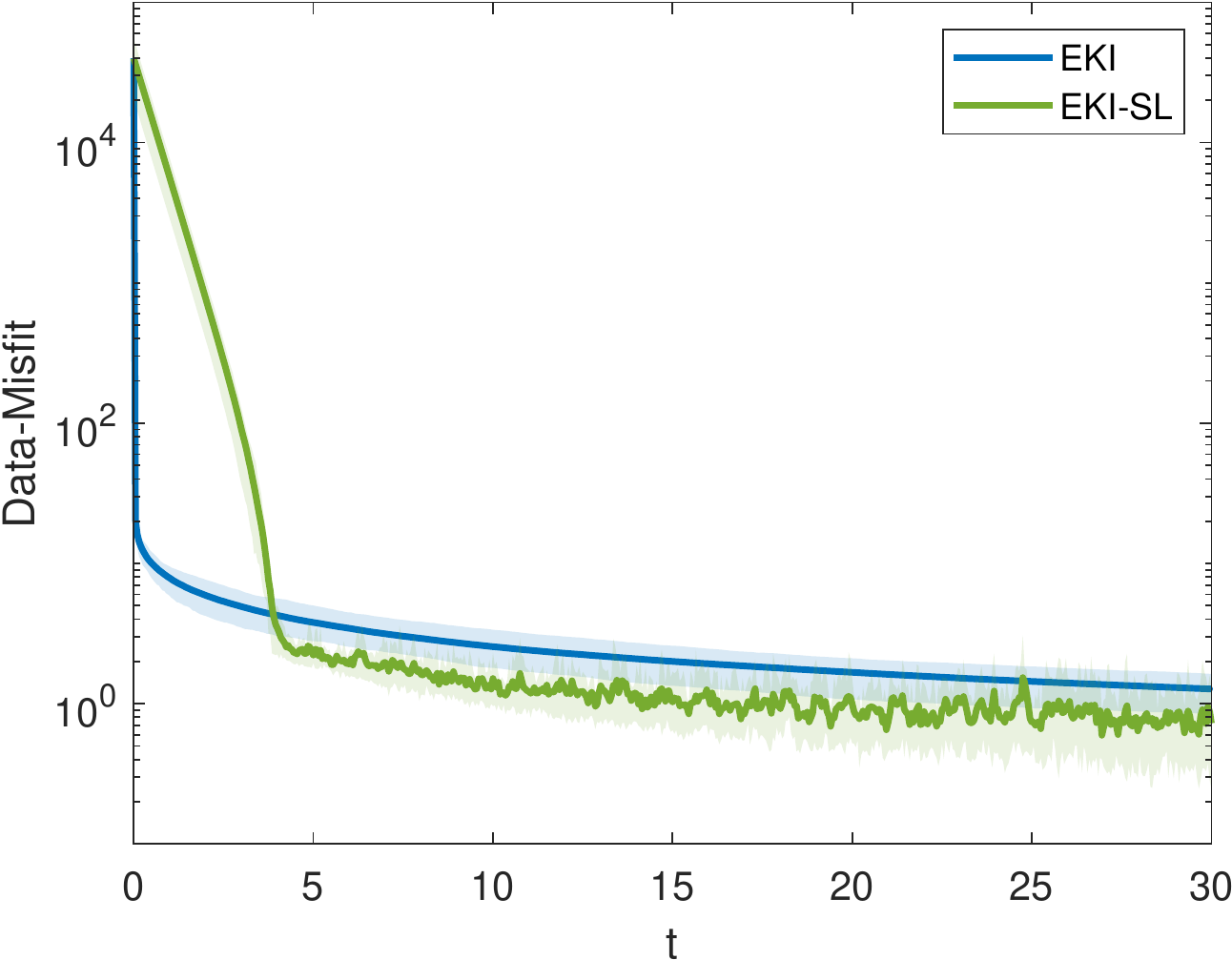}
%		\caption{EKI}
		% 			\label{fig:i1}
	\end{subfigure}%
	\caption{EKI \& EKI-SL: Relative errors and data misfit  w.r.t time $t$.}
 		\label{fig:e_dm}
\end{figure}

\begin{figure}[ht]
	\centering
	\begin{subfigure}{.48\textwidth}
		\centering
		\includegraphics[width=7cm]{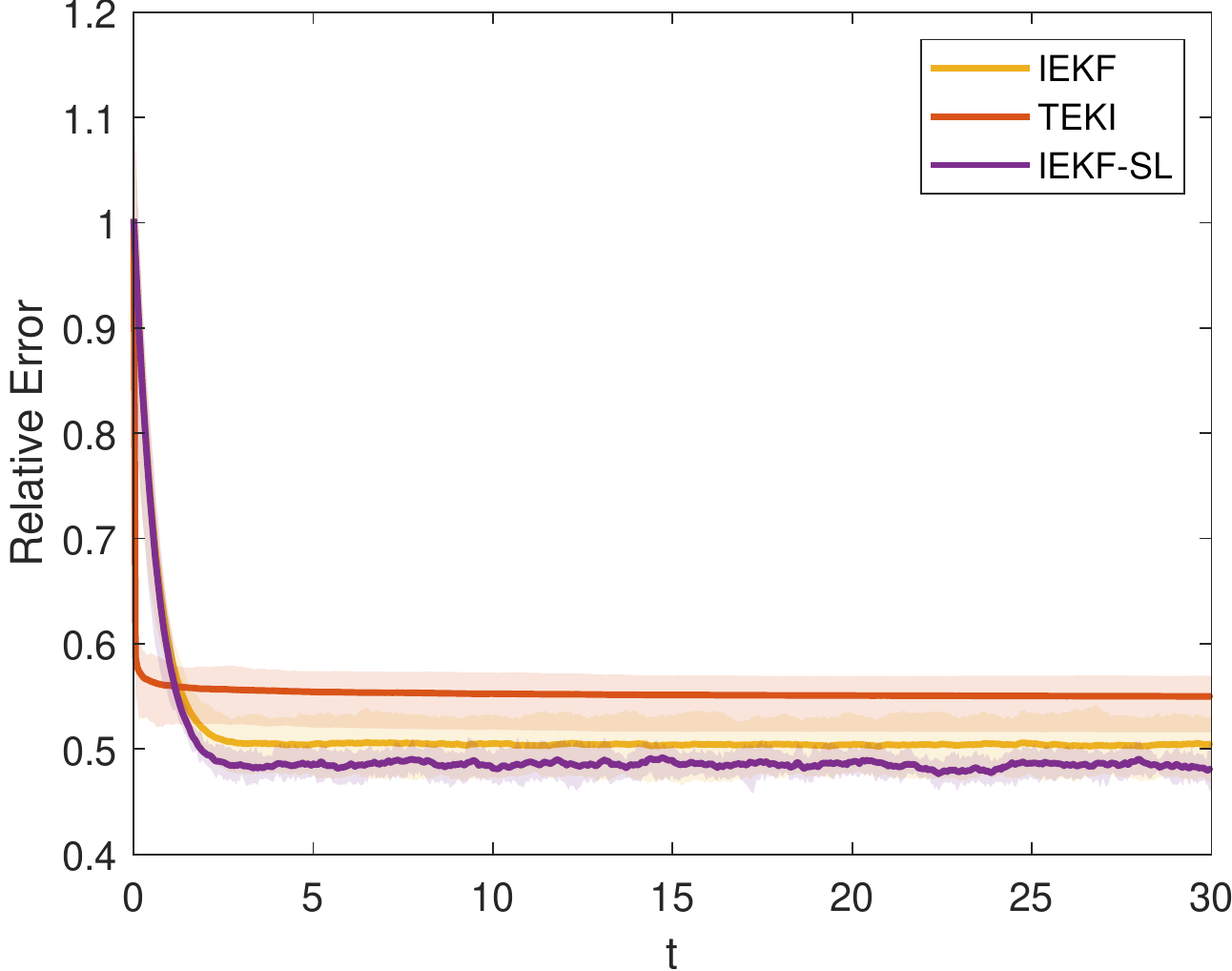}
%		\caption{TEKI}
		% 			\label{fig:i1}
	\end{subfigure}%
	%	\hfill
	\begin{subfigure}{.48\textwidth}
		\centering
		\includegraphics[width=7cm]{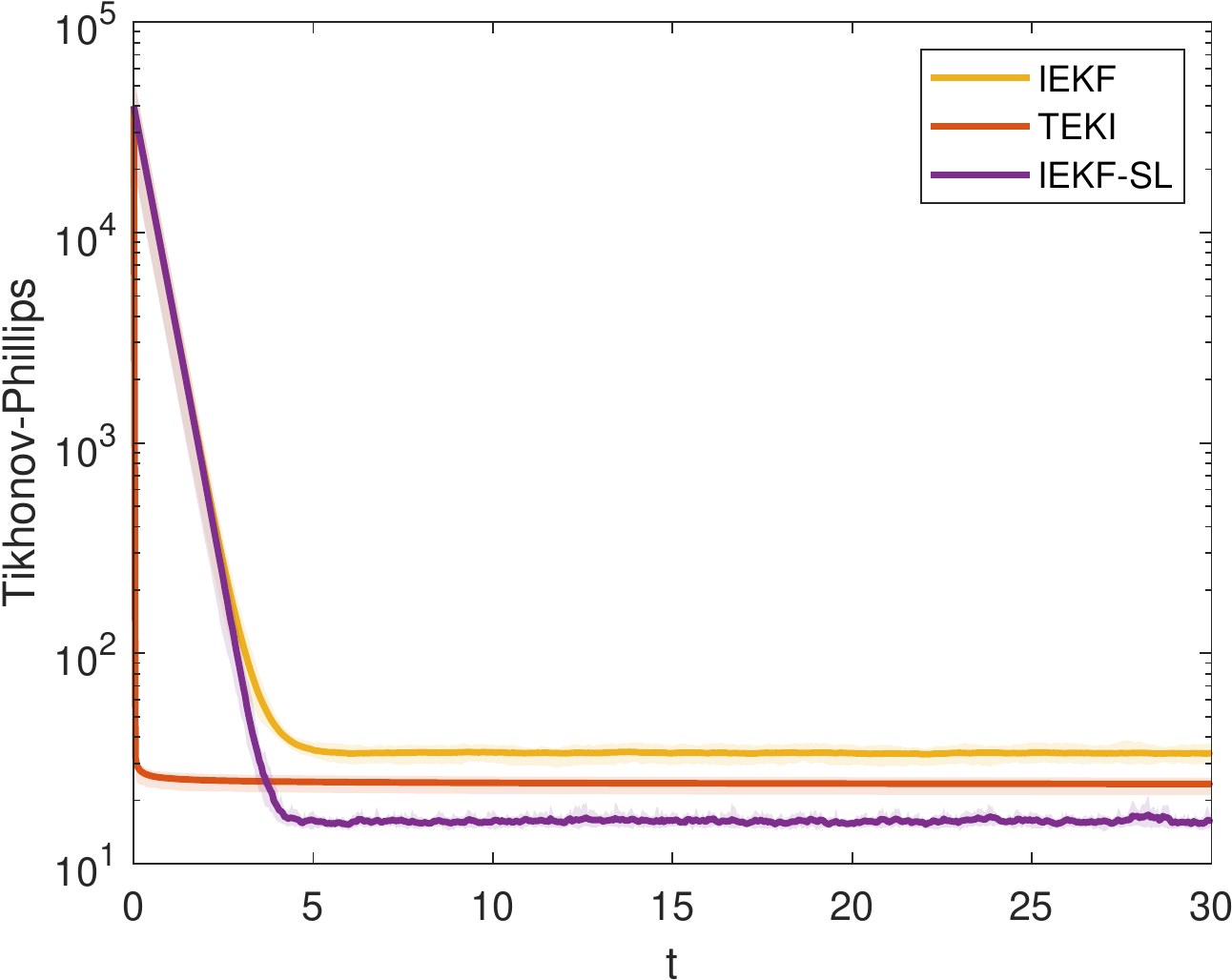}
%		\caption{IEKF-SL}
		% 			\label{fig:i1}
	\end{subfigure}%
	\caption{IEKF, TEKI \& IEKF-SL: Relative errors and Tikhonov-Phillips objective w.r.t time $t$.}
	 		\label{fig:e_tp}
\end{figure}

\subsection{Lorenz-96 Model}\label{ssec:Lorenz96}
In this subsection we investigate the use of iterative ensemble Kalman methods to recover the initial condition of the Lorenz-96 system from partial and noisy observation of the solution at two positive times. The experimental framework is taken from  \cite{CT19} and is illustrative of the use of iterative ensemble Kalman methods in numerical weather forecasting. 
\subsubsection{Problem Setup}
 Consider the dynamical system
\begin{equation}
\label{eq:l96}
\begin{gathered}
\frac{dz_l}{dt} = z_{l-1}(z_{l+1} - z_{l-2}) - z_l + F, \quad l=1,\ldots,d, \\
z_0 = z_d, \quad z_{d+1} = z_1, \quad z_{-1} = z_{d-1}.
\end{gathered}
\end{equation}
Here $z_l$ denotes the $l^{th}$ coordinate of the current state of the system and $F$ is a constant forcing with default value of $F = 8$. The dimension $d$ is often chosen as 40. We want to recover the initial condition
\begin{equation*}
u := (z_1(0), \dots, z_d(0))^T
\end{equation*}
of \eqref{eq:l96} from noisy partial measurements at discrete times $\{s_i\}_{i=1}^{I}$:
\begin{equation}\label{eq:L96_data}
y_{i,j} = u_{l_j}(s_i) + \eta_{i, j},
\end{equation}
where $\{l_j\}_{j=1}^J \subset \{1,2,\dots, d\}$ is the subset of observed coordinates, and $\eta_{i,j}\sim\Nc(0, \gamma^2)$ are assumed to be independent. In our numerical experiments we set $I = 2$, $J = 20$, $\{s_1, s_2\} = \{0.3, 0.6\}$, $\{l_j\}_{j=1}^{20} = \{1, 3, 5, \dots, 39\}$. We set the prior on $u$ to be a Gaussian $\Nc(0, 2 I_d)$.  The true parameter $u^{\dagger} \in \R^{40}$ is shown in Figure \ref{fig:l96_ensem}, and is used to generate the observation data $\{y_{i,j}\}$ according to \eqref{eq:L96_data} with noise level $\gamma = 0.01.$ 

\subsubsection{Implementation Details and Numerical Results}
We set the ensemble size to be $N = 50$. The initial ensemble $\{u_0^{(n)}\}_{n=1}^N$ is drawn independently from the prior. The length-step $\alpha$ is fixed to be 0.05 for all methods. We run each algorithm up to time $T = 30$, which corresponds to 600 iterations.

This is a moderately high dimensional nonlinear problem, where the forward model involves a black-box solver of differential equations. Figure \ref{fig:l96_ensem} clearly indicates an ensemble collapse for the EKI and TEKI methods. Although they are still capable of finding a descending direction of the loss direction (see Figures \ref{fig:l96_dm} and \ref{fig:l96_tp}), iterates get stuck in a local minimum. In contrast, we can see the advantage of IEKF, EKI-SL and TEKI-SL in this setting: intuitively, a broader spread of the ensemble allows these methods to `explore' different regions in the input space, and thereby to find a solution with potentially lower loss.

We notice that in practice `covariance inflation' is often applied in EKI and TEKI algorithms, by manually injecting small random noise after each ensemble update. In general, this technique will `force' a non-zero empirical covariance, prevent ensemble collapse and boost the performence of EKI and TEKI. However,  the amount of noise injected is an additional hyperparameter that should be chosen manually, which should depend on the input dimension, observation noise, etc. Here we give a fair comparison of the different methods introduced, under the same time-stepping setting, with as few hyperparameters as possible. 
\begin{figure}[htbp]
	\centering
	\begin{subfigure}{.32\textwidth}
		\centering
		\includegraphics[width=5cm]{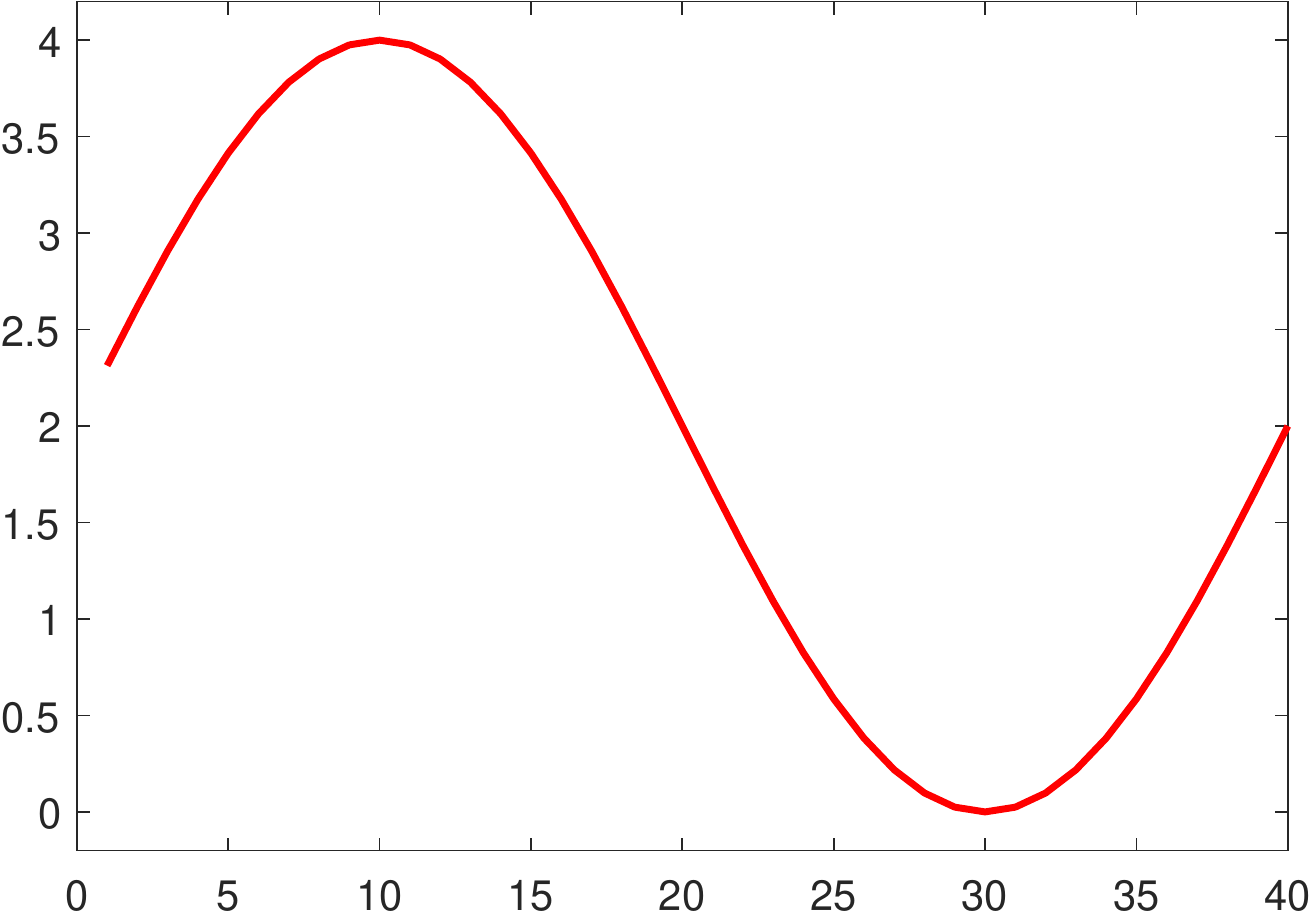}
		\caption{Truth $u^\dagger$.}
		% 			\label{fig:i1}
	\end{subfigure}%
	%	\hfill
	\begin{subfigure}{.32\textwidth}
		\centering
		\includegraphics[width=5cm]{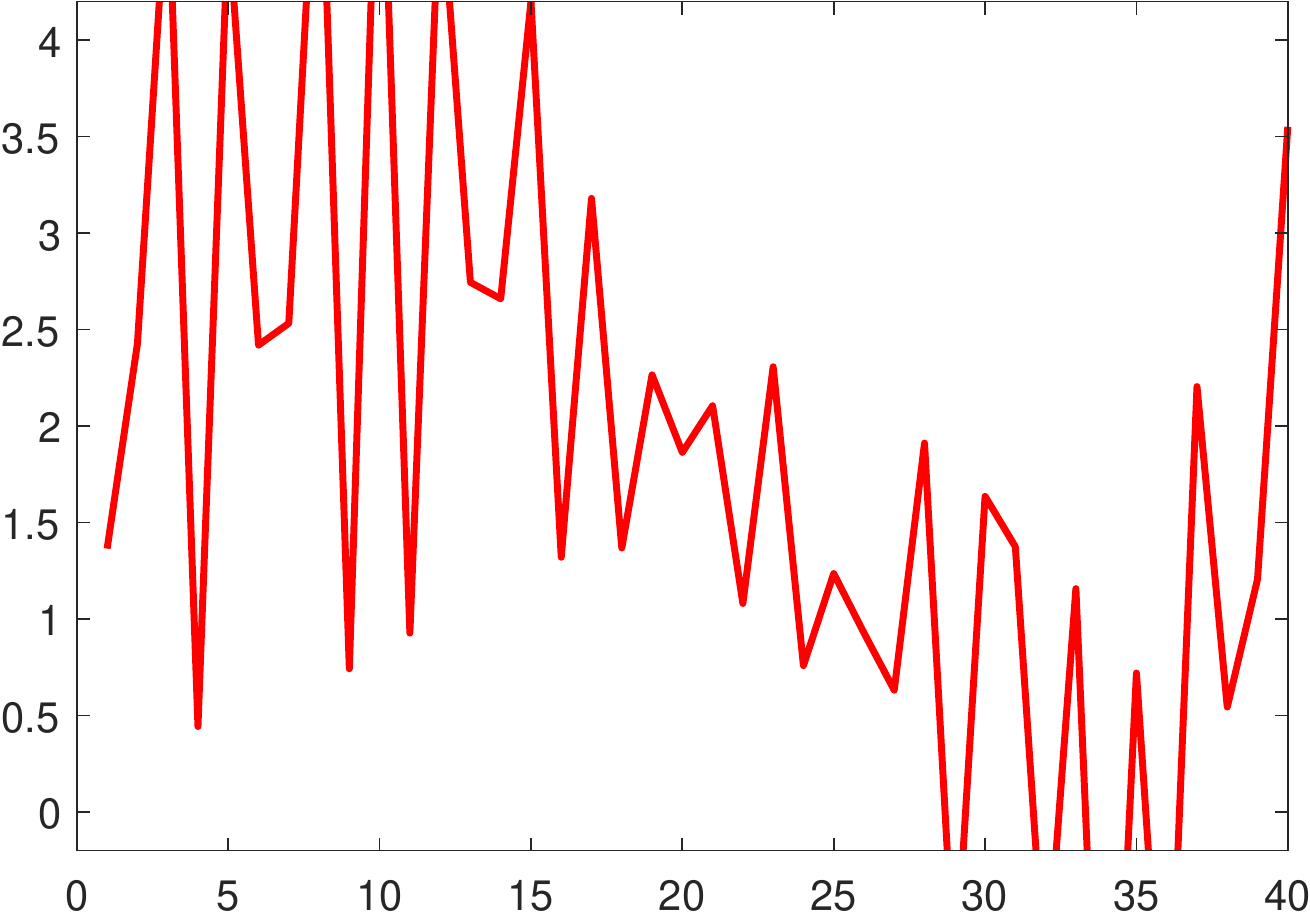}
		\caption{EKI.}
		% 			\label{fig:i1}
	\end{subfigure}%
	%	\hfill
	\begin{subfigure}{.32\textwidth}
		\centering
		\includegraphics[width=5cm]{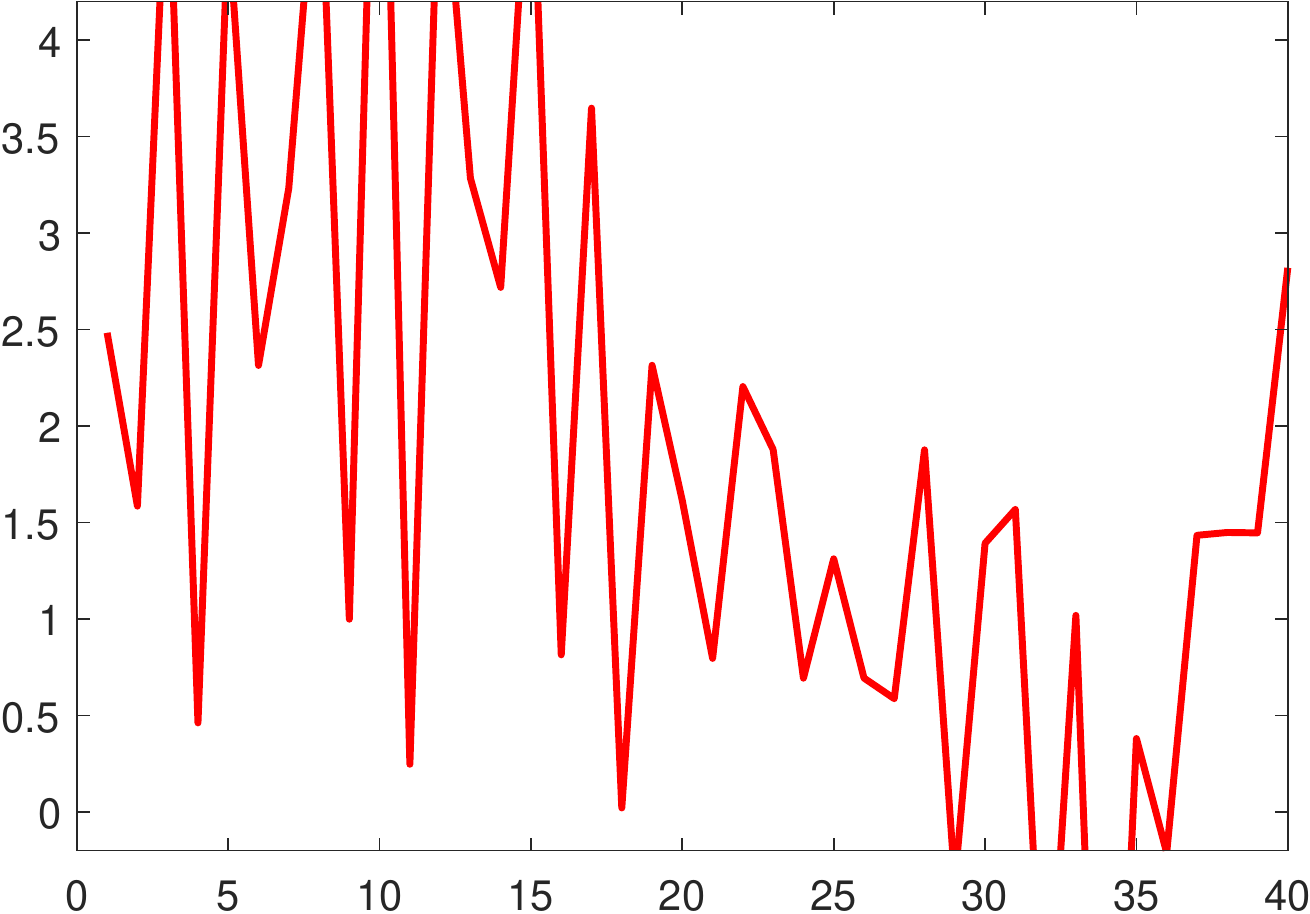}
		\caption{TEKI.}
		% 			\label{fig:i1}
	\end{subfigure}%
	\vskip\baselineskip
	\begin{subfigure}{.32\textwidth}
		\centering
		\includegraphics[width=5cm]{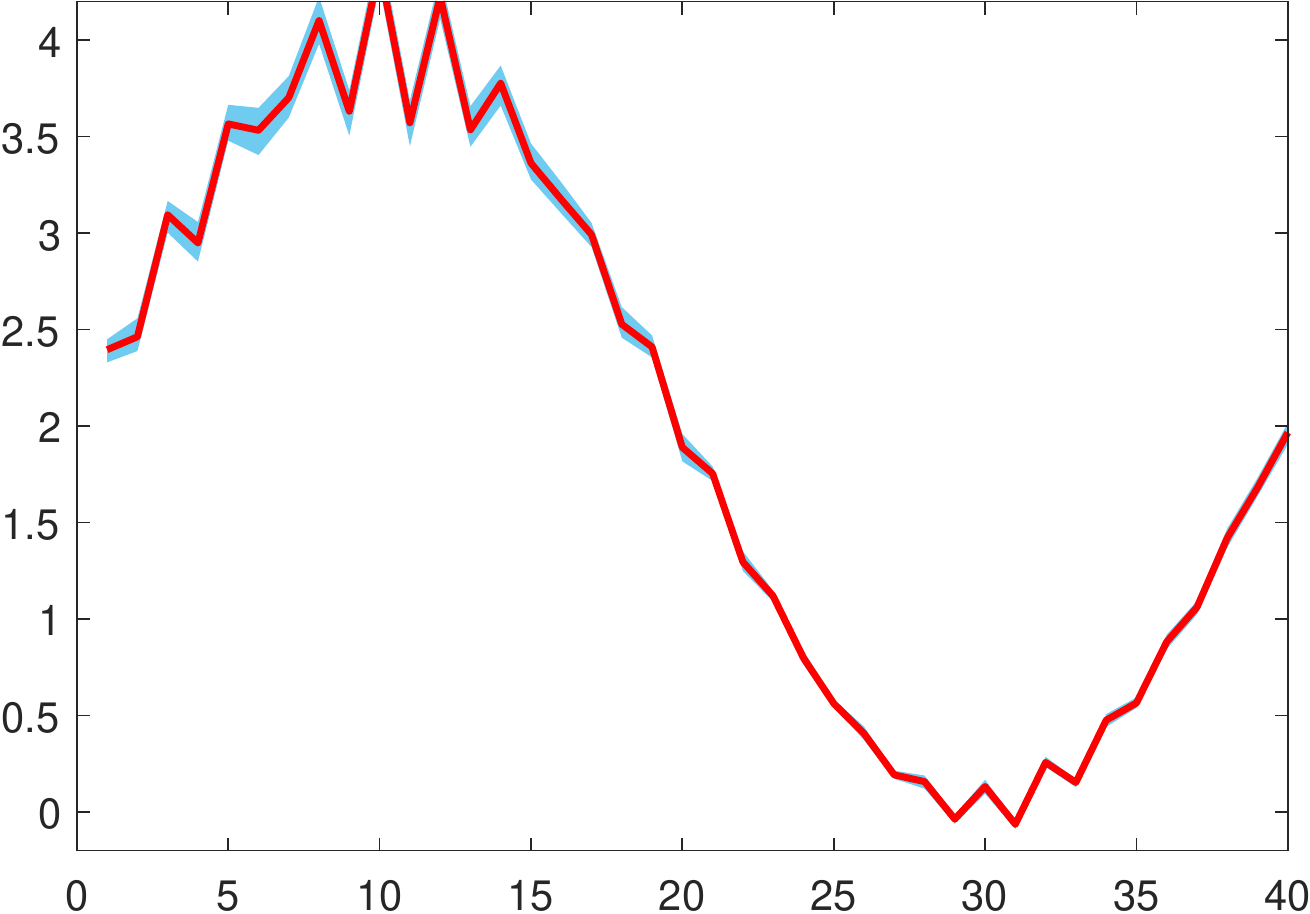}
		\caption{IEKF.}
		% 			\label{fig:i1}
	\end{subfigure}%
	%	\hfill
	\begin{subfigure}{.32\textwidth}
		\centering
		\includegraphics[width=5cm]{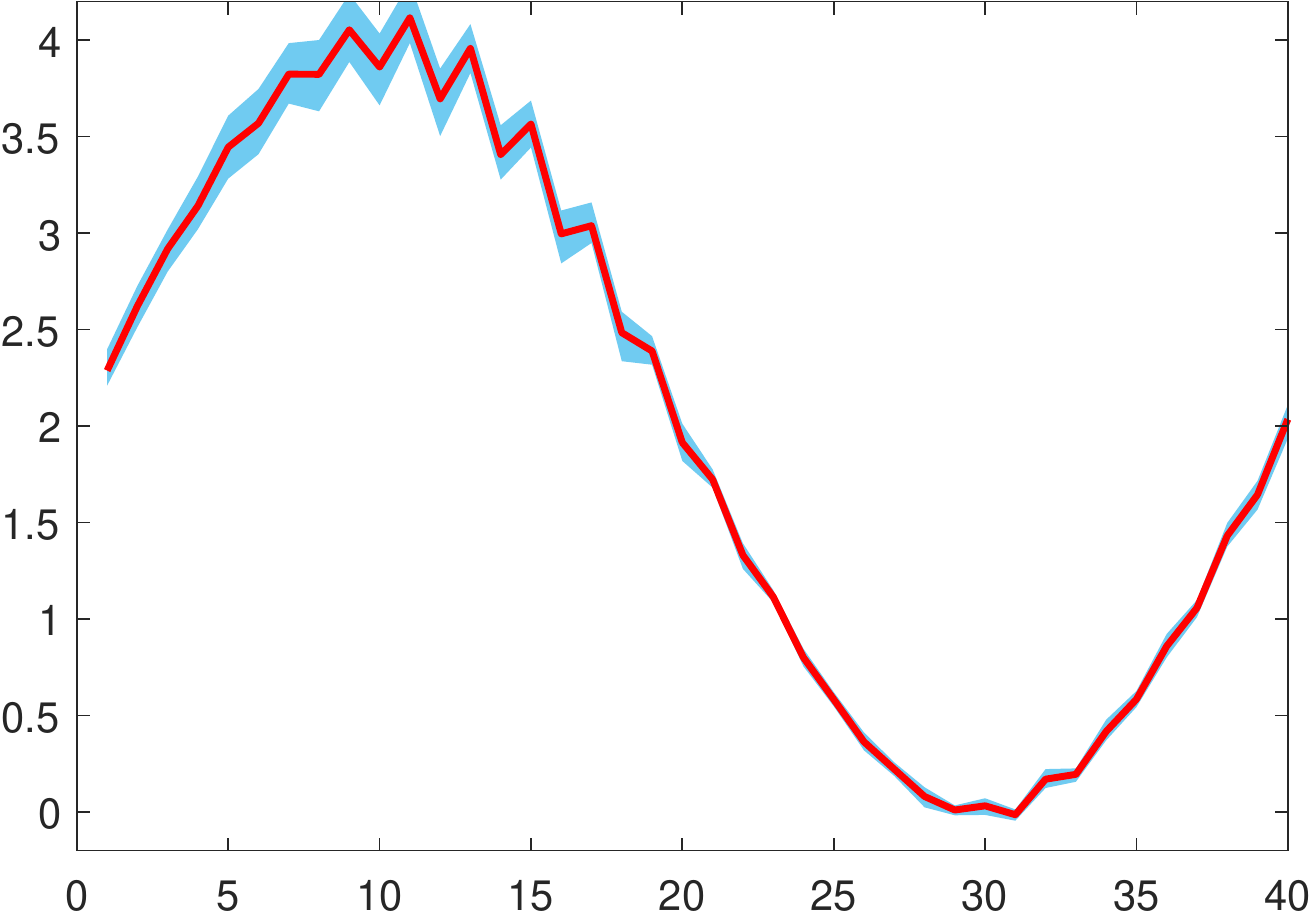}
		\caption{EKI-SL.}
		% 			\label{fig:i1}
	\end{subfigure}%
	%	\hfill
	\begin{subfigure}{.32\textwidth}
		\centering
		\includegraphics[width=5cm]{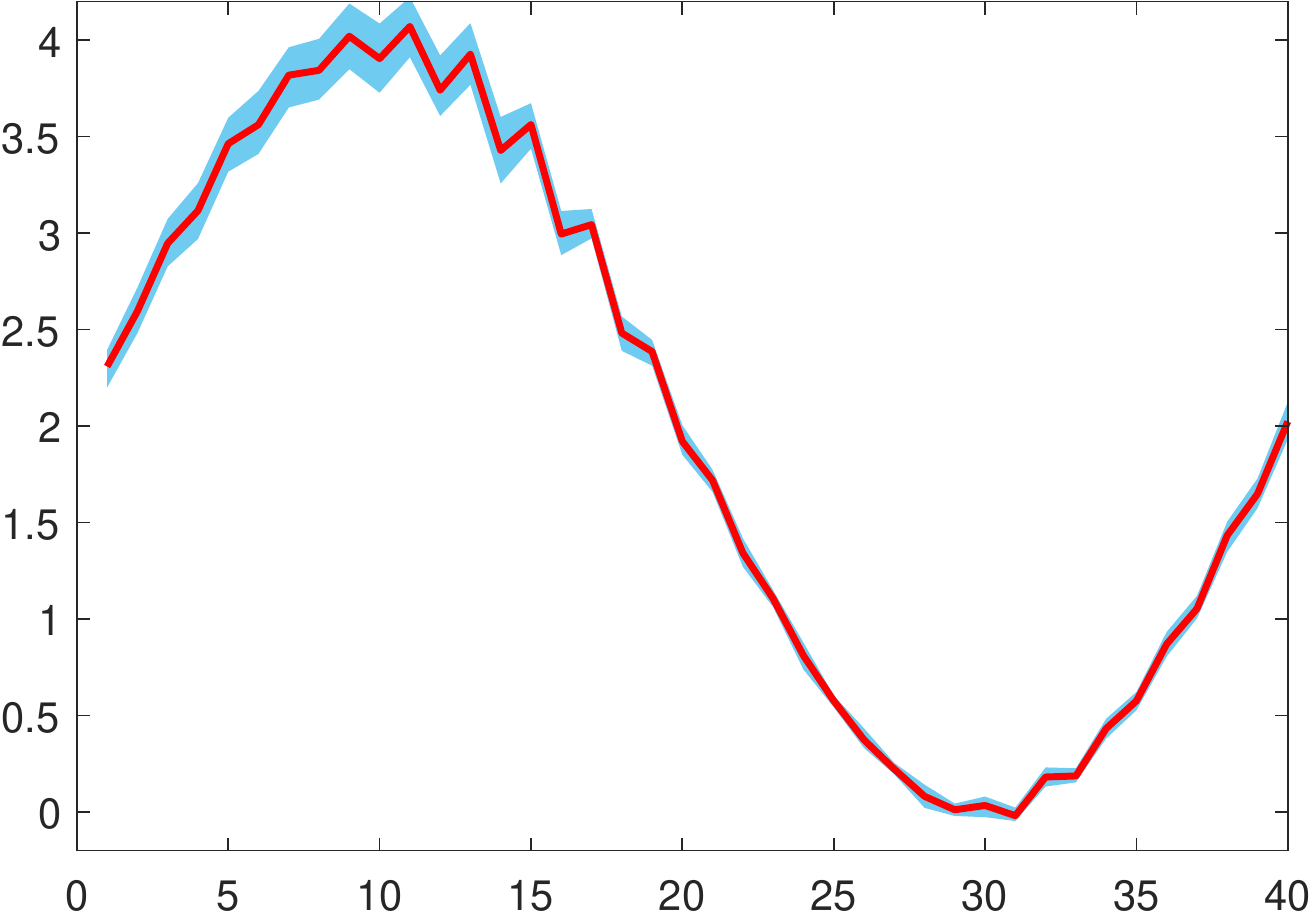}
		\caption{IEKF-SL.}
		% 			\label{fig:i1}
	\end{subfigure}%
	\caption{Ensemble mean (red) at the final iteration, with 10, 90-quantiles (blue).}
	 		\label{fig:l96_ensem}
\end{figure}
%\FloatBarrier

\begin{figure}[htbp]
	\centering
	\begin{subfigure}{.48\textwidth}
		\centering
		\includegraphics[width=7cm]{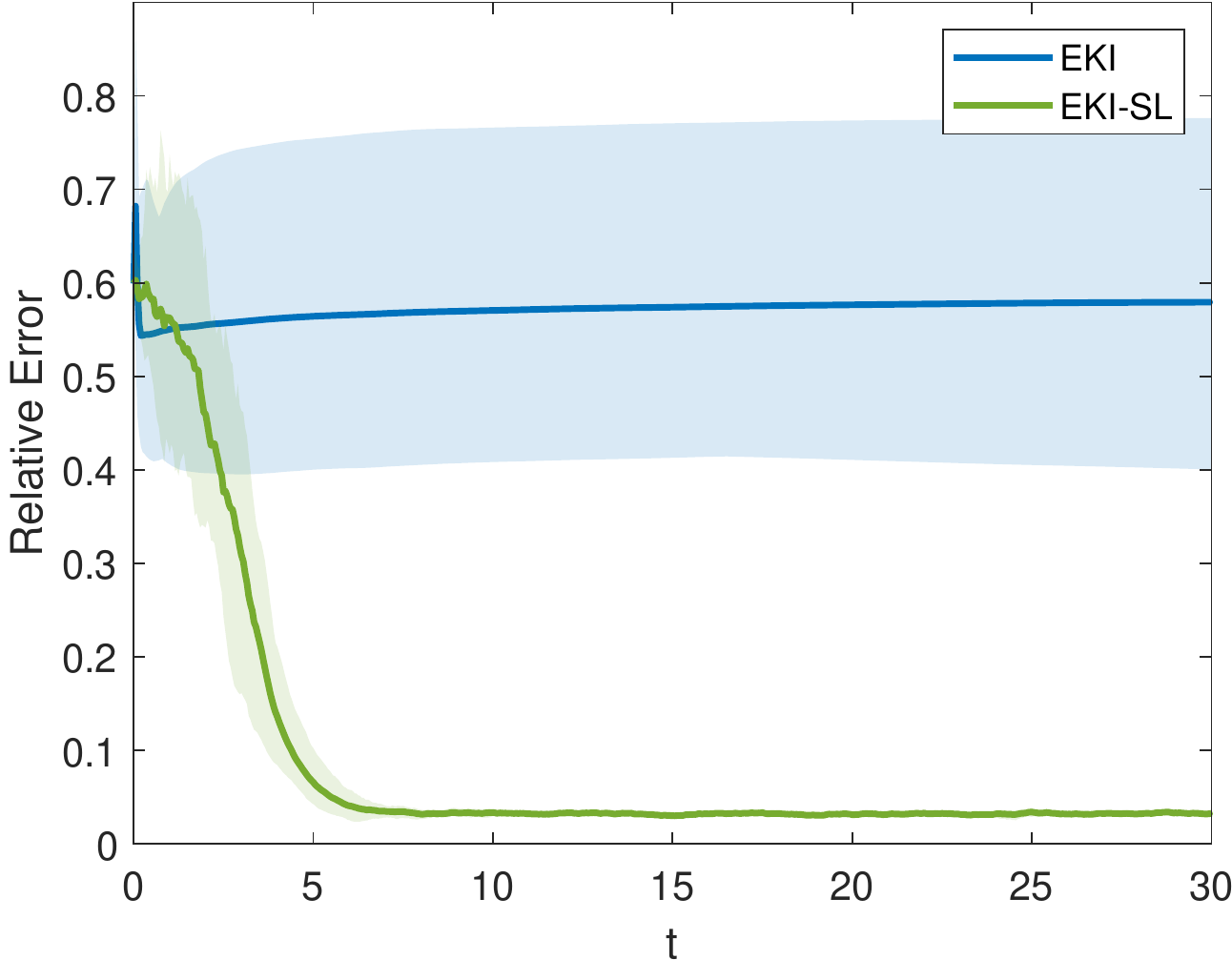}
		%		\caption{Truth $u^\dagger$}
		% 			\label{fig:i1}
	\end{subfigure}%
	%	\hfill
	\begin{subfigure}{.48\textwidth}
		\centering
		\includegraphics[width=7cm]{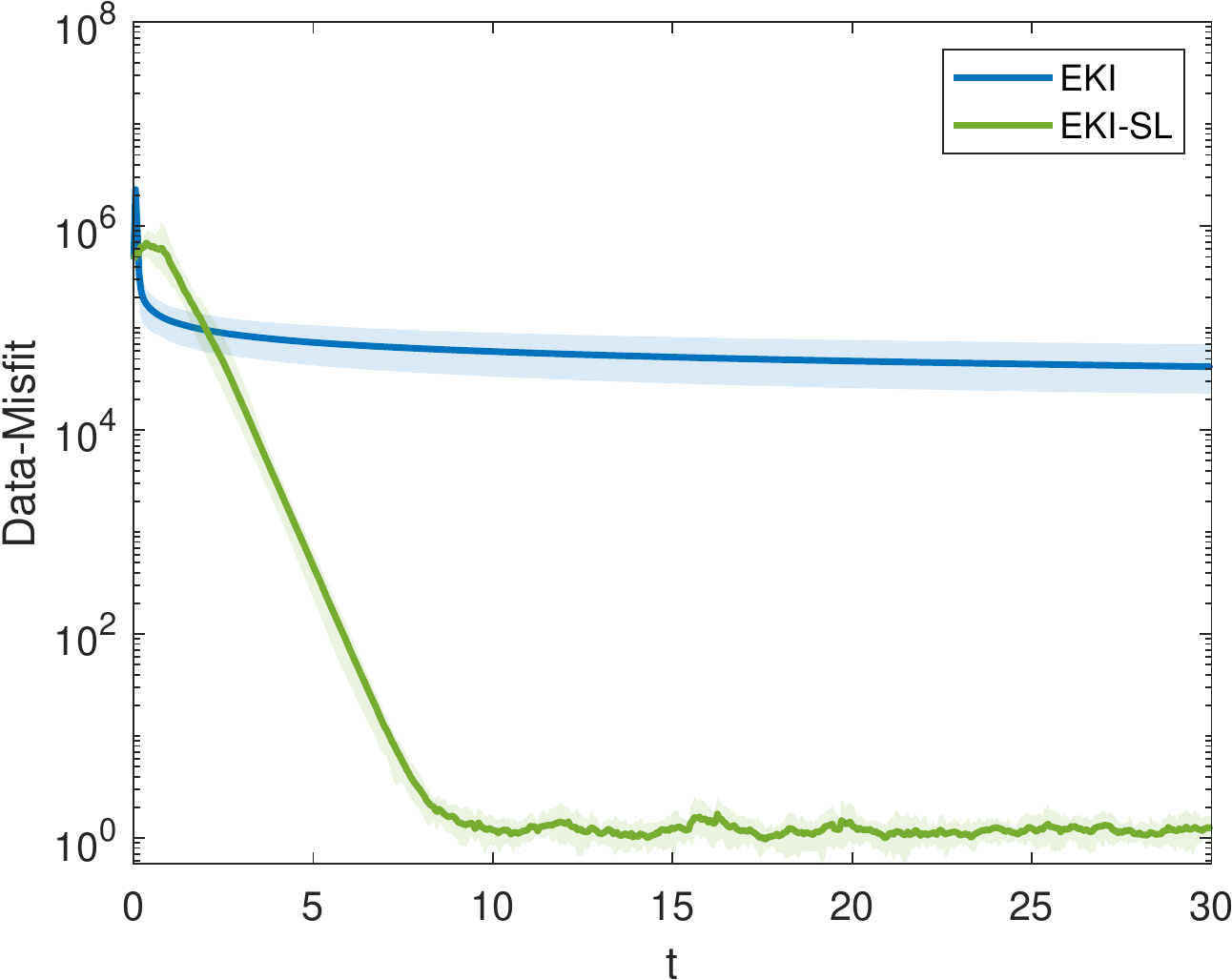}
		%		\caption{EKI}
		% 			\label{fig:i1}
	\end{subfigure}%
	\caption{EKI \& EKI-SL: Relative errors and data misfit  w.r.t time $t$.}
	 		\label{fig:l96_dm}
\end{figure}

\begin{figure}[ht]
	\centering
	\begin{subfigure}{.48\textwidth}
		\centering
		\includegraphics[width=7cm]{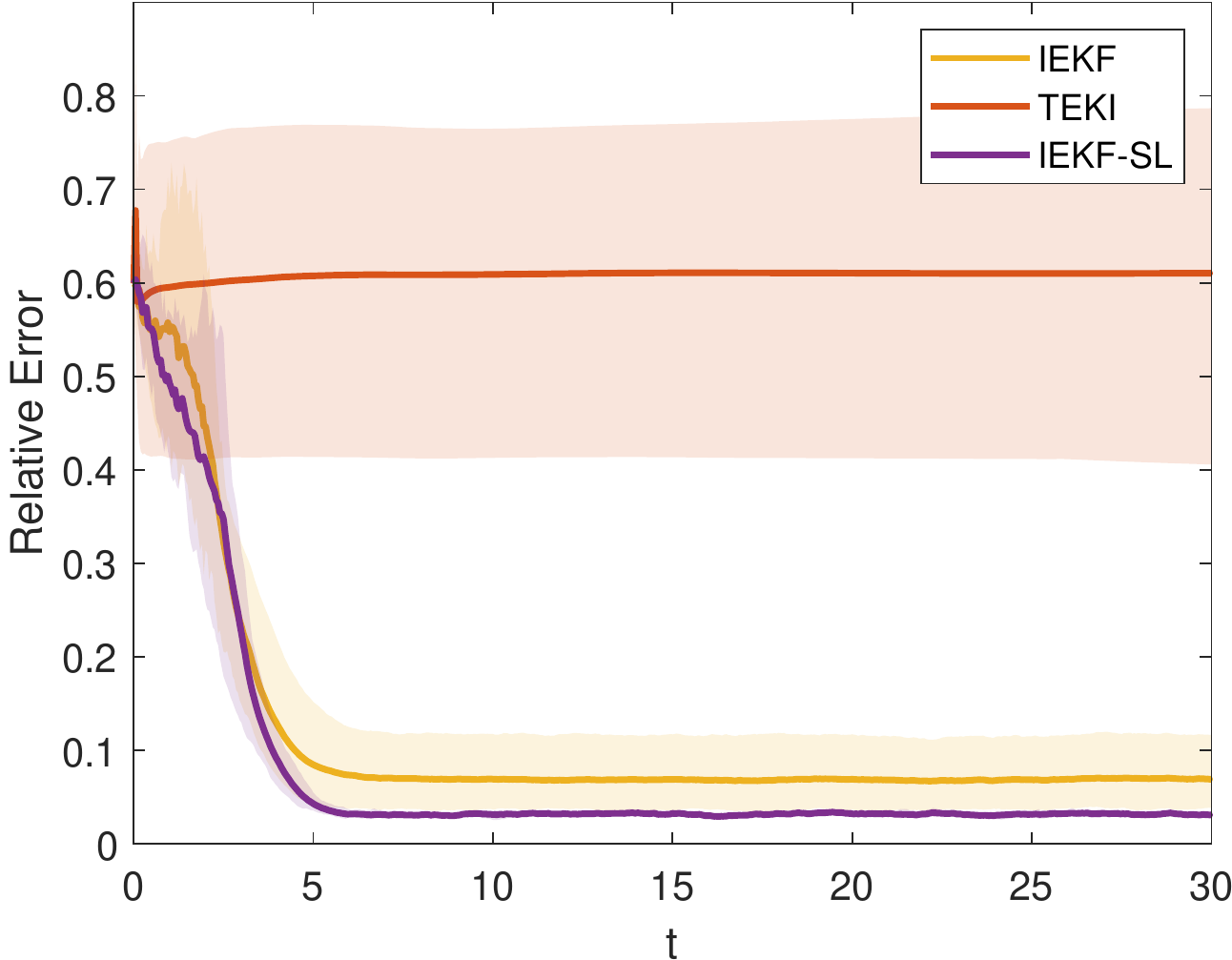}
		%		\caption{TEKI}
		% 			\label{fig:i1}
	\end{subfigure}%
	%	\hfill
	\begin{subfigure}{.48\textwidth}
		\centering
		\includegraphics[width=7cm]{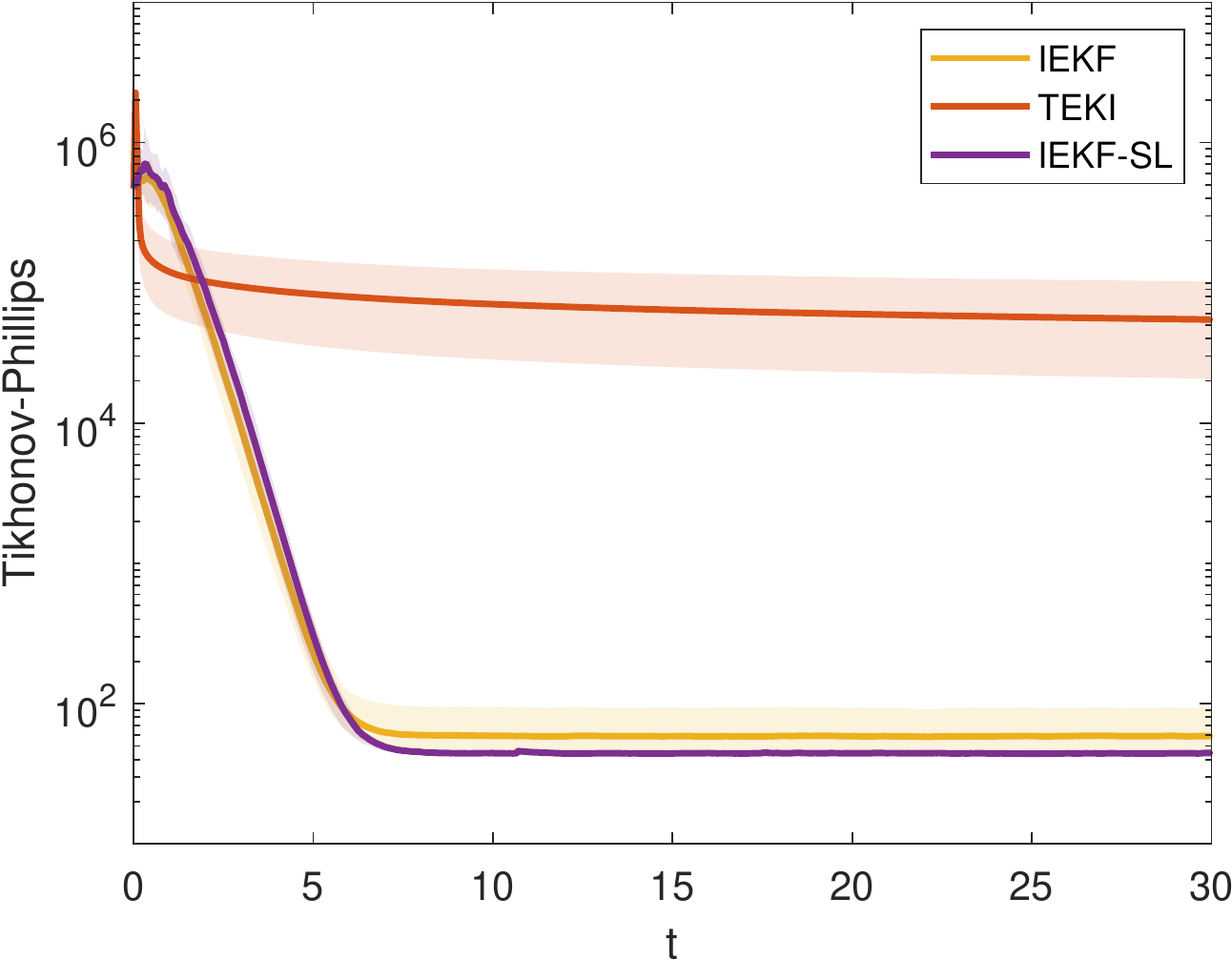}
		%		\caption{IEKF-SL}
		% 			\label{fig:i1}
	\end{subfigure}%
	\caption{IEKF, TEKI \& IEKF-SL: Relative errors and Tikhonov-Phillips objective w.r.t time $t$.}
	 		\label{fig:l96_tp}
\end{figure}
\FloatBarrier

\subsection{High-Dimensional Nonlinear Regression}\label{ssec:nonlinearregression}
In this subsection we consider a nonlinear regression problem with a highly oscillatory forward map introduced in \cite{EFL18}, where the authors investigate the use of iterative ensemble Kalman methods to train neural networks without back propagation. 
\subsubsection{Problem Setup}
 We consider a nonlinear regression problem $$y = h(u) + \eta, \quad \eta\sim \Nc(0, \gamma^2 I_k),$$ where $u \in \R^d, y \in \R^k$, and $h$ is defined by
\begin{equation}
h(u):= A u + \sin(c B u),
\end{equation}
where $A, B \in \R^{k\times d}$ are random matrices with independent $\Nc(0, 1)$ entries. We set $d = 200$, $k = 150$ and $c = 20$. We want to recover $u$ from $y$. We assume that the unknown $u$ has a Gaussian prior  $u \sim \Nc(0, 4I_d)$. The true parameter $u^\dagger$ is set to be $2\cdot\bf{1}$, where $\bm{1}$ is the all-one vector. Observation data is generated as $y = h(u^\dagger) + \eta$.

By definition, $h$ is highly oscillatory, and we may expect that the loss function, either Tikhonov-Phillips or data misfit objective, will have many local minima. Figure \ref{fig:regress_landscape} visualizes the Tihonov-Phillips objective function $\Jtp(u)$ with respect to two randomly choosen coordinates while other coordinates are fixed to value of 2. The data misfit objective exhibits a similar behavior.

\begin{figure}[htbp]
	\centering
	%\begin{subfigure}{.48\textwidth}
	\includegraphics[width=7cm]{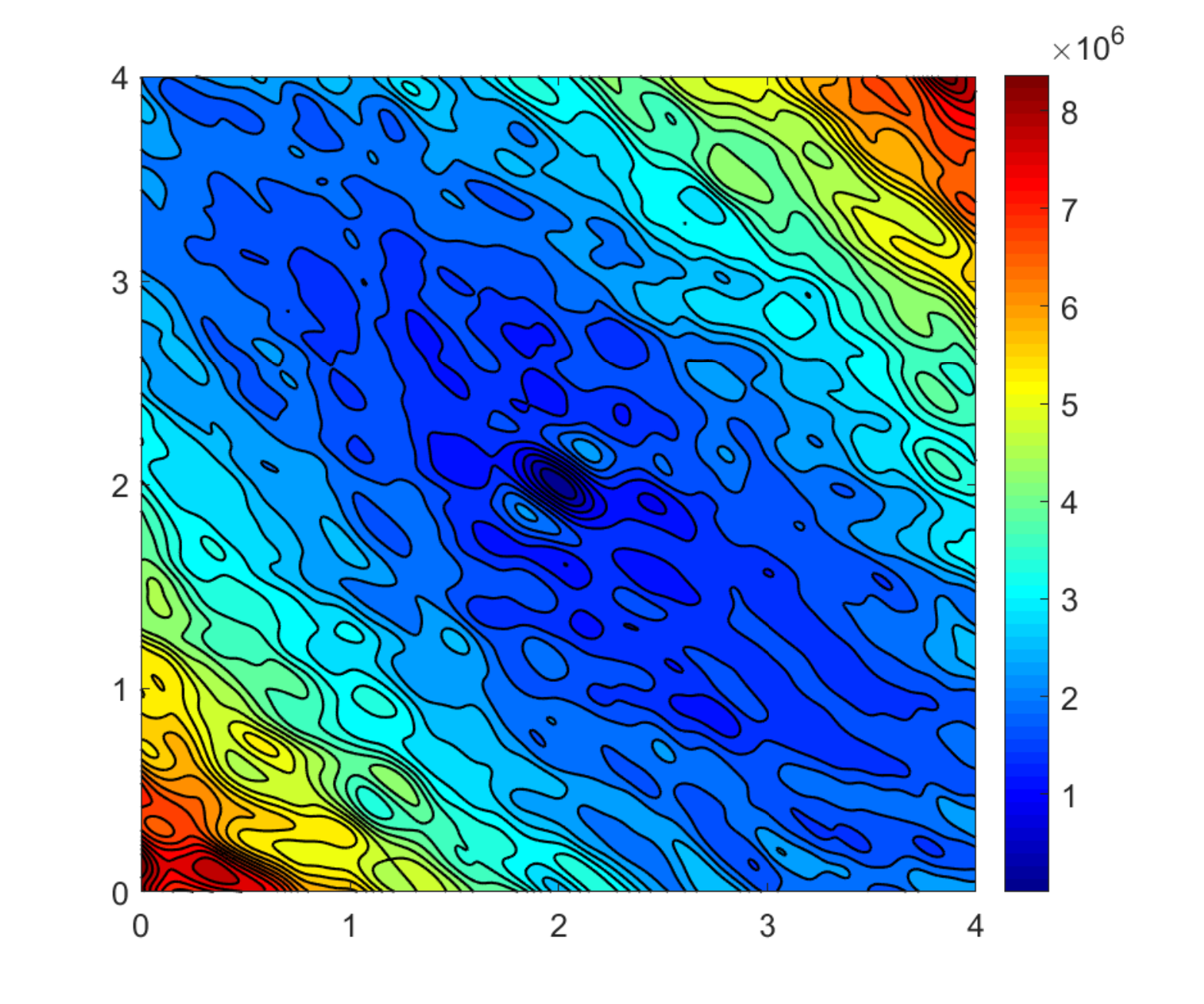}
	%\end{subfigure}
	\caption{Tikhonov-Phillips objective function with respect to two randomly chosen coordinates.}
	\label{fig:regress_landscape}
\end{figure}
%\FloatBarrier

\subsubsection{Implementation Details and Numerical Results}
We set the ensemble size to be $N = 50$. The initial ensemble $\{u_0^{(n)}\}_{n=1}^N$ is drawn independently from the prior. The length-step $\alpha$ is fixed to be 0.05 for all methods. We set $\gamma = 0.01$ for the observation noise. We run each algorithm up to time $T = 30$, which corresponds to 600 iterations.

We notice that this is a difficult problem, due to its high dimensionality and nonlinearity. All methods except for IEKF-SL are not capable of reconstructing the truth. In particular, from Figures \ref{fig:regress_ensem}, \ref{fig:regress_dm} and \ref{fig:regress_tp}, both EKI and TEKI do a poor job with relative error larger than 1, while IEKF and EKI-SL have slightly better performance. It is worth noticing from Figure \ref{fig:regress_tp} that IEKF-SL has a larger Tikhonov-Phillips objective function, with a much lower relative error.  This may suggest that other types of regularization objectives can be used, other that the Tikhonov-Phillips objective, to solve this problem. 
\begin{figure}[htbp]
	\centering
	\begin{subfigure}{.32\textwidth}
		\centering
		\includegraphics[width=5cm]{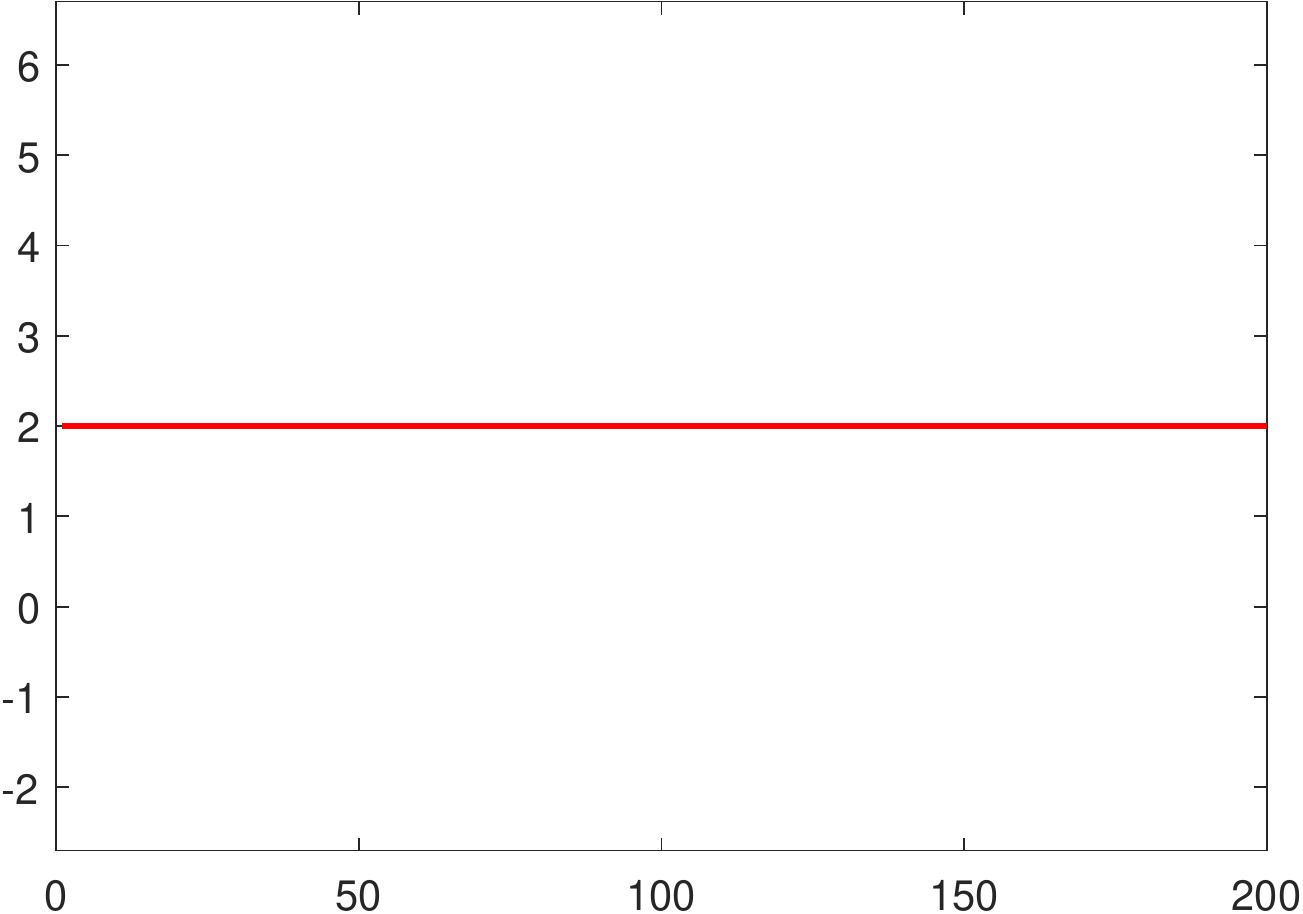}
		\caption{Truth $u^\dagger$.}
		% 			\label{fig:i1}
	\end{subfigure}%
	%	\hfill
	\begin{subfigure}{.32\textwidth}
		\centering
		\includegraphics[width=5cm]{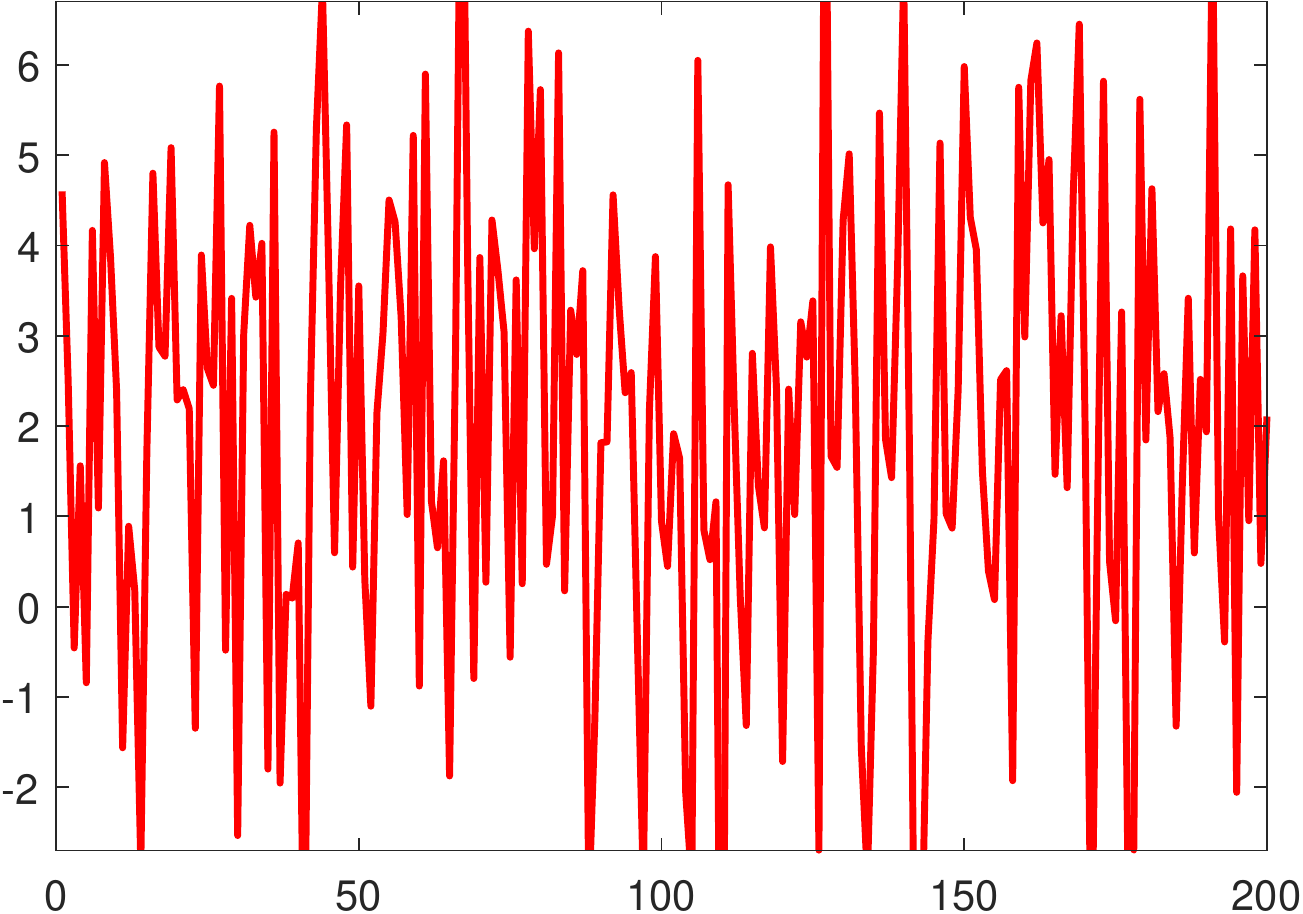}
		\caption{EKI.}
		% 			\label{fig:i1}
	\end{subfigure}%
	%	\hfill
	\begin{subfigure}{.32\textwidth}
		\centering
		\includegraphics[width=5cm]{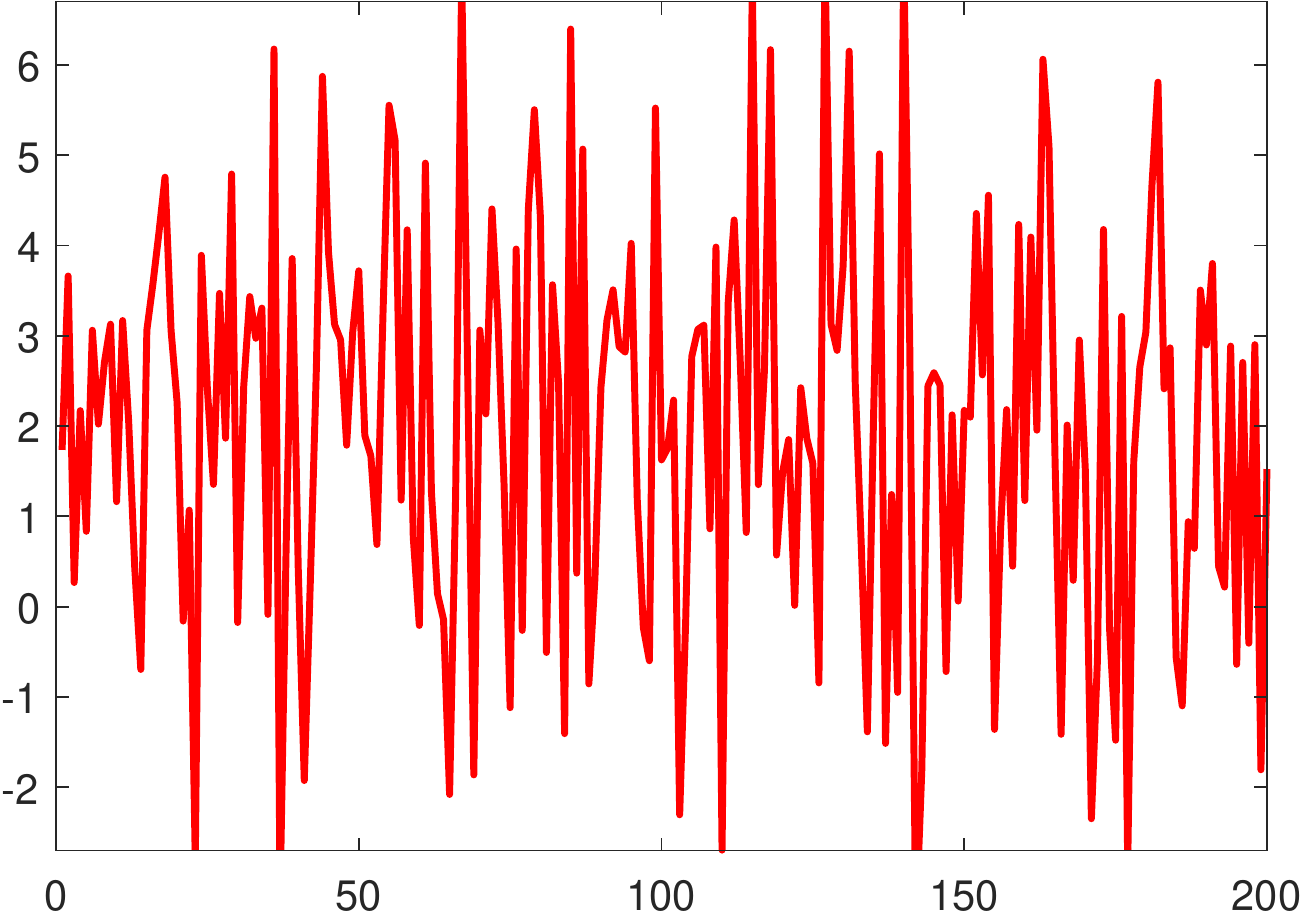}
		\caption{TEKI.}
		% 			\label{fig:i1}
	\end{subfigure}%
	\vskip\baselineskip
	\begin{subfigure}{.32\textwidth}
		\centering
		\includegraphics[width=5cm]{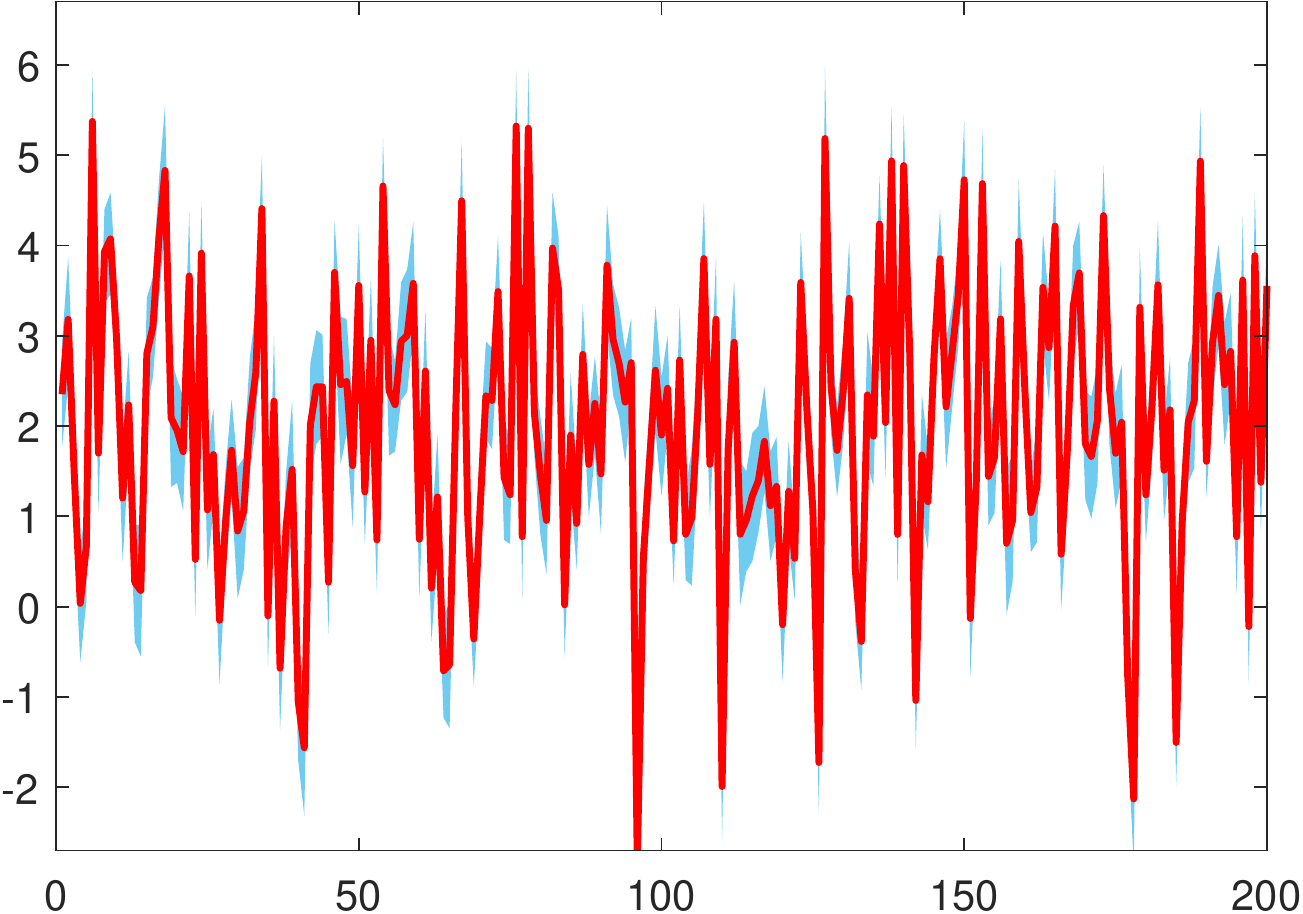}
		\caption{IEKF.}
		% 			\label{fig:i1}
	\end{subfigure}%
	%	\hfill
	\begin{subfigure}{.32\textwidth}
		\centering
		\includegraphics[width=5cm]{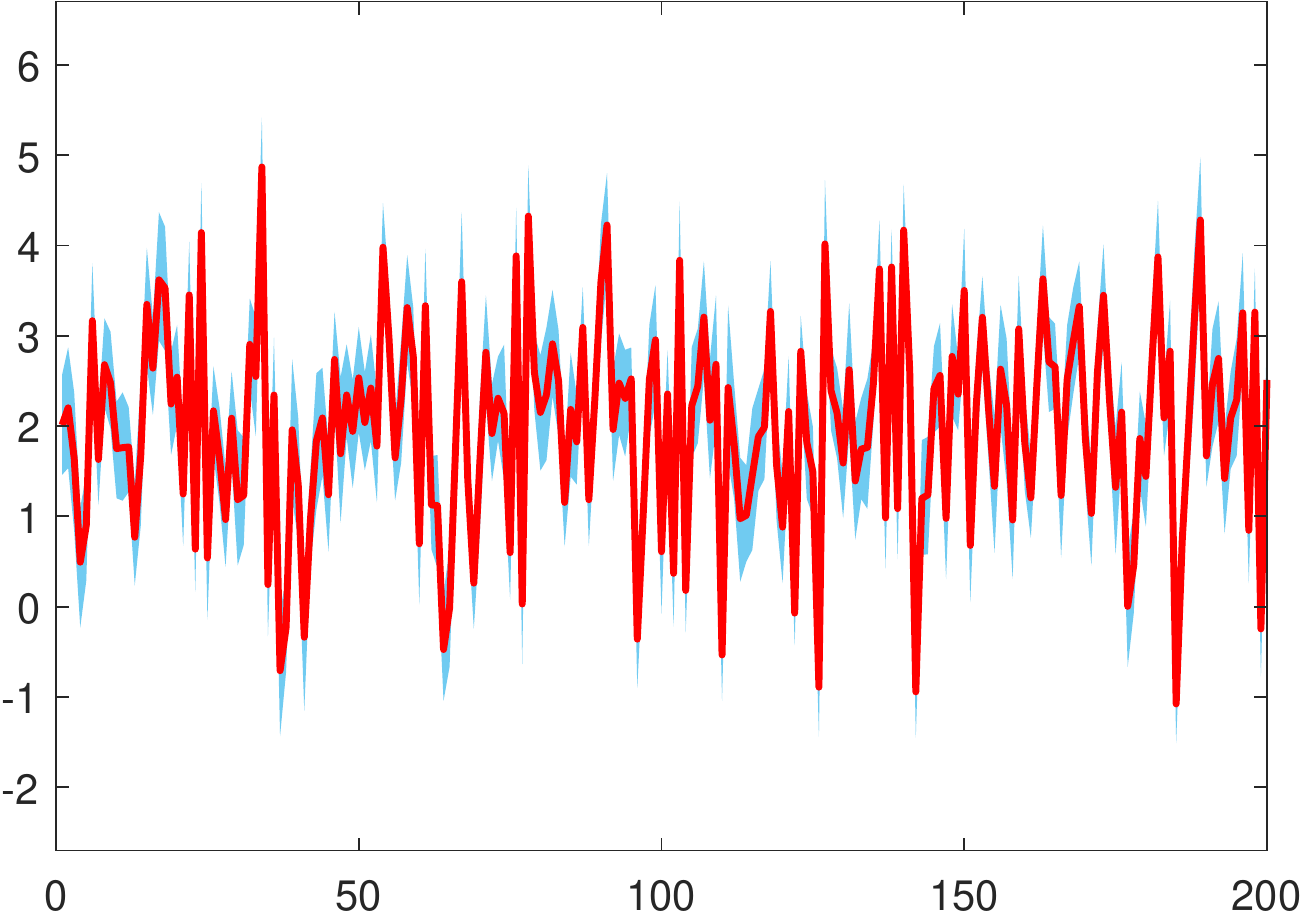}
		\caption{EKI-SL.}
		% 			\label{fig:i1}
	\end{subfigure}%
	%	\hfill
	\begin{subfigure}{.32\textwidth}
		\centering
		\includegraphics[width=5cm]{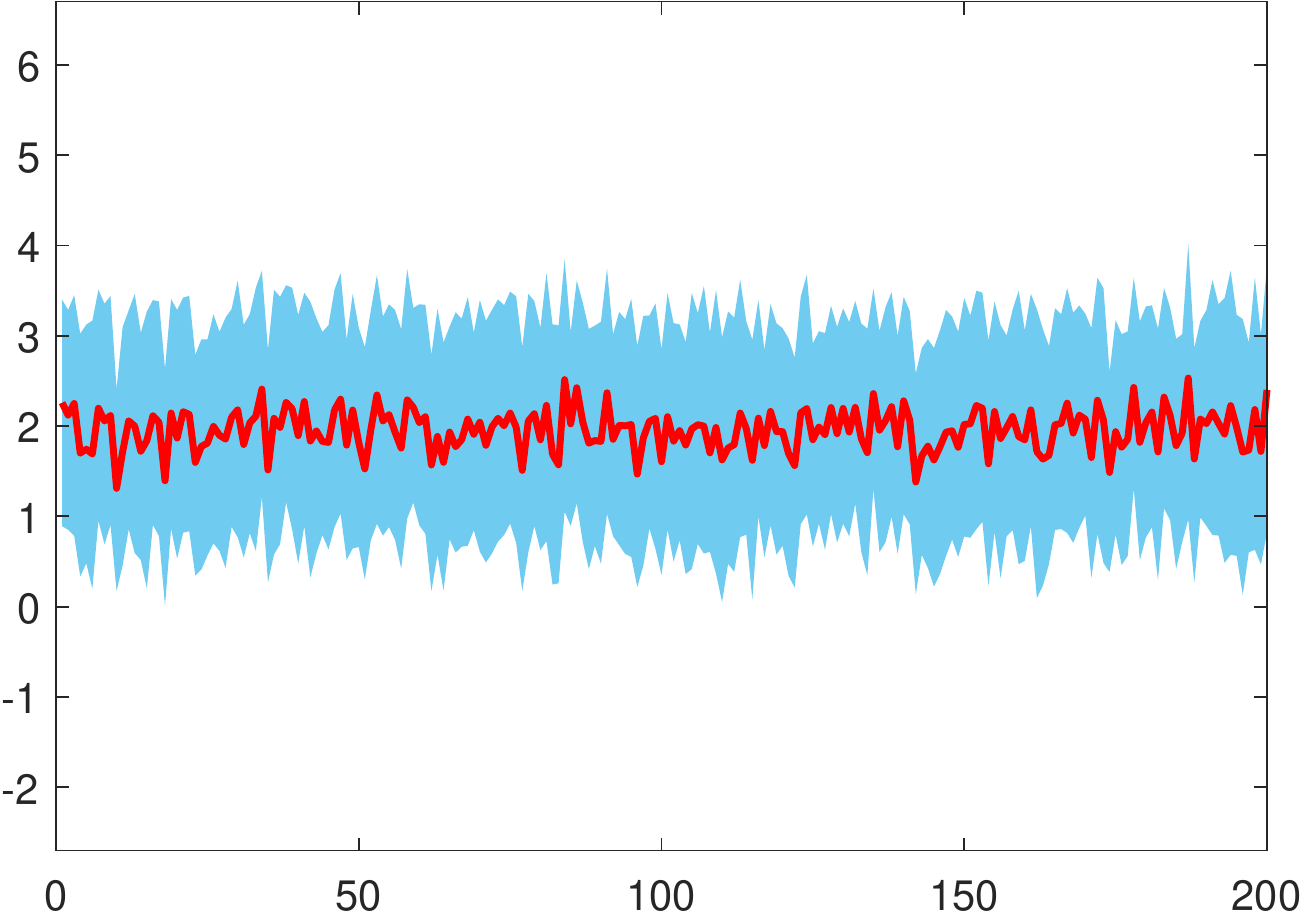}
		\caption{IEKF-SL.}
		% 			\label{fig:i1}
	\end{subfigure}%
	\caption{Ensemble mean (red) at the final iteration, with 10, 90-quantiles (blue).}
	 		\label{fig:regress_ensem}
\end{figure}
%\FloatBarrier

\begin{figure}[htbp]
	\centering
	\begin{subfigure}{.48\textwidth}
		\centering
		\includegraphics[width=7cm]{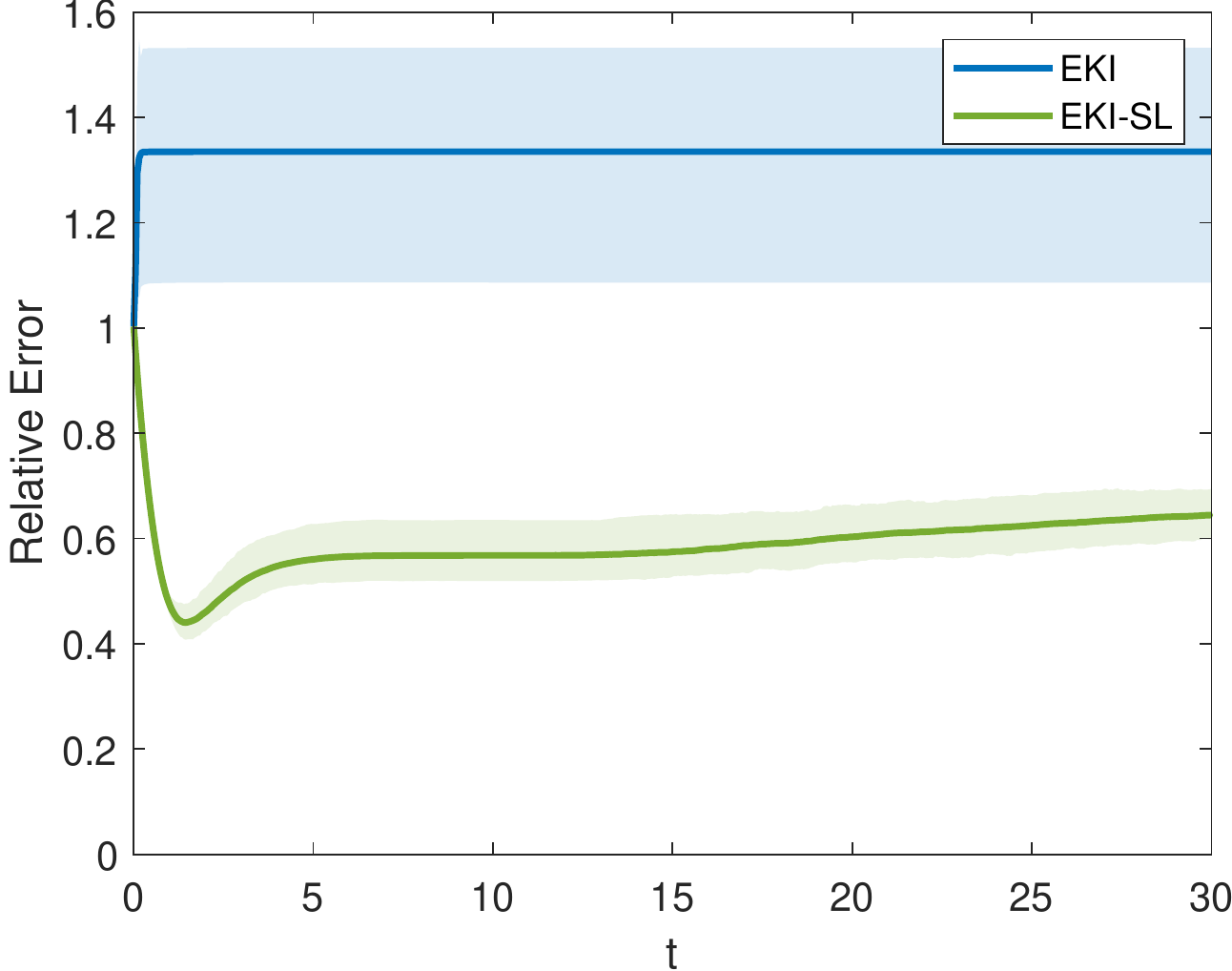}
		%		\caption{Truth $u^\dagger$}
		% 			\label{fig:i1}
	\end{subfigure}%
	%	\hfill
	\begin{subfigure}{.48\textwidth}
		\centering
		\includegraphics[width=7cm]{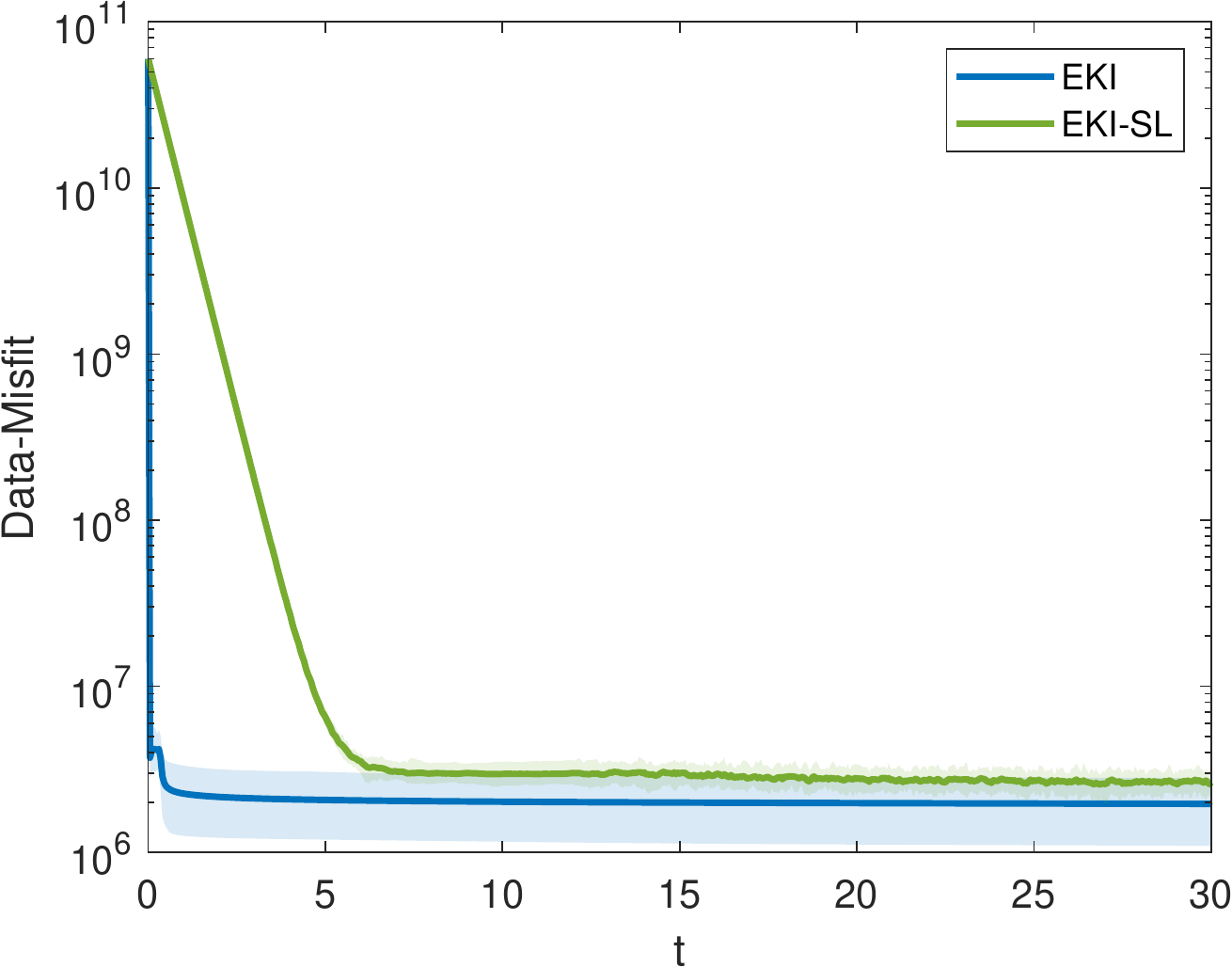}
		%		\caption{EKI}
		% 			\label{fig:i1}
	\end{subfigure}%
	\caption{EKI \& EKI-SL: Relative errors and data misfit  w.r.t time $t$.}
	 		\label{fig:regress_dm}
\end{figure}

\begin{figure}[ht]
	\centering
	\begin{subfigure}{.48\textwidth}
		\centering
		\includegraphics[width=7cm]{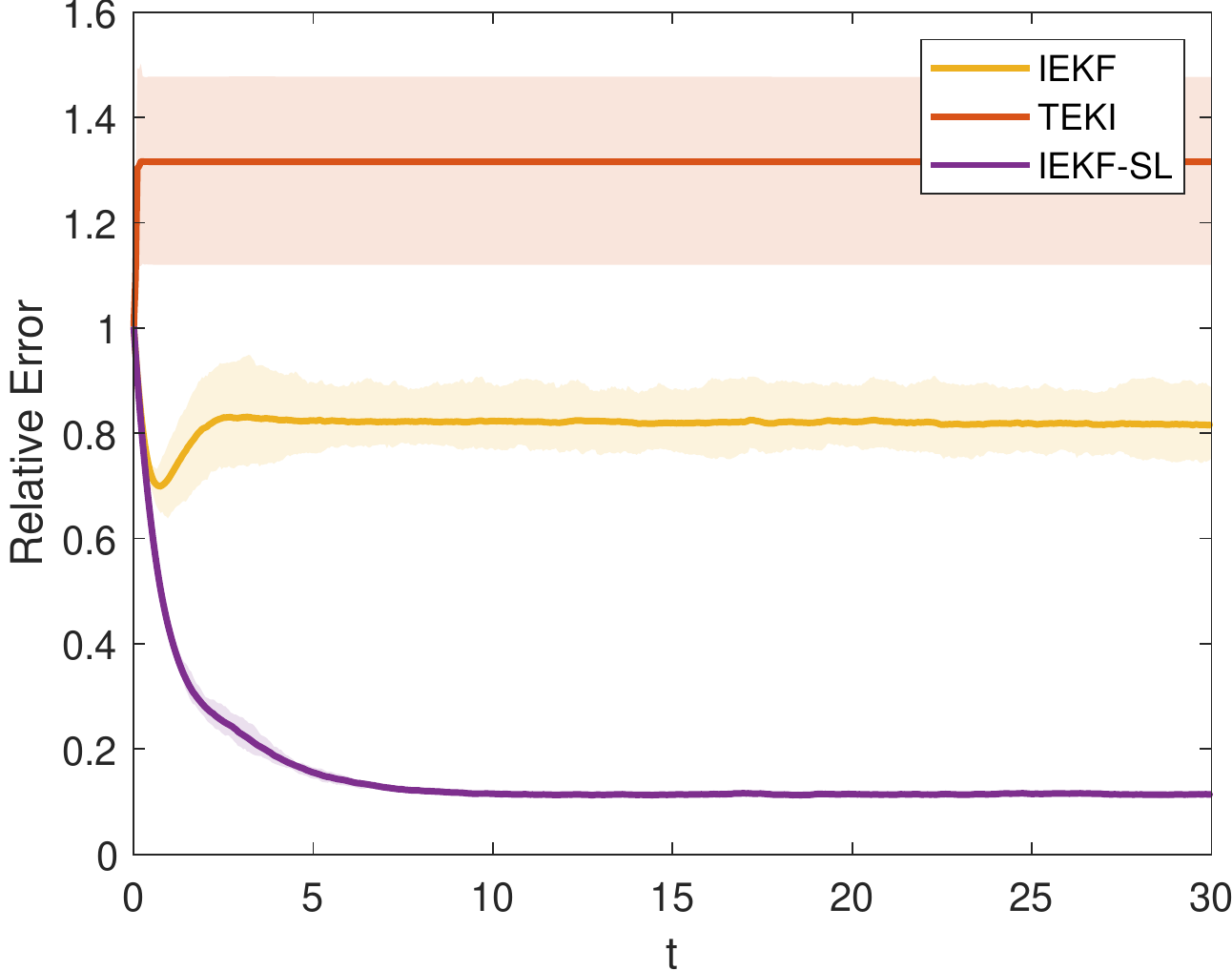}
		%		\caption{TEKI}
		% 			\label{fig:i1}
	\end{subfigure}%
	%	\hfill
	\begin{subfigure}{.48\textwidth}
		\centering
		\includegraphics[width=7cm]{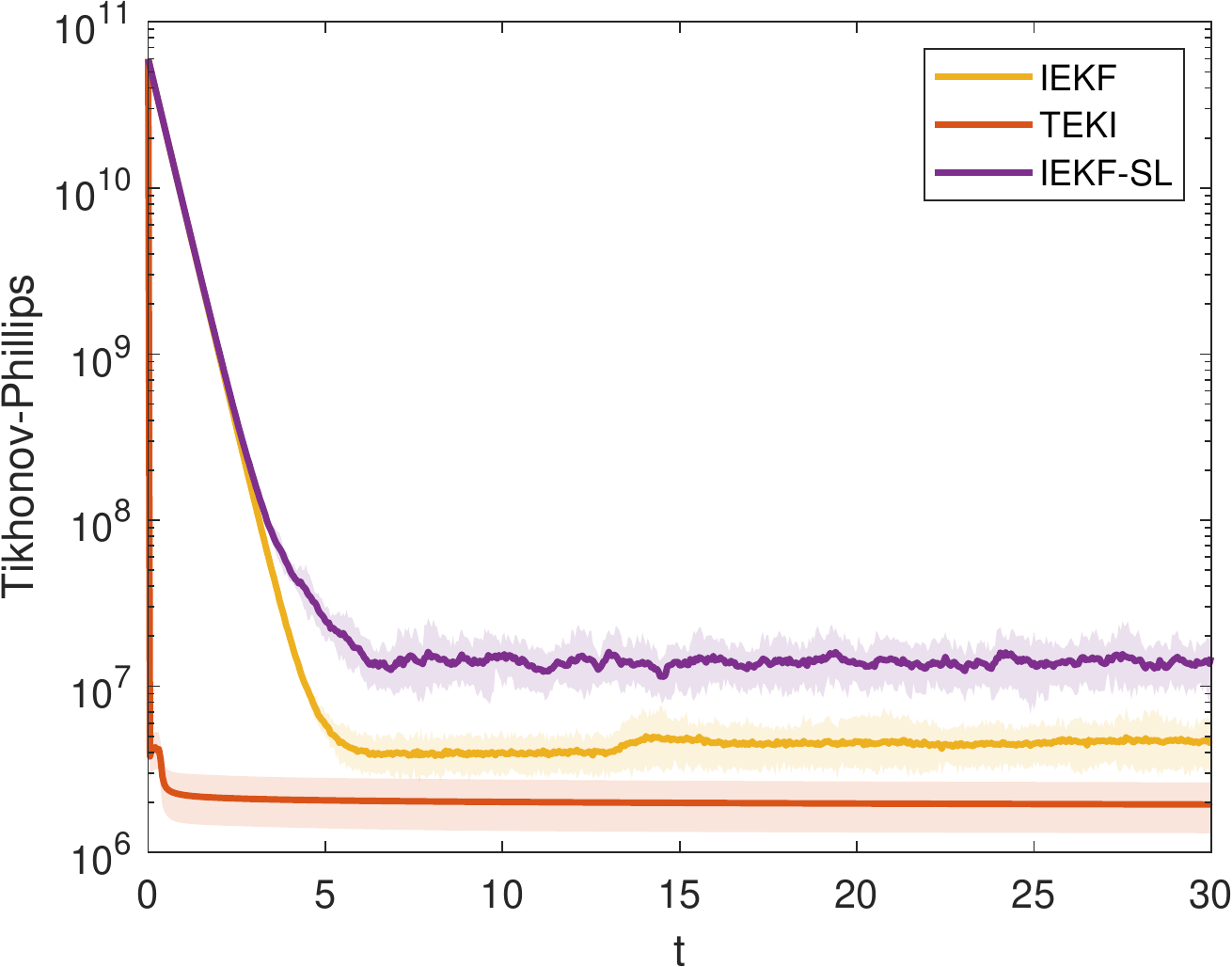}
		%		\caption{IEKF-SL}
		% 			\label{fig:i1}
	\end{subfigure}%
	\caption{IEKF, TEKI \& IEKF-SL: Relative errors and Tikhonov-Phillips objective w.r.t time $t$.}
	 		\label{fig:regress_tp}
\end{figure}

\section{Conclusions and Open Directions}\label{sec:conclusions}
In this paper we have provided a unified perspective of iterative ensemble Kalman methods and introduced some new variants. We hope that our work will stimulate further research in this active area, and we conclude with a list of some open directions.

\begin{itemize}
\item Continuum limits have been formally derived in our work.  The rigorous derivation and analysis of SDE continuum limits, possibly in nonlinear settings, deserves further research. 
\item We have advocated the analysis of continuum limits for the understanding and design of iterative ensemble Kalman methods, but it would also be desirable to develop a framework for the analysis of discrete, implementable algorithms, and to further understand the potential benefits of various discretizations of a given continuum SDE system. 
\item From a theoretical viewpoint, it would be desirable to further analyze the convergence and stability of iterative ensemble methods with small or moderate ensemble size. While mean-field limits can be revealing, in practice the ensemble size is often not sufficiently large to justify the mean-field assumption. It would also be important to  further analyze these questions in mildly nonlinear settings. 
\item An important methodological question, still largely unresolved, is the development of adaptive and easily implementable line search schemes and stopping criteria for ensemble-based optimization schemes. An important work in this direction is \cite{IY20}. 
\item Another avenue for future methodological research is the development of iterative ensemble Kalman methods that are sparse-promoting,  considering alternative regularizations beyond the least-squares objectives considered in our paper \cite{KS19,L20,S20}. 
\item Ensemble methods are cheap in comparison to derivative-based optimization methods and Markov chain Monte Carlo sampling algorithms. It is thus natural to use ensemble methods to build ensemble preconditioners or surrogate models to be used within more expensive but accurate computational approaches. 
\item Finally, a broad area for further work is the application of iterative ensemble Kalman filters to new problems in science and engineering.
\end{itemize}

\section*{Acknowledgments}
NKC is supported by KAUST baseline funding.
DSA is thankful for the support of NSF and NGA through the grant DMS-2027056. 
The work of DSA was also partially supported by the NSF Grant DMS-1912818/1912802.

\bibliographystyle{plain} %alpha
% \bibliography{references}

\end{document}